\documentclass[12pt,a4paper]{article}
\usepackage[utf8]{inputenc}
\usepackage[T1]{fontenc}
\usepackage[english]{babel}
\usepackage{indentfirst}
\usepackage{enumerate}
\usepackage{amstext}
\usepackage{authblk}
\usepackage[dvipsnames]{xcolor}
\usepackage{amsfonts}
\usepackage{latexsym}
\usepackage{tikz}
\usepackage{textcomp}
\usepackage{amssymb}
\usepackage{amsmath}
\usepackage{amscd}
\usepackage[mathcal]{euscript}
\usepackage{geometry}
\usepackage[all]{xy}
\usepackage{anysize}
\usepackage{tabularx}
\usepackage{makeidx}
\usepackage{amsthm}	
\usepackage{mathrsfs}
\usepackage{hyperref}
\usepackage{afterpage}
\usepackage{setspace}
\usepackage{epsf}
\usepackage{graphicx}
\usepackage{epsfig}
\usepackage{comment}

\DeclareMathOperator{\Supp}{Supp}
\DeclareMathOperator{\Int}{Int}
\DeclareMathOperator{\Orb}{Orb}
\DeclareMathOperator{\Fix}{Fix}
\DeclareMathOperator{\Iso}{Iso}
\DeclareMathOperator{\Span}{span}

\newcommand{\usetr}[2]{\text{$\uparrow_{\scriptscriptstyle{#2}}\hspace{-0.1cm}{#1}$}} %relative upper set

\newcommand{\xia}{\xi^{\alpha}}

\newcommand{\Z}{\mathbb{Z}}

\newcommand{\F}{\mathbb{F}}

\newcommand{\fr}{\Phi}

\newcommand{\p}{\cdot_{\theta}}
\newcommand{\nucleo}{\operatorname{ker}}

\newcommand{\id}{\operatorname{id}}

\theoremstyle{plain}
\newtheorem{theorem}{Theorem}[section]
\newtheorem{lema}[theorem]{Lemma}
\newtheorem{corolario}[theorem]{Corollary}
\newtheorem{proposition}[theorem]{Proposition}

\theoremstyle{definition}
\newtheorem{definition}[theorem]{Definition}
\newtheorem{example}[theorem]{Example}

\theoremstyle{remark}
\newtheorem{obs}[theorem]{Remark}

\theoremstyle{plain}

\date{}

\title{The ideal structure of partial skew groupoid rings with applications to topological dynamics and ultragraph algebras}

\author[D. Bagio \and D. Gon\c{c}alves \and P. S. E. Moreira \and J. \"{O}inert]
{D. Bagio
\thanks{Universidade Federal de Santa Catarina,
{\color{blue}d.bagio@ufsc.br}},
D. Gon\c{c}alves
\thanks{Universidade Federal de Santa Catarina, 
{\color{blue}daniel.goncalves@ufsc.br}},
P. S. E. Moreira
\thanks{{\bf Corresponding author:} Departamento de Matem\'atica, Universidade Federal de Santa Catarina, 88040-900	Florian\'opolis, Brasil. {\color{blue}savanamatematica@gmail.com}} \,\ 
and J. \"{O}inert
\thanks{Blekinge Institute of Technology,
{\color{blue}johan.oinert@bth.se}
\newline
The third named author was financed in part by the
Coordena\c{c}\~{a}o de Aperfei\c{c}oamento de Pessoal de N\'{i}vel Superior - Brasil (CAPES) - Finance Code 001.
}}

\begin{document}

\maketitle
\begin{abstract}
Given a partial action $\alpha$ of a groupoid $G$ on a ring $R$, we study the associated partial skew groupoid ring $R \rtimes_{\alpha} G$,
which carries a natural $G$-grading.
We show that there is a one-to-one correspondence between the $G$-invariant ideals of $R$ and the graded ideals of the $G$-graded ring $R \rtimes_{\alpha}G.$ We provide sufficient conditions for primeness, and necessary and sufficient conditions for simplicity of $R \rtimes_{\alpha}G.$ We show that every ideal of $R \rtimes_{\alpha}G$ is graded if, and only if, $\alpha$ has the residual intersection property. Furthermore, if $\alpha$ is induced by a topological partial action $\theta$, then we prove that minimality of $\theta$ is equivalent to $G$-simplicity of $R$, topological transitivity of $\theta$ is equivalent to $G$-primeness of $R$, and topological freeness of $\theta$ on every closed invariant subset of the underlying topological space is equivalent to $\alpha$ having the residual intersection property. As an application, we characterize Condition (K) for an ultragraph in terms of topological properties of the associated partial action and in terms of algebraic properties of the associated ultragraph algebra.

\vspace{5mm}

\thanks{\noindent 2000 \emph{Mathematics Subject Classification.} 16S35, 16S99, 20L05, 16N60, 16W22, 16W50, 16S88.}
\newline
\textit{Keywords}: Partial skew groupoid ring, residual intersection property, graded ideal, topological freeness, topological transitivity, Condition (K), ultragraph algebra.

\vspace{5mm}

% 16S35(Twisted and skew group rings, crossed products), 20L05(Groupoids), 16W50(Graded rings and modules), (16W22) Actions of groups and semigroups; invariant theory, (16N60) Prime and semiprime rings, (16S88) Leavitt path algebras. 

\end{abstract}

\section{Introduction}

Partial actions and their associated structures arose in the study of $C^*$-algebras as a way to realize important classes of $C^*$-algebras as partial crossed products (see 
\cite{Exel2017}). In 2005, the study of algebraic partial actions of groups, and their associated partial skew group rings, was formalized (see \cite{Dokuchaev2005}). Since then, the applications and generalizations of the theory have increased steadily. We will mention a few examples. In \cite{Goncalves2014.1}, Leavitt path algebras were realized as partial skew group rings and in \cite{Dokuchaev2015} cohomology of partial actions was developed. Topological dynamics associated with partial skew group rings were described in \cite{Goncalves2014}, with (global) skew category algebras in \cite{Lundstrom2012, Oinert2013}, with $C^*$-dynamical systems in \cite{Thierry2014}, and with skew inverse semigroup rings in \cite{TBeuter2018}. A comprehensive overview of the evolution of the research area was given in \cite{Dokuchaev2019}.

Among the generalizations of partial skew group rings, we distinguish two: Partial skew inverse semigroup rings and partial skew groupoid rings. The former was introduced in \cite{Buss2012} and the latter in \cite{Bagio2012}. The key difference between the two constructions is that when defining a partial skew inverse semigroup ring a certain quotient is taken, to get rid of so-called \emph{redundancies}  (see~\cite{TBeuter2018}). This quotient, although necessary to realize known algebras as partial skew inverse semigroup rings, brings an extra layer of complexity to the study of partial skew inverse semigroup rings. In contrast, the definition of a partial skew groupoid ring does not require such a quotient (see Definition~\ref{P1-def:D3} below).
Note that every partial skew groupoid ring appearing in this paper can in fact be realized as a partial skew inverse semigroup ring (see Remark~\ref{P2-InverseSemigroups}).

Although the aforementioned quotient used in the definition of partial skew inverse semigroups is not present in the definition of partial skew groupoid rings, many important algebras, such as e.g. Leavitt path algebras (see \cite{Goncalves2016}), can still be realized as partial skew groupoid rings. Furthermore, there is a growing interest in the dynamics of groupoid partial actions. In the recent paper \cite{Flores2020}, a purely dynamical study of groupoid partial actions was conducted, without mentioning any interplay with the associated partial skew groupoid rings. In the case of discrete groupoids, the definition of a (topological) groupoid partial action in \cite{Flores2020} becomes a particular case of Definition~\ref{P1-def:D1}.

Every partial skew groupoid ring $R \rtimes_\alpha G$ carries a natural $G$-grading, and we will make use of this insight.
We mention a few examples of the importance of gradings in the study of rings: Hazrat's talented monoid conjecture for Leavitt path algebras states that the talented monoid of a Leavitt path algebra is a complete invariant for graded Morita equivalence of the algebras (see \cite{Hazrat2013,Hazrat2013.1,Cordeiro2021}).  The ideals, and graded ideals, of algebras associated with combinatorial structures, were studied in  \cite{Duyen2021} and \cite{Vas2022.1}. Graded ideals are also in close connection with the representation theory of the algebra (see \cite{Vas2022.2}).

Next, we describe our goals and give an outline of the paper.

In Section~\ref{P1-Sec:Prelim}, we recall key  definitions and 
the connection between groupoid fibred actions and global groupoid actions on sets, which was established in Proposition 4.1 of \cite{Gilbert2005}. We also ensure that the partial skew groupoid rings that we will be working with are indeed associative rings (see Remark~\ref{P1-rem:associatividade}).

In Section~\ref{P1-Sec:Algebraic}, we study the algebraic structure of partial skew groupoid rings and pursue two goals. 
The first goal is to reach a better understanding of the (graded) ideal structure of partial skew groupoid rings.
To that end, we show that there is a one-to-one correspondence between $G$-graded ideals of $R \rtimes_{\alpha}G$ and $G$-invariant ideals of $R$ (see Theorem~\ref{P1-thm:corresp1}).
We also
show that every ideal of $R \rtimes_\alpha G$ is $G$-graded if, and only if, $\alpha$ has the residual intersection property (see Theorem~\ref{P1-thm:bijpri}).
It is worth pointing out that the aforementioned result is new even for partial skew group rings.
Nevertheless, a similar result has been proved in the context of partial actions of groups on $C^*$-algebras (see \cite{Thierry2014}).

The second goal is to describe primeness and simplicity for partial skew groupoid rings.
We show that, if the action $\alpha$ has the intersection property, then the partial skew groupoid ring $R\rtimes_{\alpha}G$ is prime if, and only if, the ring $R$ is $G$-prime (see Theorem~\ref{P1-thm:AP2} (ii)).
Furthermore, we show that $R \rtimes_\alpha G$ is simple if, and only if, $R$ is $G$-simple and $\alpha$ has the intersection property (see Theorem~\ref{P1-thm:AP2} (iii)), thereby generalizing \cite{Goncalves2014} and exemplifying \cite{Oinert2012}.

In Section~\ref{P1-sec:gen-idemp}, we show that any partial action of a groupoid on a torsion-free, unital, commutative algebra generated by its idempotents, actually corresponds to a
partial action of the same groupoid on the ring of locally constant functions with compact support over a Stone space (see Proposition~\ref{P1-prop:Gequivariant}). Therefore, given an algebraic groupoid partial action as above, we can use our results of Section~\ref{P1-Sec:Algebraic} and topological results of Section~\ref{P1-Sec:TopDynamics} to characterize the primeness, the simplicity and the graded ideals of the partial skew groupoid ring associated to it.

In Section~\ref{P1-Sec:TopDynamics}, we turn our focus to partial skew groupoid rings associated with (groupoid) topological dynamical systems, and study how properties of the dynamical system are reflected in the associated partial skew groupoid ring.
To be precise, we start out with a field $\mathbb{K}$ and a topological partial action $\theta$ of a groupoid $G$ on a zero-dimensional, locally compact, Hausdorff space $X$, and consider the associated partial skew groupoid ring $\mathcal{L}_c(X,\mathbb{K}) \rtimes_\alpha G$.
We show e.g. that
every ideal of $\mathcal{L}_c(X,\mathbb{K}) \rtimes_\alpha G$ is $G$-graded if, and only if, the partial action $\theta$ of $G$ on $X$ is topologically free on every closed invariant subset of $X$, which occurs if, and only if, the associated transformation groupoid $G \rtimes_\theta X$ is strongly effective (see Theorem~\ref{P1-thm:TE1}).
We point out that the dynamical properties appearing in Theorem~\ref{P1-thm:TE1} have analogues in the literature on Steinberg algebras (see e.g. \cite{Brown2014,Clark2019,Steinberg2019}).

In Section~\ref{P1-Sec:Ultragraphs}, 
we apply our results to 
ultragraphs and their associated algebras.
Using the characterization of an ultragraph Leavitt path algebra as a partial skew group ring associated with a certain topological partial action (see \cite{Boava2021,Castro2021,Imanfar2020}),
we show that 
given a field $\mathbb{K}$
and
an ultragraph $\mathcal{G}$ with associated tight spectrum $\textsf{T}$,
one has e.g. that
$\mathcal{G}$ satisfies Condition (K)
if, and only if,
every ideal of $\mathcal{L}_c(\textsf{T},\mathbb{K}) \rtimes_{\alpha}\mathbb{F}$ is $\Z$-graded (see Theorem~\ref{P1-thm:TFinal}).
We point out that some of the equivalences appearing in 
Theorem~\ref{P1-thm:TFinal}
are already known in the context of Leavitt path algebras of ultragraphs, but that several of them are new even in the context of Leavitt path algebras of graphs.

\section{Preliminaries}\label{P1-Sec:Prelim}

In this section, we recall some basic facts on groupoids and their partial actions on sets and rings. We refer the reader to \cite{Bagio2012} and \cite{Gilbert2005} for more details.

\subsection{Groupoids}

By a \emph{groupoid} we shall mean a small category $G$ in which every morphism is invertible.
Each object of $G$ will be identified with its corresponding identity morphism, allowing us to view $G_0$, the set of objects of $G$, as a subset of the set of morphisms of $G$.
The set of morphisms of $G$ will simply be denoted by $G$. This means that
$G_0:=\{gg^{-1} : g \in G\} \subseteq G$.

The \emph{range} and \emph{source} maps
$r,s : G \to G_0$,
indicate the range (codomain)
respectively
source (domain) of each morphism of $G$.
By abuse of notation, the set of \emph{composable pairs} of $G$
is denoted by 
$G^2 := \{(g,h) \in G \times G : s(g)=r(h)\}$.
For each $e \in G_0,$ we denote the corresponding \emph{isotropy group} by $G_e^e:=\{g \in G: s(g)=r(g)=e\}.$

\subsection{Partial actions of groupoids}
Partial actions of (ordered) groupoids on sets and rings were first introduced in \cite{Gilbert2005} and \cite{Bagio2010}, respectively. Below we recall those definitions.

\begin{definition}\label{P1-def:D1}
A \emph{partial action of a groupoid $G$ on a set $X$} is a family of pairs $ \alpha =(X_g,\alpha_g)_{g \in G}$, where $X_g$ is a subset of $X$ and $\alpha_g:X_{g^{-1}}\to X_g$ is a bijection, for all $g\in G$, that satisfy the following:
\begin{enumerate}[{\rm (i)}]
\item $X_g\subseteq X_{r(g)}$, for all $g\in G$,
\item 
$\alpha_e = \id_{X_e}$, for all $e \in G_0$,
\item $\alpha_{h}^{-1}(X_{g^{-1}}\cap X_h)\subseteq X_{(gh)^{-1}}$, whenever $(g,h) \in G^2$,
\item $\alpha_g(\alpha_h(x)) = \alpha_{gh}(x)$, for all $x \in \alpha_h^{-1}(X_{g^{-1}}\cap X_h)$ and $(g,h) \in G^2$.
\end{enumerate}
\end{definition}

We say that $\alpha$ is \emph{global} if $\alpha_g\alpha_h=\alpha_{gh}$ for all $(g,h) \in G^2.$ Similarly to \cite[Lemma~1.1~(i)]{Bagio2012}, $\alpha$ is global if and only if $X_g=X_{r(g)}$ for all $g \in G.$ 

\begin{definition}\label{P1-def:TopPartAction}
A \emph{topological partial action of a groupoid $G$ on a topological space $X$} is a partial action $\alpha\ = (X_g,\alpha_g)_{g \in G}$ of $G$ on the set $X$,
such that $X_g$ is an open subset of $X$ and $\alpha_g : X_{g^{-1}} \to X_g$ is a homeomorphism, for all $g\in G$. 
\end{definition}

\begin{definition}
Let $\alpha =(X_g,\alpha_g)_{g \in G}$ be a partial action of a groupoid $G$ on a set $X$.
A subset $M$ of $X$ is said to be \emph{$G$-invariant} if $\alpha_g(X_{g^{-1}} \cap M) \subseteq M$, for all $g \in G$.
\end{definition}

In the literature, there is a more restrictive notion of global groupoid action in which the groupoid acts on a fibred set (see e.g. \cite{Lawson1998}). We recall the definition.

\begin{definition}\label{P1-D2}
A \emph{groupoid fibred action} is a $4$-tuple $(G,\rho , \theta, X)$ consisting of a groupoid $G,$ a set $X,$ a surjective map $\rho: X \to G_0$ (called the \emph{anchor} or \emph{moment map}) and a map 
  $\theta:   G  \ltimes  X \to X$, $(g,x)  \longmapsto  \theta_g(x):=g \cdot_{\theta} x$, where $G  \ltimes  X=\{ (g,x) : s(g)=\rho(x) \} $ is the pullback,  
  that satisfy the following axioms:
\begin{enumerate}[{\rm (i)}]
   \item $\rho(x) \p x = x,$ for all $x \in X,$
    \item if $(g, h) \in G^2$ and $(h, x) \in  G  \ltimes  X,$ then $(g, h \p x) \in G  \ltimes  X$, 
    \item $(gh)\p x = g \p (h \p x)$, for all $(g, h) \in G^2$ and $(h, x) \in  G  \ltimes  X$.
\end{enumerate}
    An action of a topological groupoid on a topological space is an action $(G, \rho, \theta, X)$, where $G$ is a topological groupoid, $X$ is a Hausdorff space, and the maps $\rho, \theta$ are continuous.
\end{definition}
    
  For the convenience of the reader, we present a proof from \cite[Proposition~4.1]{Gilbert2005}, which relates global actions to  fibred actions.

\begin{proposition}
Let $G$ be a groupoid and let $X$ be a set.  The fibred actions of $G$ on $X$ are in one-to-one correspondence with the global actions of $G$ on $X$ satisfying  $X= \bigsqcup_{e \in G_0}X_e.$
\end{proposition}
    
\begin{proof}
 \noindent  Let $(G, \rho, \theta, X)$ be a groupoid fibred action. 
 For each $g\in G$, 
 define $X_g:= \rho^{-1}(r(g))$
 and
 $\alpha_g:X_{g^{-1}} \to X_g$ by $\alpha_g(x):= g \p x$, for all $x\in X_{g^{-1}}$. Notice that $x\in X_{g^{-1}}$ implies that $\rho(x)=s(g)$ and thus $(g,x)\in G \ltimes X$. Hence, $\alpha_g$ is well-defined. We also have that
     $$X=\rho^{-1}(G_0)= \bigsqcup_{e \in G_0}\rho^{-1}(\{e\})= \bigsqcup_{e \in G_0}X_e.$$
     It is straightforward to verify that $\alpha=(X_g,\alpha_g)_{g\in G}$ is a global action of $G$ on $X$.
     
     Now, let $\alpha=(X_g,\alpha_g)_{g\in G}$ be a global action of $G$ on $X$ satisfying  $X= \bigsqcup_{e \in G_0}X_e$. Define the anchor map $\rho:  X  \to G_0$ by $\rho(x):=e$ whenever $x\in X_e$. Since $\rho(X_e)=e$ for all $e \in G_0$, it follows that $\rho$ is surjective. Define $\theta: G \ltimes X  \to X$ by $g \p x:=\alpha_g(x)$, for all $(g,x)\in G \ltimes X$. It is straightforward to verify that $(G,\rho , \theta, X)$ is a fibred action.
   
 Clearly, the two procedures outlined above are mutually inverse.
\end{proof}

In \cite{Flores2020}, topological dynamics of (global) actions of topological groupoids on topological spaces was 
studied.
In this paper, we aim to study
topological dynamics of partial actions of discrete groupoids on topological spaces.

\begin{definition}\label{P1-def:D3}
A \emph{partial action of a groupoid $G$ on a ring $R$} is a partial action $\alpha\ = (D_g,\alpha_g)_{g \in G}$ of $G$ on the set $R$,
such that $D_{r(g)}$ is an ideal of $R$,
$D_g$ is an ideal of $D_{r(g)}$,
and $\alpha_g : D_{g^{-1}} \to D_g$ is a ring isomorphism, for all $g\in G$.
Given a partial action $\alpha$ of a groupoid $G$ on a ring $R$ one may define the \emph{partial skew groupoid ring $R \rtimes_{\alpha}G$} as the set of all formal sums $\displaystyle\sum_{g \in G} a_g \delta_g$, where $a_g \in D_g$ is zero for all but finitely many $g\in G$ and $\delta_g$ is a symbol. Addition on $R \rtimes_\alpha G$ is defined in the natural way and multiplication is given by the rule
\begin{equation*}
a_g\delta_g \cdot b_h \delta_h :=
\begin{cases}
 \alpha_g(\alpha_{g^{-1}}(a_g)b_h)\delta_{gh},& \text{ if } (g,h) \in G^2,\\
0, & \text{ otherwise. } 
\end{cases}
\end{equation*}
\end{definition}

\begin{obs}\label{P1-rem:associatividade}
Let $\alpha\ =(D_g,\alpha_g)_{g \in G}$ be a partial action of a groupoid $G$ on a ring $R$.

(a) Recall that a ring $T$ is said to be \emph{$s$-unital} if
$t \in tT \cap Tt$ for every $t\in T$.
One can show that for any finite subset $F$ of an $s$-unital ring $T$, there is some $u\in T$ such that $uf=fu=f$ for every $f\in F$ (see \cite[Proposition~2.10]{NystedtSurvey2019}).
In that case, the element $u$ will be referred to as an \emph{$s$-unit for the set $F$}.

(b) In general, the partial skew groupoid ring $R \rtimes_{\alpha}G$ need not be associative (not even for partial actions of groups).
By \cite[Proposition~3.1]{Bagio2010}, if $D_g$ is $(\mathcal{L}, \mathcal{R})$-associative for every $g\in G$, then $R \rtimes_{\alpha}G$ is associative. If $D_g$ is idempotent, then \cite[Proposition~2.5]{Dokuchaev2005} implies that $D_g$ is $(\mathcal{L}, \mathcal{R})$-associative.  In particular, if $D_g$ 
is $s$-unital for every $g\in G$, then $D_g$ is $(\mathcal{L}, \mathcal{R})$-associative and consequently $R \rtimes_{\alpha}G$ is associative.

(c) Suppose that $D_e$ is a unital ring, for all $e\in G_0$. Then, there exists a central idempotent $1_e\in R$ such that $D_e=1_eR$. Moreover, if $G_0$ is finite, then \cite[Proposition~3.3]{Bagio2010} implies that  $R \rtimes_{\alpha}G$ is unital with multiplicative identity element $\sum_{e\in G_0}1_e\delta_e$.
\end{obs}

\subsection{The $G$-grading on $R\rtimes_\alpha G$}

Any partial skew groupoid ring $S:=R \rtimes_\alpha G$ is, in a natural way, graded by the groupoid $G$.
Indeed, we may endow it with a $G$-grading by putting $S_g := D_g \delta_g$, for every $g\in G$.
With that choice of homogeneous components, one easily verifies that $S=\bigoplus_{g\in G} S_g$, and that $S_gS_h \subseteq S_{gh}$ if $(g,h)\in G^2$, and $S_g S_h = \{0\}$ otherwise.
If $I$ is an ideal of the $G$-graded ring $R \rtimes_\alpha G$, then $I$ is said to be a \emph{graded ideal} (or \emph{$G$-graded ideal}) if
$I=\bigoplus_{g \in G} (I \cap D_g\delta_g)$.

The \emph{support} of an element $s = \sum_{g\in G} a_g \delta_g \in R \rtimes_\alpha G$ will be denoted by $\Supp(s)$ and is defined as the finite set $\{g\in G : a_g \neq 0 \}$.
Note that $s = \sum_{g\in \Supp(s)} a_g \delta_g$.

\section{The ideal structure of partial skew groupoid rings}\label{P1-Sec:Algebraic}

Throughout this section, unless stated otherwise, $R$ will denote an arbitrary associative ring, $G$ will denote an arbitrary groupoid and $\alpha = (D_g, \alpha_g)_{g \in G}$ will be an arbitrary partial action of $G$ on $R$, such that $D_g$ is $s$-unital for every $g \in G.$
Moreover, we will assume that
$R=\bigoplus_{e \in G_0}D_e.$

\begin{obs}\label{P1-remark:skew-s-unital}
Note that, by our assumptions, $R$ is clearly $s$-unital. Using that fact, it easily follows that the partial skew groupoid ring $R\rtimes_{\alpha}G$ is $s$-unital.
\end{obs}

We begin by establishing a couple of auxiliary results, which will be useful later on.

\begin{lema}\label{P1-di}
$D_g$ is an ideal of $R$, for all $g \in G.$
\end{lema}

\begin{proof}
Take $g \in G,$ $x \in D_g$ and $y \in R.$
Choose an $s$-unit $u\in D_{r(g)}$ for $x$.
Using that $D_g$ is an ideal of $D_{r(g)}$, and that $D_{r(g)}$ is an ideal of $R$, we note that
$xy = (xu)y = x(uy) \in D_g D_{r(g)} \subseteq D_g$.
Similarly, $yx\in D_g$.
\end{proof}

\begin{obs}\label{P2-InverseSemigroups}
Every partial skew groupoid ring $R \rtimes_\alpha G$ appearing in this paper can in fact be realized as a partial skew inverse semigroup ring.
Indeed, suppose that $\alpha = (D_g, \alpha_g)_{g \in G}$ is a partial action of $G$ on $R$.
Let $S:=G \cup \{z\}$ be the inverse semigroup induced by $G$ (cf.~\cite[p.19]{NystedtSystems2020}).
Define $D_g':=D_g$, for every $g\in G$, and $D'_z:=\{0\}$. Moreover, define $\alpha'_g := \alpha_g$, for every $g\in G$, and $\alpha'_z :=\id_{D'_z}$.
It is not difficult to verify that $\alpha':=(D'_s, \alpha'_s)_{s \in S}$ is a partial action of $S$ on $R$, and that
 $R\rtimes_{\alpha}G$ is isomorphic to the partial skew inverse semigroup ring $R \rtimes_{\alpha'} S$.
\end{obs}

Let $A:=\bigoplus_{e \in G_0} D_e \delta_e \subseteq R \rtimes_\alpha G$.
We now define a 
``projection''
$P_0: R \rtimes_{\alpha} G\to R$
by
\begin{align}\label{P1-map-po}
    P_0 \! \left( \sum_{g \in G} a_g \delta_g \right)
    :=\sum_{e \in G_0} a_e,
\end{align}
and an inclusion
$\psi: R\to A$ by 
\begin{align}\label{P1-map-psi}
    \psi \! \left(\sum_{e \in G_0} a_e\right) :=\sum_{e \in G_0} a_e \delta_e.
\end{align}

\begin{obs}
Note that, in general $P_0$ is not a ring homomorphism.
\end{obs}

\begin{lema}\label{P1-lem:some-obs}
Let $A$ be as above, and let $I$ be an ideal of $R\rtimes_{\alpha} G$. The following assertions hold:

\begin{enumerate}[{\rm (i)}]

\item $A$ is a subring of $R\rtimes_{\alpha} G$ and the map $\psi: R\to A$ is a ring isomorphism. In particular, $R$ is commutative if, and only if, $A$ is commutative.
\item If $\sum_{e \in G_0} a_e\delta_e \in I \cap A$, then $a_f \delta_f \in I \cap A$, for all $f \in G_0$. 

\item If $a\in A$ and $b\in R \rtimes_{\alpha} G$, then $P_0(ab)=P_0(a)P_0(b)$ and $P_0(ba)=P_0(b)P_0(a)$. In particular, $P_0\lvert_{A}:A\to R$ is a ring isomorphism whose inverse is $\psi$, defined in \eqref{P1-map-psi}.

\item The map $E : R \rtimes_{\alpha}G \to A$, 
defined by $E := \psi \circ P_0$,
is an $A$--$A$-bimodule map, in the sense that it is additive and
satisfies $E(a)=a$,  
$E(ab)=aE(b)$ and $E(ba)=E(b)a$,
for all $a\in A$ and $b\in R \rtimes_{\alpha}G$.

\item $P_0(I)=P_0 \! \left(I\cap \bigoplus_{e \in G_0}(D_e \rtimes_{\alpha^e} G_e^e)\right)$, where $\alpha^e:=(D_h,\alpha_h)_{h\in G_e^e}$ is the partial action of the isotropy group $G_e^e$ on $D_e$.
\end{enumerate}
\end{lema}

\begin{proof}
The proof of (i) is straightforward.
For (ii),
suppose that $\sum_{e\in G_0} a_e \delta_e \in I\cap A$.
Take $f\in G_0$ and choose $u_f \in D_f$ to be an $s$-unit for $\alpha_f(a_f)$.
Notice that
$$a_f \delta_f = \alpha_f(\alpha_f(a_f) u_f) \delta_f
= \left( \sum_{e\in G_0} a_e \delta_e \right) u_f \delta_f \in I \cap A.$$
For (iii), take  $a=\sum_{e \in G_0} a_e \delta_e \in A$ and $b = \sum_{g \in G} b_g \delta_g \in R\rtimes_{\alpha} G$. Observe that
$$P_0(ab)  = P_0 \! \left(\sum_{\substack{g \in G}}\alpha_{r(g)}(\alpha_{r(g)}(a_{r(g)})b_{g})\delta_{r(g)g} \right) = \sum_{e \in G_0} a_{e}b_{e} = P_0(a)P_0(b).$$
Similarly, $P_0(ba)=P_0(b)P_0(a).$ The second part of (iii) is straightforward.
The proof of (iv) is straightforward, using (iii) and the definitions of $P_0$ and $\psi$.
Finally, for (v), we take 
$a=\sum_{g \in G}a_g\delta_g \in I$
and put $v:= P_0(a)$. 
Take $e\in G_0$, put $F_e:=G_e^e\cap \Supp(a)$ and $F_0:=G_0\cap \Supp(a)$, and let $u_e \in D_e$ be an $s$-unit for $a_g$ and $\alpha_{g^{-1}}(a_g),$ for all $g \in F_e$. Define  $y_e:= (u_e \delta_e)a(u_e \delta_e)$ and notice that
$$\begin{aligned}
y_e & = \left(u_e \delta_e \right)\left(\sum_{\substack{g \in G \\ s(g)=e}} \alpha_g \left(\alpha_{g^{-1}} \left(a_g \right)u_e \right)\delta_g \right)  
=\left(u_e \delta_e \right)\left(\sum_{g \in G_e^e}a_g\delta_g + \sum_{\substack{g \in G \\ s(g)=e \\ r(g)\neq e}} \alpha_g \left(\alpha_{g^{-1}} \left(a_g \right)u_e \right)\delta_g \right) \\
& = \sum_{g \in F_e} u_ea_g\delta_g +0 
 = \sum_{g \in F_e}a_g\delta_g \in I \cap (D_e\rtimes_{\alpha^e}G_e^e).
\end{aligned}$$
Hence $y:= \sum_{e \in F_0}y_e \in I \cap \bigoplus_{e \in G_0}(D_e \rtimes_{\alpha^e} G_e^e)$ satisfies $P_0(y)= \sum_{e \in F_0}a_e=P_0(a)=v.$ Consequently, $v \in P_0 \! \left( I \cap \bigoplus_{e \in G_0}(D_e \rtimes_{\alpha^e} G_e^e) \right).$ The reverse inclusion is trivial.
\end{proof}

\subsection{$G$-invariant ideals of $R$ versus $G$-graded ideals of $R\rtimes_{\alpha} G$}

In this section, we describe a correspondence between $G$-invariant ideals of the ring $R$ and graded ideals of $R\rtimes_{\alpha} G$.
The proof of the next result is partially inspired by the proof of \cite[Theorem~2.3]{Goncalves2014}.

\begin{lema}\label{P1-lem:inv}
If $I$ is an ideal of $R \rtimes_{\alpha} G$, then $P_0(I)$ is a $G$-invariant ideal of $R.$
\end{lema}

\begin{proof}
Suppose that $I$ is an ideal of $R \rtimes_{\alpha} G$.
Take $a\in I$, $r\in R$, and consider $c:=P_0(a)\in P_0(I)$.
Notice that $\psi(r)a \in I$ and $a \psi(r) \in I$, where $\psi$ is the map defined in \eqref{P1-map-psi}. By Lemma~\ref{P1-lem:some-obs}~(iii), we have that $cr=P_0(a)P_0(\psi(r))=P_0(a\psi(r))\in P_0(I)$. Similarly, $rc \in P_0(I)$. 
Thus, $P_0(I)$ is an ideal of $R.$

Now, take $t\in G$
and $d_t \in P_0(I) \cap D_t.$
Then, there is an element $b: = \sum_{g\in G}b_g\delta_g \in I$ such that $P_0(b) = d_t \in D_t \subseteq D_{r(t)}$  and 
$b = d_t \delta_{r(t)} + \sum_{g \notin G_0} b_g \delta_g$. Let $u \in D_t$ be an $s$-unit for $d_t$. 
Consider the element 
$y: = \left(\alpha_{t^{-1}}\left(u\right) \delta_{t^{-1}}\right) b \left(u\delta_t\right) \in I.$ Observe that if $g \in \Supp(b),$ then $t^{-1}g t \in G_0$ if, and only if, $g=tt^{-1}=r(t)$. Therefore,
\begin{align}\label{P1-proof-inv} 
\alpha_{t^{-1}}(u) \delta_{t^{-1}} d_t \delta_{r(t)} u \delta_t 
=\alpha_{t^{-1}}\left(u \right) \delta_{t^{-1}} d_t \delta_t = \alpha_{t^{-1}}\left(\alpha_t\left(\alpha_{t^{-1}}\left(u \right)\right)d_t \right) \delta_{t^{-1}t} = \alpha_{t^{-1}} \left(d_t \right)\delta_{s(t)}.\end{align} 
We now conclude that
$\alpha_{t^{-1}} \left(d_t \right)=
P_0(\alpha_{t^{-1}}(u) \delta_{t^{-1}} \left(d_t \delta_{r(t)} \right) u \delta_t)
=P_0(y) \in P_0(I)$
and hence
$P_0(I)$ is  $G$-invariant.
\end{proof}

\begin{lema}\label{P1-lem:inv2}
If $I$ is an ideal of $R \rtimes_{\alpha} G$, then $P_0(I\cap\bigoplus_{e \in G_0} D_e\delta_e )$ is a $G$-invariant ideal of $R.$
\end{lema}

\begin{proof}
Suppose that $I$ is an ideal of $R \rtimes_{\alpha} G$.
Using the same arguments as in the proof of Lemma~\ref{P1-lem:inv}, it is easy to verify that   $P_0(I\cap\bigoplus_{e \in G_0} D_e\delta_e)$ is an ideal of $R$. 
Take $t\in G$ and $d_t \in P_0(I\cap\bigoplus_{e \in G_0} D_e\delta_e) \cap D_t.$ Then, there is an element $b \in I\cap\bigoplus_{e \in G_0} D_e\delta_e$ such that $P_0(b) = d_t \in D_t \subseteq D_{r(t)}$ and 
$b = d_t \delta_{r(t)}.$ Let $u\in D_t$ be an $s$-unit for $d_t$. Consider the element $y := \alpha_{t^{-1}}(u) \delta_{t^{-1}} b u \delta_t  \in I$. Similarly to \eqref{P1-proof-inv},  $y= \alpha_{t^{-1}} \left(d_t \right)\delta_{s(t)} \in 
\bigoplus_{e \in G_0} D_e\delta_e.$ 
Thus, $\alpha_{t^{-1}}(d_t) = P_0(y) \in P_0 \! \left( I\cap\bigoplus_{e \in G_0} D_e\delta_e \right)$
and hence $P_0 \! \left( I\cap\bigoplus_{e \in G_0} D_e\delta_e \right)$ is 
$G$-invariant.
\end{proof}

We define the map $\fr:\{\text{Ideals of } R\rtimes_{\alpha}G \}\to \{G\text{-invariant ideals of } R\}$ by 
\begin{align}\label{P1-map-fr}
\fr(I):=P_0 \! \left( I\cap\bigoplus_{e \in G_0} D_e\delta_e \right).    
\end{align}
Notice that, by Lemma~\ref{P1-lem:inv2}, $\fr$ is well-defined. 

\begin{obs}
Let $J$ be a $G$-invariant ideal of $R$. One can show that $\alpha_g(D_{g^{-1}} \cap J) =D_g \cap J$, for all $g \in G$. Thus, we may define a partial action $\alpha^J:=(J\cap D_g, \alpha^J_g)_{g\in G}$ of $G$ on $J$, where $\alpha^J_g:J\cap D_{g^{-1}}\to J\cap D_g$, for all $g\in G$.
Note that there is a natural injection of rings $\iota : J\rtimes_{\alpha^J} G \to R\rtimes_{\alpha} G$.
We will often suppress the use of $\iota$ and simply identify $\iota(J\rtimes_{\alpha^J} G)$ with $J\rtimes_{\alpha^J} G$.
\end{obs}

\begin{proposition}\label{P1-prop:grad}
The following assertions hold:
\begin{enumerate}[{\rm (i)}]
    \item If $J$ is a $G$-invariant ideal of $R$, then $J\rtimes_{\alpha^J} G$ is a $G$-graded ideal of $R\rtimes_{\alpha} G$.
    \item If $I$ is a $G$-graded ideal of $R \rtimes_{\alpha} G$, then $P_0 \! \left( I\cap \bigoplus_{e \in G_0} D_e\delta_e \right)= P_0(I)$.
\end{enumerate}
\end{proposition}

\begin{proof}
Suppose that $J$ is a $G$-invariant ideal of $R$ and that $I$ is a $G$-graded ideal of $R\rtimes_\alpha G$.

(i) By definition, $J\rtimes_{\alpha^J} G =\bigoplus_{g \in G} (J \cap D_g)\delta_g$. It is not difficult to verify that $J\rtimes_{\alpha^J} G$ is an ideal of $R\rtimes_{\alpha} G$, and clearly $(J \cap D_g)\delta_g=(J \rtimes_{\alpha} G)\cap D_g\delta_g$, for all $g\in G$. Hence, $J\rtimes_{\alpha^J} G$ is a $G$-graded ideal of $R\rtimes_{\alpha} G$.\smallbreak

\noindent (ii)  Take $a \in P_0(I)$ and choose $x :=\sum_{t \in G} x_t\delta_t \in I$ such that $P_0(x)=\sum_{e \in G_0} x_e = a$. Note that
$\sum_{e \in G_0} x_e\delta_e \in I \cap \bigoplus_{e \in G_0} D_e \delta_e$, because $I$ is $G$-graded.
Therefore, $a = P_0 \! \left(\sum_{e \in G_0} x_e\delta_e\right) \in P_0 \! \left(I\cap \bigoplus_{e \in G_0} D_e \delta_e \right)$.
The other inclusion is trivial.
\end{proof}

\begin{proposition}\label{P1-prop:psj}
Let $I$ be an ideal of $S:=R \rtimes_{\alpha} G$, and write $J:=P_0(I)$, $I_0:=I \cap \bigoplus_{e \in G_0}D_e\delta_e$ and $J_0:=P_0(I_0)$.  The following assertions hold: 
\begin{enumerate}[{\rm (i)}]
    \item $I \subseteq J \rtimes_{\alpha^J} G$. 
    \item If $b_t\delta_t \in J_0 \rtimes_{\alpha^{J_0}} G$ for some $t\in G$ and $b_t\in D_t \cap J_0$, then $b_t\delta_{r(t)}\in I_0.$
    \item $J_0 \rtimes_{\alpha^{J_0}} G \subseteq I.$
    \item $J_0\rtimes_{\alpha^{J_0}}G = SI_0S.$
\end{enumerate}
\end{proposition}

\begin{proof}
Notice that,
by Lemma~\ref{P1-lem:inv} and  Lemma~\ref{P1-lem:inv2},
both $J$ and $J_0$ are $G$-invariant ideals of $R$.
Thus, $J \rtimes_{\alpha^J} G$ and $J_0 \rtimes_{\alpha^{J_0}} G$ are well-defined.

\noindent (i) Let $a=\sum_{g \in G} a_g\delta_g \in I.$ Take $t \in G$ and let $u_t \in D_t$ be an $s$-unit for $a_t$.
Consider the element $y:=a(\alpha_{t^{-1}}(u_t)\delta_{t^{-1}}) \in I$, and observe that
$$y:=\sum_{\substack{g\in G \\ s(g)=s(t)}} \alpha_g(\alpha_{g^{-1}}(a_g)\alpha_{t^{-1}}(u_t))\delta_{gt^{-1}}.$$ Thus,
$P_0(y)  =\alpha_t(\alpha_{t^{-1}}(a_t)\alpha_{t^{-1}}(u_t))= \alpha_t(\alpha_{t^{-1}}(a_t u_t))= a_t \in J \cap D_t$.
We conclude that $a \in \bigoplus_{g\in G}(J\cap D_t) \delta_t= J \rtimes_{\alpha^J}G.$ \smallbreak

\noindent (ii) Suppose that $b_t\delta_t \in J_0 \rtimes_{\alpha^{J_0}} G.$ Then, $b_t \in D_t \cap J_0=D_t\cap P_0(I_0)\subseteq D_{r(t)} \cap P_0(I_0).$ Therefore, there is an element $a=\sum_{e \in G_0}a_e\delta_e \in I_0$ such that $P_0(a)=\sum_{e \in G_0}a_e=b_t \in D_{r(t)}.$ Notice that
$b_t-a_{r(t)}=\sum_{\substack{e \in G_0 \\ e \neq r(t)}}a_e \in D_{r(t)} \cap \sum_{\substack{e \in G_0 \\ e \neq r(t)}}D_e = \{0\}.$
Thus, $a=b_t\delta_{r(t)}.$
\smallbreak

\noindent (iii)
By Proposition~\ref{P1-prop:grad} (i),  $J_0 \rtimes_{\alpha^{J_0}} G$ is a $G$-graded ideal of $R \rtimes_{\alpha} G.$
Take $t\in G$ and let $b_t \delta_t \in J_0 \rtimes_{\alpha^{J_0}} G$. By (ii) we get that,
$b_t \delta_{r(t)} \in I_0 \subseteq I$. Let $v \in D_t$ be an $s$-unit for $b_t$. Then,
$b_t \delta_t = (b_t \delta_{r(t)})(v\delta_t) \in I$.

\noindent (iv)
Take $u:=\sum_{g \in G}a_g\delta_g \in S$, $v:=\sum_{h \in G}c_h\delta_h \in S$ and $x:=\sum_{e \in G_0}b_e \delta_e \in I_0.$
Since $I$ is an ideal of $S$, it follows from Lemma~\ref{P1-lem:some-obs} (ii)  that $b_f \in J_0$, for all $f \in G_0.$ Notice that $$uxv =\left(\sum_{g \in G}a_g\delta_g\right)\left(\sum_{e \in G_0}b_e \delta_e\right)\left(\sum_{h \in G}c_h\delta_h\right)=\sum_{\substack{e \in G_0 \\ s(g)=r(h)=e}}\alpha_g(\alpha_{g^{-1}}(a_g)b_{r(h)}c_h)\delta_{gh}.$$ 
Using that $J_0$ is a $G$-invariant ideal of $R$, we get that $uxv \in \bigoplus_{g \in G}\big(J_0\cap D_g\big)\delta_g = J_0\rtimes_{\alpha^{J_0}}G.$
This shows that
$SI_0S \subseteq J_0\rtimes_{\alpha^{J_0}}G$.

Conversely, by Lemma~\ref{P1-lem:inv2} and Proposition~\ref{P1-prop:grad} (i), $J_0\rtimes_{\alpha^{J_0}}G$ is a $G$-graded ideal of $S.$ Let $d_g\delta_g \in J_0 \rtimes_{\alpha^{J_0}} G.$ Then, by (ii), $d_g\delta_{r(g)}\in I_0.$ Let $u_g \in D_g$ be an $s$-unit for $d_g$. A short calculation now reveals that $d_g\delta_g=(u_g\delta_{r(g)})( d_g\delta_{r(g)})(u_g \delta_g) \in SI_0S.$
\end{proof}

Now, consider the map $\Psi: \{G\text{-invariant ideals of } R\} \to \{G\text{-graded ideals of } R\rtimes_{\alpha}G \}$ defined by $\Psi(J):=J\rtimes_{\alpha^J} G$. By Proposition~\ref{P1-prop:grad} (i), $\Psi$ is well-defined. Also, consider the restriction  $\fr_{gr}:\{G\text{-graded ideals of } R\rtimes_{\alpha}G \}\to \{G\text{-invariant ideals of } R\}$ 
of the map $\fr$ defined in \eqref{P1-map-fr}. 
Notice that, by Proposition~\ref{P1-prop:grad} (ii), $\fr_{gr}(I)=P_0(I).$ 

We will now show that the above map $\Psi$  is in fact a bijection. The corresponding result for partial actions of groups was proved in \cite[Theorem~4.7, Proposition~12.2]{Lannstrom2021}.

\begin{theorem}\label{P1-thm:corresp1}
Let $R$ be a ring, let $G$ be a groupoid and let $\alpha=(\alpha_g, D_g)_{g \in G}$ be a partial action of $G$ on $R$, such that $D_g$ is $s$-unital for every $g \in G$, and $R =\bigoplus_{e \in G_0}D_e.$
Then $\Psi$ is a bijection whose inverse is $\fr_{gr}$ defined above. In particular, there is a one-to-one correspondence between $G$-graded ideals of $R \rtimes_{\alpha}G$ and $G$-invariant ideals of $R$.
\end{theorem}

\begin{proof}
Let $J$ be a $G$-invariant ideal of $R$. By Proposition~\ref{P1-prop:grad} (i) and (ii), $$J=P_0 \! \left((J \rtimes_{\alpha^J}G)\cap \bigoplus_{e \in G_0}D_e\delta_e\right)=P_0( J \rtimes_{\alpha^J}G)=P_0(\Psi(J))=\fr_{gr}(\Psi(J)).$$
Now, let $I$ be a $G$-graded ideal of $R\rtimes_{\alpha}G$, and write
$J:=P_0(I)$, $I_0:=I \cap \bigoplus_{e \in G_0}D_e\delta_e$ and $J_0:=P_0(I_0)$.
Notice that, by Proposition~\ref{P1-prop:grad} (ii), $J=J_0$.
By Proposition~\ref{P1-prop:psj} (i) and (iii), 
we get that $I= J\rtimes_{\alpha^J}G=\Psi(P_0(I))=\Psi(\fr_{gr}(I)).$
This shows that $\Psi$ and $\fr_{gr}$ are mutual inverses.
\end{proof}

\subsection{The residual intersection property}

In order to characterize graded ideals in partial skew groupoid rings, we introduce the notion of a residual intersection property. To that end, we first need to define partial actions on quotient rings.

\begin{obs}
Suppose that $J$ is a $G$-invariant ideal of $R$.
Notice that $J$ is not necessarily contained in $D_g$ for any $g \in G.$ On the other hand, for each $g \in G,$ $D_g \cap J$ is an ideal of $D_g,$ and by the second isomorphism theorem, $\frac{D_g}{D_g \cap J}$ is isomorphic to $\frac{D_g+J}{J}.$
\end{obs}

Guided by the above remark, we make the following definition.

\begin{definition}\label{P1-def:quotientpartial}
Let $R\rtimes_\alpha G$ be a partial skew groupoid ring.
Suppose that $J$ is a $G$-invariant ideal of $R$. The \emph{quotient partial action of $G$ on $R/J$}, denoted by $\overline{\alpha}:=(\overline{D}_g,\overline{\alpha}_g)_{g\in G}$, is defined by: 
\begin{align}\label{P1-par-quot}
   &\overline{D}_g:= \frac{D_g+J}{J},& &\overline{\alpha}_g: \overline{D}_{g^{-1}} \to \overline{D}_g,\quad x+J\mapsto \alpha_g(x)+J,& &\text{for all }x+J\in \overline{D}_{g^{-1}}.&
\end{align}
We will sometimes refer to $\overline{\alpha}=(\overline{D}_g,\overline{\alpha}_g)_{g\in G}$ as the \emph{quotient partial action (of $\alpha$) with respect to $J$}.
\end{definition}

Notice that, since $J$ is a $G$-invariant ideal of $R$, $\overline{\alpha}_g$ is well-defined, for all $g\in G$. Also observe that $\overline{\alpha}_g$ is a ring isomorphism with inverse $\overline{\alpha}_{g^{-1}}$.

\begin{lema}\label{P1-lem:prop-local-units}
 Let $R$ be a ring and let $I, J, K$ be ideals of $R$. If $I$ or $K$ is $s$-unital, then $(I+J)\cap (K+J)=(I \cap K) +J.$
\end{lema}

\begin{proof}
Let $x \in (I+J)\cap (K+J).$ There exist $i \in I, \,\,k \in K$ and $j, j' \in J$ such that $x= i+j=k+j'.$
Without loss of generality, suppose that $K$ is $s$-unital. Let $u_k \in K$ be an $s$-unit for $k$, and notice that 
$$x=k+j'=u_k k+j'=u_k(i+j-j')+j'=u_ki+u_k(j-j')+j' \in (K \cap I)+J.$$
The reverse inclusion is trivial.
\end{proof}

We will now show that the partial actions that we consider in this section (defined on $s$-unital ideals) will always have well-defined associated quotient partial actions.

\begin{proposition}\label{P1-prop:prop-inter}
Let $J$ be a $G$-invariant ideal of $R$, 
and let $\overline{\alpha}=(\overline{D}_g, \overline{\alpha}_g)_{g \in G}$ be defined as in Definition~\ref{P1-def:quotientpartial}.
Then $\overline{\alpha}$ is a partial action of $G$ on $R/J$.
\end{proposition}

\begin{proof}
Using that $D_g$ is $s$-unital, for every $g\in G$, by Lemma~\ref{P1-lem:prop-local-units}, we get that 
$(D_h +J)\cap (D_{g^{-1}}+J)=(D_h \cap D_{g^{-1}})+J $, for all $g,h \in G.$
Since $\alpha$ is a partial action of $G$ on $R,$ it follows that $\overline{D}_{r(g)}$ is an ideal of $R/J,$ $\overline{D}_g$ is an ideal of $\overline{D}_{r(g)},$ for all $g \in G,$ and $\overline{\alpha}_e=\id_{\overline{D}_e},$ for all $e \in G_0.$  Notice that
$\overline{D}_{g^{-1}}\cap \overline{D}_h = ((D_{g^{-1}}+J)\cap(D_h+J))/J.$ By assumption, $((D_{g^{-1}}+J)\cap(D_h+J))/J= ((D_{g^{-1}}\cap D_h)+J)/J.$ Let $(g,h) \in G^2$ and $x \in \overline{D}_{g^{-1}}\cap \overline{D}_h = ((D_{g^{-1}}\cap D_h)+J)/J.$ There is some $z \in D_{g^{-1}}\cap D_h$ such that $x=z+J.$ Therefore, $\alpha_{h^{-1}}(z) \in \alpha_{h^{-1}}(D_{g^{-1}}\cap D_h) \subseteq D_{h^{-1}g^{-1}}$. Consequently,
$\overline{\alpha}_{h^{-1}}(x)=\alpha_{h^{-1}}(z)+J \in \overline{D}_{h^{-1}g^{-1}}.$ It is clear that 
$\overline{\alpha}_{gh}(y)=\overline{\alpha}_g(\overline{\alpha}_h(y))$, for all $y \in \overline{\alpha}_{h^{-1}}(\overline{D}_{g^{-1}}\cap \overline{D}_h)$. 
\end{proof}

\begin{obs}\label{P1-obs-local-units}
Let $R$ be a ring and let $I, J$ be ideals of $R.$ If $I$ is $s$-unital, then $(I+J)/J$ is $s$-unital. Indeed, consider $i+J$ for $i \in I$.
Let $u\in I$ be an $s$-unit for $i$.
We get that
$(u+J)(i+J)= i + J$ and $(i+J)(u+J)= i + J$.
\end{obs}

\begin{proposition}\label{P1-prop:quotientSkew} 
Let $J$ be a $G$-invariant ideal of $R,$ and let $\overline{\alpha}=(\overline{D}_g, \overline{\alpha}_g)_{g \in G}$ be defined as in Definition~\ref{P1-def:quotientpartial}. Then $\overline{D}_g$ is $s$-unital for every $g \in G,$ the partial skew groupoid ring $R/J \rtimes_{\overline{\alpha}}G$ is associative, and $R/J=\bigoplus_{e \in G_0}\overline{D}_e.$
\end{proposition}

\begin{proof}
For each $g\in G$, the ideal $D_g$ of $R$ is $s$-unital and hence, by Remark~\ref{P1-obs-local-units}, the ideal $\overline{D}_g$ is $s$-unital.
Thus, Remark~\ref{P1-rem:associatividade} (b) implies that $R/J \rtimes_{\overline{\alpha}}G$ is associative. For the last part, if there is $e \in G_0$ such that $x \in \overline{D}_e \cap \sum_{f \in G_0\setminus\{e\}}\overline{D}_f,$ then $x=b_e + J \in \overline{D}_e$ and  $x=\sum_{f\in G_0\setminus\{e\}}(b_f+J) \in \sum_{f \in G_0\setminus\{e\}}\overline{D}_f.$ Let $u_e+J \in \overline{D}_e$ be an $s$-unit for $b_e + J$. Then, $x=(b_e + J)(u_e+J)=\sum_{f\in G_0\setminus\{e\}}(b_f+J)(u_e+J)=0+J.$
For any $a=\sum_{e \in G_0}a_e \in R,$ we have $a+J=(\sum_{e \in G_0}a_e)+J=\sum_{e \in G_0}(a_e+J) \in \bigoplus_{e \in G_0}\overline{D}_e$,
because $R=\bigoplus_{e \in G_0}D_e.$
\end{proof}

The following definition is inspired by \cite{Oinert2012} and \cite{Thierry2014}. 

\begin{definition}
Let $R\rtimes_\alpha G$ be a partial skew groupoid ring.
\begin{enumerate}[{\rm (a)}]
    \item The partial action $\alpha$ is said to have the \emph{intersection property} if
$I\cap \bigoplus_{e \in G_0} D_e\delta_e\neq \{0\}$, for every nonzero ideal $I$ of $R \rtimes_{\alpha} G$.
    \item The partial action $\alpha$ is said to have the \emph{residual intersection property} if, for every  $G$-invariant ideal $J$ of $R$, the quotient partial action $\overline{\alpha}$ with respect to $J$
    has the intersection property.
    
\end{enumerate}
\end{definition}

\begin{obs}
By considering $J=\{0\}$, it immediately becomes clear that
the residual intersection property implies the intersection property.
In Example~\ref{P1-ex:exnrip}, however, we will see that the converse does not hold in general.
That is, the residual intersection property is strictly stronger than the intersection property.
\end{obs}

\begin{obs}
Let $J$ be a $G$-invariant ideal of $R.$ Then, we have a natural injection of rings
$\iota: J \rtimes_{\alpha^J} G \to R \rtimes_{\alpha} G$ and a natural surjection of rings
$\pi: R\rtimes_{\alpha}G \to R/J  \rtimes_{\overline{\alpha}}G $. Moreover, it is easy to check that 
\begin{align}\label{P1-exact-sequence}
	\xymatrix{& 0\ar[r]  &J \rtimes_{\alpha^J} G\ar[r]^{\iota} &R \rtimes_{\alpha} G\ar[r]^{\pi}&  R/J  \rtimes_{\overline{\alpha}}G\ar[r]&0}
\end{align}
 is a short exact sequence of rings.
\end{obs}

\begin{proposition}\label{P1-prop:t2}
The partial action $\alpha$ has the residual intersection property if, and only if, every ideal of $R \rtimes_{\alpha} G$ is $G$-graded.
\end{proposition}

\begin{proof}
We first show the ''only if'' statement.
Suppose that $\alpha$ has the residual intersection property.
Let $I$ be an ideal of $R \rtimes_{\alpha} G.$  By Lemma~\ref{P1-lem:inv2}, $J_0:=P_0 \! \left(I \cap \bigoplus_{e \in G_0} D_e\delta_e\right)$ is a $G$-invariant ideal of $R$. Hence, it follows from Proposition~\ref{P1-prop:grad} (i) that  $J_0 \rtimes_{\alpha^{J_0}} G$ is a $G$-graded ideal of $R \rtimes_{\alpha} G$. We claim that $I= J_0 \rtimes_{\alpha^{J_0}} G.$ The inclusion $\supseteq$ follows from Proposition~\ref{P1-prop:psj} (iii). For the other inclusion, observe that, since $\pi: R \rtimes_{\alpha} G \to R/J_0  \rtimes_{\overline{\alpha}}G$ is surjective, $\pi(I)$ is an ideal of $R/J_0  \rtimes_{\overline{\alpha}}G$.\smallbreak

\noindent{\bf Claim} $\pi(I) \cap \bigoplus_{e \in G_0} \overline{D}_e\delta_e=\{0\}$, where $\overline{D}_e:=\frac{D_e+J_0}{J_0}$.

\noindent Take $x \in \pi(I) \cap \bigoplus_{e \in G_0} \overline{D}_e\delta_e $ and $y=\sum_{t \in F} a_t\delta_t \in I$, such that 
$F=\Supp(y)$ and $\pi(y)=x.$ 
Let $F_0\subseteq G_0$ be finite and such that $x =\sum_{e \in F_0}(a_e+J_0)\delta_e$.
Take $t\in F$.
Let $u \in D_t$ be an $s$-unit for $a_t$ and $\alpha_{t^{-1}}(a_t).$
Notice that
$a_t\delta_{r(t)} = u \delta_{r(t)} y u \delta_{r^{-1}} \in I\cap \bigoplus_{e \in G_0} D_e\delta_e.$
It now follows that $a_t \delta_t = (a_t\delta_{r(t)})(u\delta_t)\in I$.
We get that
$\sum_{e \in F_0}a_e \delta_e= y- \sum_{t \in F \setminus F_0} a_t\delta_t \in I \cap \bigoplus_{e \in G_0} D_e\delta_e.$
Let $f \in F_0$ and choose an $s$-unit $u_f \in D_f$  for $a_f.$
Observe that
$a_f \delta_f=(u_f \delta_f)\left(\sum_{e \in F_0}a_e\delta_e\right) \in I \cap \bigoplus_{e \in G_0} D_e\delta_e$. Then, $a_f \in P_0(I \cap \bigoplus_{e \in G_0} D_e\delta_e)=J_0$ and hence $a_f+J_0=0+J_0$, for all $f \in F_0$. Thus, $x=0$. 
\smallbreak

 Since $\alpha$ has the residual intersection property it follows from the claim above that $\pi(I)=\{0\}$. Consequently, $I \subseteq \nucleo \pi = \iota(J_0 \rtimes_{\alpha^{J_0}} G)=J_0 \rtimes_{\alpha^{J_0}} G$, where $\iota$ and $\pi$ are the ring homomorphisms given in \eqref{P1-exact-sequence}. 
Thus, $J_0 \rtimes_{\alpha^{J_0}} G = I$ and we have that $I$ is $G$-graded.\smallbreak

Now we show the ``if'' statement.
Suppose that every ideal of $R \rtimes_{\alpha} G$ is $G$-graded.
Let $J$ be a $G$-invariant ideal of $R$ and suppose that $I$ is an ideal of $R/J \rtimes_{\overline{\alpha}} G$ such that 
$I \cap \left(\bigoplus_{e \in G_0} \overline{D}_e \delta_e\right) =\{0\}$, where $
\overline{D}_e:=(D_e+J)/J$.
We want to show that $I=\{0\}$.
By assumption, $\pi^{-1}(I)$ is a $G$-graded ideal of $R\rtimes_{\alpha}G$ and thus
$\pi^{-1}(I)= \bigoplus_{t \in G}\pi^{-1}(I) \cap D_t \delta_t$. As $\pi$ is surjective, we have 
$$\pi(\pi^{-1}(I) \cap \bigoplus_{e \in G_0} D_e \delta_e) \subseteq \pi(\pi^{-1}(I)) \cap \pi\left(\bigoplus_{e \in G_0} D_e \delta_e\right)= I \cap \bigoplus_{e \in G_0} \overline{D}_e \delta_e=\{0\}.$$
Therefore, 
$J_0:=\pi^{-1}(I) \cap \bigoplus_{e \in G_0} D_e \delta_e\subseteq \bigoplus_{e \in G_0} (J \cap D_e) \delta_e$ and consequently we obtain that 
$L:=P_0(J_0) \subseteq \bigoplus_{e \in G_0} (J \cap D_e)=J$, because $R=\bigoplus_{e\in G_0} D_e$. Thus, $L \rtimes_{\alpha^L}G \subseteq J \rtimes_{\alpha^J} G$.
Next we show that $\pi^{-1}(I) \subseteq L \rtimes_{\alpha^L}G$. Take $t\in G$ and $b_t \delta_t \in \pi^{-1}(I)$.  Let $w \in D_{t^{-1}}$ be an $s$-unit for $\alpha_{t^{-1}}(b_t)$ and note that $b_t \delta_{r(t)} = (b_t\delta_t)(w\delta_{t^{-1}}) \in \pi^{-1}(I)$. Hence $b_t \delta_{r(t)} \in \pi^{-1}(I)\cap D_{r(t)} \delta_{r(t)}$, which implies that $b_t \in P_0(\pi^{-1}(I)\cap D_{r(t)} \delta_{r(t)}) \subseteq L$. Thus, $b_t \delta_t \in L \rtimes_{\alpha^L}G
\subseteq J \rtimes_{\alpha^J} G = \iota(J \rtimes_{\alpha^J} G)=\ker \pi$. Consequently, $I = \pi(\pi^{-1}(I))=\{0\}$.
\end{proof}

The following theorem is similar to \cite[Theorem~3.2]{Thierry2014} which was proved in the context of partial actions of groups on $C^*$-algebras.

\begin{theorem}\label{P1-thm:bijpri}
Let $R$ be a ring, let $G$ be a groupoid and let $\alpha=(\alpha_g, D_g)_{g \in G}$ be a partial action of $G$ on $R$, such that $D_g$ is $s$-unital for every $g \in G$ and $R =\bigoplus_{e \in G_0}D_e.$ The following statements are equivalent: 
\begin{enumerate}[{\rm (i)}]
    \item There is a one-to-one correspondence between the ideals of $R \rtimes_{\alpha}G$ and the $G$-invariant ideals of $R$ (given by \eqref{P1-map-fr});
    \item Every ideal of $R \rtimes_{\alpha} G$ is $G$-graded;
    \item $\alpha$ has the residual intersection property.
\end{enumerate}
\end{theorem}

\begin{proof}
Notice that (iii)$\Leftrightarrow$(ii)$\Rightarrow$(i) follows from Proposition~\ref{P1-prop:t2} and Theorem~\ref{P1-thm:corresp1}. In order to prove that (i)$\Rightarrow$(ii), suppose that $\fr$ is bijective. Let $I$ be an ideal of ${R \rtimes_{\alpha}G}.$ By Lemma~\ref{P1-lem:inv2}, $\fr(I)=P_0(I \cap \bigoplus_{e \in G_0}D_e \delta_e)$ is a $G$-invariant ideal of $R.$ Then, by Proposition~\ref{P1-prop:grad} (i), $J:=\Phi(I) \rtimes_{\alpha^{\phi(I)}} G$ is a $G$-graded ideal of $R \rtimes_{\alpha}G.$ Observe that $\fr(J)= P_0 \! \left( \left( P_0 \! \left(I \cap \bigoplus_{e \in G_0}D_e \delta_e\right) \rtimes_{\alpha^{\phi(I)}} G \right) \cap \bigoplus_{e \in G_0}D_e \delta_e\right)=P_0 \! \left(I \cap \bigoplus_{e \in G_0}D_e \delta_e\right)=\fr(I).$ Since $\fr$ is injective, $I=J$ and $I$ is $G$-graded. 
\end{proof}

Let $T$ be a ring and let $S$ be a nonempty subset of $T$. Recall that the {\it centralizer of $S$ in $T$} is the subring $C_T(S):=\{t\in T\,:\,ts=st\,\text{ for all }\, s\in S\}$ of $T$. A subring $S$ of $T$ is said to be a {\it maximal commutative subring of $T$} if $C_T(S)=S.$

We define $\tau : R \rtimes_\alpha G \to R$ by $\tau\left( \sum_{g \in G} a_g \delta_g\right):= \sum_{g \in G} a_g$, and note that it is additive. %The following result partially generalizes \cite[Theorem~3]{Oinert2012} and \cite[Theorem~4]{Oinert2012} to possibly non-unital rings. 
The following result generalizes \cite[Theorem~2.1]{Goncalves2014} 
from 
partial skew group rings
to
partial skew groupoid rings.
Nevertheless, we point out that it actually follows from \cite[Proposition~5.7]{NystedtSystems2020}.
For the convenience of the reader, we include a proof.

\begin{proposition}\label{P1-prop:max-com-int-prop}
Let $R\rtimes_\alpha G$ be a partial skew groupoid ring, and suppose that $R$ is commutative. The following statements are equivalent:
\begin{enumerate}[{\rm (i)}]
        \item $\bigoplus_{e \in G_0} D_e\delta_e$ is a maximal commutative subring of $R  \rtimes_{\alpha}G;$
    \item $\alpha$ has the intersection property. 
\end{enumerate} 
\end{proposition}

\begin{proof}
Write $S:=R \rtimes_\alpha G$ and $A:=\bigoplus_{e \in G_0} D_e\delta_e$.

(i)$\Rightarrow$(ii) 
Suppose that $A$ is maximal commutative in $S.$ Let $I$ be a nonzero ideal of $S$.
Choose $x = \sum_{t \in F} x_t \delta_t \in I$ to be a nonzero element of minimal support amongst all nonzero elements of $I$. 
Pick $g \in F$ such that $x_g \neq 0.$
Let $u \in D_{g^{-1}}$ be an $s$-unit for $\alpha_{g^{-1}}(x_g)$ and define 
$y := x u\delta_{g^{-1}} \in I.$
Notice that
$$y = x_g\delta_g u\delta_{g^{-1}} + \sum_{t \in F\setminus \{g\}} x_t\delta_t u\delta_{g^{-1}} 
= x_g \delta_{r(g)} + \sum_{\substack{t \in F\setminus \{g\} \\ s(t)=s(g) }} \alpha_t(\alpha_{t^{-1}}(x_t)u)\delta_{tg^{-1}}. $$
Hence, $|\Supp(y)| \leq  |\Supp(x)|$ and $y \neq 0.$
Let $b = \sum\limits_{e \in G_0} b_e\delta_e \in A$ and define $z:= by - yb \in I.$ Observe that $|\Supp(z)| < |\Supp(y)| \leq |\Supp(x)|,$ because
$b x_g \delta_{r(g)}-x_g \delta_{r(g)}b= (b_{r(g)}x_g - x_g b_{r(g)}) \delta_{r(g)} =0.$
By minimality of the support of $x$, we get that $z=0.$ Thus,  $b y = y b$ for all $b \in A.$ By assumption, $y \in A$ and hence $I \cap A \neq \{0\}.$

(ii)$\Rightarrow$(i) 
We show the contrapositive statement.
Suppose that $A$ is not maximal commutative in $S$.
There is a nonzero element $a = \sum_{t \in G}a_t\delta_t \in C_S(A)  \setminus A$. In particular, there exist some $e\in G_0$ and $g\in \Supp(a)\setminus G_0$, with $s(g)=r(g)=e$, such that $a_g \delta_g \in C_S(A).$ 
Let $L$ be the ideal of $S$ generated by $a_g \delta_e - a_g \delta_g$.
Using that $S$ is $s$-unital, by Remark~\ref{P1-remark:skew-s-unital}, it is clear that $L$ is nonzero.

We claim that $L \subseteq \ker(\tau)$.
If we assume that the claim holds, then since $\tau\lvert_A = P_0\lvert_A$ is a ring isomorphism, by Lemma~\ref{P1-lem:some-obs}, we get that
$A \cap L \cong \tau\lvert_A (A \cap L) \subseteq \tau(L) =\{0\}$. This shows that $\alpha$ does not have the intersection property.
Now we show the claim.

By the definition of $L$ it follows that it is enough to show that $\tau$ maps elements of the form $b_t \delta_t (a_g \delta_e - a_g \delta_g) c_k \delta_k$ to zero,
where $t,k\in G$ and $b_t\in D_t$, $c_k\in D_k$ satisfy $s(t)=e=r(k)$.

Let $u \in D_e$ be an $s$-unit for $\{c_k,\alpha_{g^{-1}}(c_k a_g)\}$.
Using that $a_g \delta_g \in C_S(A)$, and that $R$ is commutative, we get that
\begin{align*}
b_t \delta_t (a_g \delta_e - a_g \delta_g) c_k \delta_k
&=
b_t \delta_t (a_g c_k \delta_k - a_g  \delta_g c_k \delta_e u \delta_k)
=
b_t \delta_t (a_g c_k \delta_k - c_k \delta_e a_g  \delta_g  u \delta_k) \\
&=
b_t \delta_t (a_g c_k \delta_k - c_k a_g  \delta_g  u \delta_k)
=
b_t \delta_t (a_g c_k \delta_k - \alpha_g(\alpha_{g^{-1}}(c_k a_g) u)  \delta_{gk}) \\
&=
b_t \delta_t (a_g c_k \delta_k - a_g c_k  \delta_{gk})
=
\alpha_t(\alpha_{t^{-1}}(b_t) a_g c_k) \delta_{tk} - 
\alpha_t(\alpha_{t^{-1}}(b_t) 
a_g c_k) \delta_{tgk}.
\end{align*}
Clearly, 
$\tau(b_t \delta_t (a_g \delta_e - a_g \delta_g) c_k \delta_k)=0$.
This proves the claim.
\end{proof}

\begin{corolario}\label{P1-cor::max-com-int-prop}
Let $J$ be a $G$-invariant ideal of $R.$ 
Suppose that $R$ is commutative.
The following statements are equivalent:
\begin{enumerate}[{\rm (i)}]
        \item $\bigoplus_{e \in G_0} \overline{D}_e\delta_e$ is a maximal commutative subring of $R/J  \rtimes_{\overline{\alpha}}G;$
    \item $\overline{\alpha}$ has the intersection property. 
\end{enumerate} 
\end{corolario}

\begin{proof}
Notice that $R/J$ is commutative. The proof follows from Proposition~\ref{P1-prop:max-com-int-prop}. 
\end{proof}

We finish this section by summarizing our results in the case of partial actions on commutative rings. 

\begin{theorem}\label{P1-thm:AE1}
Let $R$ be a ring, let $G$ be a groupoid and let $\alpha=(\alpha_g, D_g)_{g \in G}$ be a partial action of $G$ on $R$, such that $D_g$ is $s$-unital for every $g \in G$ and $R =\bigoplus_{e \in G_0}D_e.$ If $R$ is commutative, then the following statements are equivalent:
\begin{enumerate}[{\rm (i)}]
    \item $\bigoplus_{e \in G_0} \overline{D}_e\delta_e$ is a maximal commutative subring of $R/J  \rtimes_{\overline{\alpha}}G$ for every $G$-invariant ideal $J$ of $R$; 
    \item $\alpha$ has the residual intersection property;
    \item Every ideal of $R \rtimes_{\alpha} G$ is $G$-graded.
\end{enumerate}
\end{theorem}

\begin{proof}
  The proof follows from Proposition~\ref{P1-prop:t2} and Corollary~\ref{P1-cor::max-com-int-prop}.
\end{proof}

\subsection{Simplicity and primeness of $R\rtimes_{\alpha} G$}

In this section, we study  simplicity and primeness of partial skew groupoid rings.
To that end, we recall the following definitions.

\begin{definition}
Let $R\rtimes_\alpha G$ be a partial skew groupoid ring.
\begin{enumerate}[{\rm (a)}]

 \item $R$ is called \emph{$G$-prime} if $IJ=\{0\}$, for $G$-invariant ideals $I,J$ of $R$, implies $I=\{0\}$ or $J=\{0\}$.
    
    \item $R$ is called \emph{$G$-simple} if the only $G$-invariant ideals of $R$ are $\{0\}$ and $R$. 

    \item $R\rtimes_\alpha G$ is said to be \emph{graded prime}
    if there are no nonzero graded ideals $I,J$ of $R\rtimes_\alpha G$ such that $IJ=\{0\}$.
    
    \item $R\rtimes_\alpha G$ is said to be \emph{graded simple}
    if there is no graded ideal of $R\rtimes_\alpha G$ other than
    $\{0\}$ and $R\rtimes_\alpha G$.
\end{enumerate}
\end{definition}

\begin{proposition}\label{P1-prop:GradedsSimpleGSimple}
$R$ is $G$-simple if, and only if, $R\rtimes_{\alpha}G$ is graded simple.
\end{proposition}

\begin{proof} 
We first show the ``only if'' statement by contrapositivity.
Suppose that $R\rtimes_{\alpha}G$ is not graded simple. There is a nonzero proper graded ideal $I$ of $R\rtimes_{\alpha}G.$ Write $J:=P_0(I)$ and $J_0:=P_0(I \cap \bigoplus_{e \in G_0}D_e\delta_e).$
By Proposition~\ref{P1-prop:grad} (ii) and Proposition~\ref{P1-prop:psj} (iii), $J\rtimes_{\alpha^{J}}G=J_0\rtimes_{\alpha^{J_0}}G \subseteq I \subsetneq R\rtimes_{\alpha}G$. Hence, $J=P_0(I)$ is a proper nonzero $G$-invariant ideal of $R$. Thus, $R$ is not $G$-simple.

Now we show the ``if'' statement.
Suppose that $R \rtimes_{\alpha} G$ is graded simple.
Let $J$ be a nonzero $G$-invariant ideal of $R.$
By Proposition~\ref{P1-prop:grad} (i), $J \rtimes_{\alpha^J} G$ is a nonzero graded ideal of $R \rtimes_{\alpha} G$.
Hence, $R \rtimes_{\alpha} G = J \rtimes_{\alpha^J} G$,
and, by Theorem~\ref{P1-thm:corresp1}, $J=P_0(J \rtimes_{\alpha^J} G)=P_0(R \rtimes_{\alpha} G)=R.$
Thus, $R$ is $G$-simple.
\end{proof}

Note that a result similar to Proposition~\ref{P1-prop:GradedsSimpleGSimple} appears in \cite[Proposition~5.6]{NystedtSystems2020}.

\begin{proposition}\label{P1-prop:GradedPrimeGPrime}
$R$ is $G$-prime if, and only if, $R\rtimes_{\alpha}G$ is graded prime. 
\end{proposition}

\begin{proof}
We first show the ``only if'' statement by contrapositivity. Suppose that $R\rtimes_{\alpha}G$ is not graded prime.
There are nonzero graded ideals $I_1, I_2$ of $R\rtimes_{\alpha}G$ such that $I_1I_2=\{0\}$.
By Lemma~\ref{P1-lem:inv2},
 $P_0(I_1 \cap \bigoplus_{e \in G_0}D_e\delta_e)$ and $P_0(I_2 \cap \bigoplus_{e \in G_0}D_e\delta_e)$ are nonzero $G$-invariant ideals of $R.$ 
We conclude that $R$ is not $G$-prime, because Lemma~\ref{P1-lem:some-obs} (iii) yields
$$P_0(I_1 \cap \bigoplus_{e \in G_0}D_e\delta_e)P_0(I_2 \cap \bigoplus_{e \in G_0}D_e\delta_e)= P_0((I_1 \cap \bigoplus_{e \in G_0}D_e\delta_e)(I_2 \cap \bigoplus_{e \in G_0}D_e\delta_e))\subseteq P_0(I_1I_2) = \{0\}.$$

Now we show the ``if'' statement.
Suppose that $R\rtimes_\alpha G$ is graded prime.
Let $J_1$ and $J_2$ be $G$-invariant ideals of $R$ such that $J_1 J_2 =\{0\}$. By Proposition~\ref{P1-prop:grad} (i), $J_1\rtimes_{\alpha^{J_1}}G$ and $J_2\rtimes_{\alpha^{J_2}}G$ are graded ideals of $R \rtimes_{\alpha}G$.
Moreover, using that $J_1J_2=\{0\}$ and that $J_1$ is $G$-invariant, it is clear that
$(J_1\rtimes_{\alpha^{J_1}}G) \cdot (J_2\rtimes_{\alpha^{J_2}}G) =\{0\}$. 
Hence, by assumption, $J_1\rtimes_{\alpha^{J_1}}G =\{0\}$ or $J_2\rtimes_{\alpha^{J_2}}G =\{0\}$, i.e. $J_1=\{0\}$ or $J_2=\{0\}$.
This shows that $R$ is $G$-prime.
\end{proof}

We point out that part (i) of the theorem below was proved in \cite[Corollary~5.8]{Lannstrom2021} for non-degenerately group graded rings, and part (iii) generalizes \cite[Theorem~2.3]{Goncalves2014}.

\begin{theorem}\label{P1-thm:AP2}
Let $R$ be a ring, let $G$ be a groupoid, and let $\alpha=(\alpha_g, D_g)_{g \in G}$ be a partial action of $G$ on $R$, such that $D_g$ is $s$-unital for every $g \in G$ and $R =\bigoplus_{e \in G_0}D_e.$
The following assertions hold:
\begin{enumerate}[{\rm (i)}]
    \item If $R \rtimes_{\alpha} G$ is prime, then $R$ is $G$-prime.
    \item Suppose that $\alpha$ has the intersection property. Then $R$ is $G$-prime if, and only if, $R \rtimes_{\alpha} G$ is prime.
    \item $R$ is $G$-simple and $\alpha$ has the intersection property if, and only if,  $R \rtimes_{\alpha} G$ is simple.
\end{enumerate}
\end{theorem}

\begin{proof}
(i) This follows from Proposition~\ref{P1-prop:GradedPrimeGPrime}.

\noindent (ii) 
The ``if'' statement follows from (i).
We show the ``only if'' statement by contrapositivity.
Suppose that $R \rtimes_\alpha G$ is not prime.
Let $I,J$ be nonzero ideals of $R \rtimes_{\alpha} G$ such that $IJ=\{0\}$.
Since $\alpha$ has the intersection property, 
we have that  $I_0:=P_0 \! \left(I\cap \bigoplus_{e \in G_0} D_e\delta_e\right)\neq \{0\}$ and $J_0:=P_0 \! \left(J\cap \bigoplus_{e \in G_0} D_e\delta_e\right)\neq \{0\}.$ 
By Proposition~\ref{P1-prop:grad} (i),
$I_0 \rtimes_{\alpha^{I_0}} G$
and
$J_0 \rtimes_{\alpha^{J_0}} G$
are nonzero graded ideals of $R \rtimes_{\alpha} G$.
Moreover, by Proposition~\ref{P1-prop:psj} (iii), 
$(I_0 \rtimes_{\alpha^{I_0}} G) \cdot (J_0 \rtimes_{\alpha^{J_0}} G) \subseteq IJ =\{0\}$.
This shows that $R \rtimes_\alpha G$ is not graded prime, and hence, by Proposition~\ref{P1-prop:GradedPrimeGPrime}, $R$ is not $G$-prime.\smallbreak

\noindent (iii)
We first show the ``only if'' statement.
Suppose that $R$ is $G$-simple and that $\alpha$ has the intersection property.
Let $I$ be a nonzero ideal of $R \rtimes_{\alpha} G.$ 
By Lemma~\ref{P1-lem:inv2}, $P_0 \! \left(I \cap \bigoplus_{e \in G_0} D_e\delta_e\right)$ is a nonzero $G$-invariant ideal of $R$.
Thus,
$P_0 \! \left(I \cap \bigoplus_{e \in G_0} D_e\delta_e\right) = R$.
By Lemma~\ref{P1-lem:some-obs}, we get that $I \cap \bigoplus_{e \in G_0} D_e\delta_e = \bigoplus_{e \in G_0} D_e\delta_e.$
Take $t\in G$ and $b_t \delta_t \in R \rtimes_{\alpha} G$. Let $u\in D_t$ be an $s$-unit for $b_t$.
Using that $u\delta_{r(t)} \in \bigoplus_{e \in G_0} D_e\delta_e \subseteq I$, it follows that
$b_t \delta_t= (u\delta_{r(t)})(b_t \delta_t) \in I$. Thus, $ R \rtimes_{\alpha} G = I$ and hence $R \rtimes_{\alpha} G$ is simple. 

We now show the ``if'' statement.
Suppose that $R \rtimes_{\alpha} G$ is simple. 
Clearly, $\alpha$ has the intersection property and, by Proposition~\ref{P1-prop:GradedsSimpleGSimple}, $R$ is $G$-simple.
\end{proof}

\begin{obs}
By \cite[Theorem~8]{Connell1963}, the 
group ring of a group $H$ over a ring $R,$ denoted by $R[H],$ is prime if, and only if, $R$ is prime and $H$ has no nontrivial finite normal subgroup. In particular, notice that the converse of Theorem~\ref{P1-thm:AP2} (i) does not hold.
\end{obs}

\begin{corolario}\label{P1-corprime}
Suppose that $R$ is commutative and that
$\bigoplus_{e \in G_0} D_e\delta_e$ is a maximal commutative subring of $R \rtimes_{\alpha} G.$
Then $R$ is $G$-prime if, and only if, $R \rtimes_{\alpha} G$ is prime.
\end{corolario}

\begin{proof}
It follows from
Proposition~\ref{P1-prop:max-com-int-prop} 
and Theorem~\ref{P1-thm:AP2} (ii).
\end{proof}

We point out that the following result also appears in \cite[Theorem~7.7]{NystedtSystems2020}.

\begin{corolario}
Suppose that $R$ is commutative. Then $R$ is $G$-simple and $\bigoplus_{e \in G_0} D_e\delta_e$ is a maximal commutative subring of $R \rtimes_{\alpha} G$ if, and only if, $R \rtimes_{\alpha} G$ is simple.
\end{corolario}

\begin{proof}
It follows from Theorem~\ref{P1-thm:AP2} (iii) and Proposition~\ref{P1-prop:max-com-int-prop}. 
\end{proof}

In Example~\ref{P1-ex:expnpi}, we will see that neither maximal commutativity 
of
$\bigoplus_{e \in G_0} D_e\delta_e$ in $R \rtimes_{\alpha} G$
nor the intersection property of $\alpha$ are necessary conditions for primeness
of $R \rtimes_{\alpha} G$.

\begin{example}\label{P1-ex:expnpi}
Let $q \in \mathbb{C}$  be a root of unity, that is $q^k=1$ for some $k \in \Z.$
For each $n \in \Z$, define the ring isomorphism $\sigma_n:\mathbb{C}[x] \to  \mathbb{C}[x]$, by $\sigma_n(f)(x):= f(q^nx)$. Notice that $\sigma_{m+n}=\sigma_m\circ \sigma_n,$ for all $m,n \in \Z$. Hence $\sigma$ defines a global action of the additive group $\Z$ on the complex polynomial ring $\mathbb{C}[x]$. Consider the corresponding skew group ring $\mathbb{C}[x]\rtimes_{\sigma}\Z.$
By \cite[Theorem~13.5]{Lannstrom2021}, 
$\mathbb{C}[x]\rtimes_{\sigma}\Z$ is prime. Nevertheless, $\mathbb{C}[x]\delta_0$ is not maximal commutative in $\mathbb{C}[x]\rtimes_{\sigma}\Z$. In particular, $\sigma$ does not have the intersection property.
\end{example}

\section{Partial actions of groupoids on torsion-free commutative algebras generated by idempotents}\label{P1-sec:gen-idemp}

This section is inspired by \cite[Examples~7.8--7.9]{Boava2021} and allows us to study almost any algebraic groupoid partial action as a topological one.
Let $T$ be a commutative unital ring. The set of idempotents of $T,$ denoted by $E(T)$, is a Boolean algebra with operations given by
$ e \vee f   :=e+f -ef,$ $e  \wedge f  :=ef $ and $\neg e :=1-e,$ for all $e,f\in E(T)$.
Recall that if $\mathcal{B}$ is a Boolean algebra and $P$ is a subset of $\mathcal{B},$ then $\uparrow P:=\bigcup_{z \in P} \uparrow z = \{y \in \mathcal{B}: z=z\wedge y \text{ for some } z \in P\}.$  

Throughout this section, 
let $G$ be a groupoid,
let $\mathcal{R}$ be a commutative unital ring and suppose that $R$ is a commutative, torsion-free, unital, $\mathcal{R}$-algebra generated by its idempotents $E(R).$ By torsion-free, we mean that if $a$ is a nonzero element of $\mathcal{R}$ and $e$ is a nonzero element of $E(R),$ then $ae \neq 0.$ \smallbreak

As we will see in this section, every partial action of a groupoid $G$ on $\mathcal{R}$ corresponds to a partial action of $G$ on 
the commutative $\mathcal{R}$-algebra $\mathcal{L}_c(X(R),\mathcal{R})$ of all locally constant functions $f:X(R) \to  \mathcal{R}$ with compact support, where $X(R)$ is the Stone dual space of the Boolean algebra $E(R)$. Before we define the partial action on $\mathcal{L}_c(X(R),\mathcal{R})$, we need to develop a few concepts.  \smallbreak

Recall that $X(R)$ is the Hausdorff  space comprised of all ultrafilters on $E(R).$ The basis of the topology on $X(R)$ is given by the compact open subsets 
$Z_e:=\{\mathcal{F} \in X(R): e \in \mathcal{F} \},$ for all $e \in E(R).$ By \cite[Theorem~1]{Keimel1970}, $R\simeq \mathcal{L}_c(X(R),\mathcal{R})$ as $\mathcal{R}$-algebras, via an isomorphism that identifies $e$ with $1_{Z_e}$, for all $e \in E(R).$ As usual, addition and multiplication in $\mathcal{L}_c(X(R),\mathcal{R})$ are defined pointwise.

Next, we will see how a homomorphism between algebras induces a continuous function between the Stone duals of their respective set of idempotents. 
Let $R$ and $S$ be commutative, unital, torsion-free, $\mathcal{R}$-algebras generated by their respective set of idempotents and $\rho: R \to S$ a unital $\mathcal{R}$-algebra homomorphism. Then, the restriction of $\rho$ to $E(R),$ $\rho\lvert_{E(R)}: E(R)\to E(S)$, is a Boolean algebra homomorphism. If $X(R)$ and $X(S)$ are the Stone duals of $E(R)$ and $E(S)$, respectively, then we define $\hat{\rho}:  X(S)\to  X(R)$ by
\begin{equation}\label{P1-fun-stone}
    \hat{\rho}(\mathcal{F}):=\rho^{-1}(\mathcal{F}),
    \quad
    \text{for all }\mathcal{F}\in X(S).
\end{equation}
In a Boolean algebra the notions of prime filter and ultrafilter agree, and hence it is easy to prove that the inverse image of an ultrafilter by a Boolean algebra homomorphism is again an ultrafilter. We will now show that $\hat{\rho}$ is continuous. Let $\{\mathcal{F}_{\lambda}\}_{\lambda \in \Lambda}$ be a net converging to $\mathcal{F} \in X(S)$. Suppose that $x \in E(R)$ is such that $x \in \hat{\rho}(\mathcal{F})=\rho^{-1}(\mathcal{F}).$ Then, $\rho(x) \in \mathcal{F}.$ Since the net $\{\mathcal{F}_{\lambda}\}_{\lambda \in \Lambda}$ converges to $\mathcal{F} \in X(S),$ there exists $\lambda_0 \in \Lambda$ such that for all $\lambda \geq \lambda_0,$ $\rho(x) \in \mathcal{F}_{\lambda}.$ Thus, $x \in \rho^{-1}(\mathcal{F}_{\lambda})=\hat{\rho}(\mathcal{F}_{\lambda})$ for all $\lambda \geq \lambda_0$. The case where $x \in E(R)$ is such that $x \notin \hat{\rho}(\mathcal{F})$ is proved similary. We conclude that $\hat{\rho}$ is continuous.

The final ingredient that we need in order to define a partial action on $\mathcal{L}_c(X(R),\mathcal{R})$ is to extend the aforementioned isomorphism $R \cong \mathcal{L}_c(X(R),\mathcal{R})$ to ideals of $R$. Let $I$ be a unital ideal of $R$ and let $u$ be the multiplicative identity element of $I.$
Then, $u \in E(R)$ and 
we get that $I \cong \mathcal{L}_c (Z_u, \mathcal{R})$
via the above isomorphism.

Let $\tau=(D_g,\tau_g)_{g \in G}$ be a partial action of $G$ on $R,$ and suppose that $D_g$ is unital with a multiplicative identity element $u_g$, for every $g \in G$, and that $R=\bigoplus_{e \in G_0}D_e.$ By Lemma~\ref{P1-di}, $D_g$ is an ideal of $R,$ for every $g \in G.$ Therefore, $D_g \cong \mathcal{L}_c(Z_{u_g},\mathcal{R}).$

We will define a topological partial action of $G$ on $X(R).$ Notice that, for all $g \in G,$ the restriction $\tau_g\lvert_{E(D_{g^{-1}})}: E(D_{g^{-1}}) \to E(D_g)$ is a Boolean algebra homomorphism. Therefore, the corresponding map $\hat{\tau}_g: X(D_g)\to X(D_{g^{-1}})$, defined by \eqref{P1-fun-stone}, is continuous.

For each $g \in G,$ there is a homeomorphism between $X(D_g),$ the Stone dual of $E(D_g),$ and $Z_{u_g},$ the set of ultrafilters on $E(R)$ containing $u_g.$ This correspondence is given by 
\begin{align*}
&\zeta_g: X(D_g)\to  Z_{u_g},\,\,\, \mathcal{G}\mapsto \uparrow \mathcal{G}, &  &\zeta_g^{-1}: Z_{u_g}\to X(D_g),\,\,\,  \mathcal{F}\mapsto \{ u_g x : x \in \mathcal{F} \}.&
\end{align*}
Define $\theta:=(U_g,\theta_g)_{g \in G}$, where, for all $g\in G,$ $U_g:=Z_{u_g}$ and
\begin{align} \label{P1-map-theta}
  &\theta_g:  U_{g^{-1}} \to U_g,\quad \mathcal{F}\,\mapsto\,\, \uparrow_{E(R)} \hat{\tau}_{g^{-1}}(\{ u_{g^{-1}} x : x \in \mathcal{F}\}).&
\end{align}
Observe that $\theta_g= \zeta_g \circ \hat{\tau}_{g^{-1}}\circ \zeta^{-1}_{g^{-1}} $ and we have the following commutative diagram: 

\[
 \xymatrix{& X(D_{g^{-1}})\ar[r]^{\hat{\tau}_{g^{-1}}}\ar[d]_{\zeta_{g^{-1}}}  &  X(D_g)\ar[d]_{\zeta_g} \\
&U_{g^{-1}}\ar[r]_{\theta_g} &U_g}
\]

\begin{lema}\label{P1-lemma:PartialActionOnStoneDual}
$\theta=(U_g,\theta_g)_{g \in G}$, as defined in \eqref{P1-map-theta}, is a topological partial action of $G$ on $X(R).$
\end{lema}

\begin{proof}
Notice the following:
\begin{enumerate}[{\rm (a)}]

\item For all $g \in G$, $U_g$ is an element in the basis of the topology on  $X(R).$ 
We claim that $U_g \subseteq U_{r(g)}$.
Take 
$\mathcal{F} \in U_g$ and observe that $u_g \in \mathcal{F}.$
From $D_g \subseteq D_{r(g)},$ we get that $u_g = u_{r(g)}\wedge u_g \leqslant u_{r(g)}.$ Hence, $u_{r(g)} \in \mathcal{F}$ and $\mathcal{F} \in U_{r(g)}.$
\item For all $g\in G$,
$\hat{\tau}_g$ and $\zeta_g$ are homeomorphisms,
and hence $\theta_g : U_{g^{-1}} \rightarrow U_g$ is also a homeomorphism  that satisfies:
\begin{enumerate}[{\rm (i)}]
\item $\theta_e=\id_{U_e}$, for all $e \in G_0,$ because $\tau_e=\id_{D_e};$

\item $\theta_{h^{-1}}(U_{g^{-1}}\cap U_h)\subseteq U_{(gh)^{-1}}$, whenever $(g,h) \in G^2$.
Indeed, if $\mathcal{F} \in U_{g^{-1}}\cap U_h,$ then $u_{g^{-1}} u_h=u_{g^{-1}}\wedge u_h \in \mathcal{F}.$ Since $\tau_{h^{-1}}(D_{g^{-1}}\cap D_h) \subseteq D_{(gh)^{-1}},$ we have that $\tau_{h^{-1}}(u_{g^{-1}} u_h)=\tau_{h^{-1}}(u_{g^{-1}} u_h)u_{(gh)^{-1}} \leqslant u_{(gh)^{-1}}.$ Therefore, $u_{(gh)^{-1}} \in \theta_{h^{-1}}(\mathcal{F})$ and $\theta_{h^{-1}}(U_{g^{-1}}\cap U_h) \subseteq U_{(gh)^{-1}};$ 

\item 
$\theta_g(\theta_h(x)) = \theta_{gh}(x)$, for all $x \in \theta_{h^{-1}}(U_{g^{-1}}\cap U_h)$ and $(g,h) \in G^2.$ 
Indeed, let $\mathcal{F} \in \theta_{h^{-1}}(U_{g^{-1}}\cap U_h)$ and $z \in \theta_{gh}(\mathcal{F}).$ Then, there exists $x \in \mathcal{F}$ such that $$z \geqslant \tau_{gh}(xu_{h^{-1}g^{-1}}) \geqslant \tau_{gh}(xu_{h^{-1}g^{-1}})u_g = \tau_{g}(\tau_h(x u_{h^{-1}})u_{g^{-1}}).$$  Therefore,  $z \in \uparrow_{E(R)} \tau_g(\{ u_{g^{-1}} y : y \in \uparrow_{E(R)} \tau_{h}\{ xu_{h^{-1}}: x \in \mathcal{F}\})=\theta_g(\theta_h(\mathcal{F})).$
\qedhere
\end{enumerate} 
\end{enumerate}
\end{proof}

Let $\theta=(U_g,\theta_g)_{g \in G}$ be the partial action of the groupoid $G$ on $X(R)$ defined above. For each $g \in G,$ we consider the ring isomorphism
\begin{displaymath}
\rho_g:\mathcal{L}_c(U_{g^{-1}},\mathcal{R})  \longrightarrow  \mathcal{L}_c(U_g,\mathcal{R}),
\quad
f  \longmapsto  f \circ \theta_{g^{-1}}.
\end{displaymath}
Then, $\rho=(\mathcal{L}_c(U_g,\mathcal{R}),\rho_g)_{g \in G}$ is a partial action of $G$ on $\mathcal{L}_c(X(R),\mathcal{R}).$

\begin{proposition}\label{P1-prop:Gequivariant}
The partial actions $\tau=(D_g,\tau_g)_{g \in G}$ and $\rho=(\mathcal{L}_c(U_g,\mathcal{R}),\rho_g)_{g \in G}$, as defined above, are $G$-equivariant.
\end{proposition}

\begin{proof}
Take $g\in G$.
Recall that 
there is an isomorphism $\phi_g:\mathcal{L}_c(U_g,\mathcal{R})\to D_g$  given by $\phi_g(1_{U_{eu_g}}):=eu_g,$ for every $e \in E(R).$ We claim that the following diagram commutes:

\[
 \xymatrix{& \mathcal{L}_c(U_{g^{-1}},\mathcal{R})\ar[r]^{\rho_g}\ar[d]_{\phi_{g^{-1}}}  & \mathcal{L}_c(U_{g},\mathcal{R})\ar[d]_{\phi_g} \\
&D_{g^{-1}}\ar[r]_{{\tau}_{g}} &D_g}
\]
Let $e \in E(R)$ and $\mathcal{F} \in U_g.$ Notice that $$\phi_g^{-1}\circ \tau_g\circ \phi_{g^{-1}}(1_{U_{eu_{g^{-1}}}})(\mathcal{F})=\phi_g^{-1}(\tau_g(eu_{g^{-1}}))(\mathcal{F}) =1_{U_{\tau_g(eu_{g^{-1}})}}(\mathcal{F})$$
and $\rho_g(1_{U_{eu_{g^{-1}}}})(\mathcal{F})=1_{U_{eu_{g^{-1}}}}\circ \theta_{g^{-1}} (\mathcal{F})=1_{U_{eu_{g^{-1}}}}(\uparrow_{E(R)}\tau_{g^{-1}}(\{ u_g x : x \in \mathcal{F}\}).$ We claim that $1_{U_{\tau_g(eu_{g^{-1}})}}(\mathcal{F})=1_{U_{eu_{g^{-1}}}}(\uparrow_{E(R)}\tau_{g^{-1}}(\{ u_g x : x \in \mathcal{F}\})$. Indeed, if $1_{U_{\tau_g(eu_{g^{-1}})}}(\mathcal{F})=1_{\mathcal{R}},$ then $\tau_g(eu_{g^{-1}}) \in \mathcal{F}$. Hence, $eu_{g^{-1}}=\tau_{g^{-1}}(\tau_g(eu_{g^{-1}})u_g) \in \uparrow_{E(R)}\tau_{g^{-1}}(\{ u_g x : x \in \mathcal{F}\}$. Therefore, $1_{U_{eu_{g^{-1}}}}(\uparrow_{E(R)}\tau_{g^{-1}}(\{ u_g x : x \in \mathcal{F}\})=1_{\mathcal{R}}.$

On the other hand, if $1_{U_{eu_{g^{-1}}}}(\uparrow_{E(R)}\tau_{g^{-1}}(\{ u_g x : x \in \mathcal{F}\})=1_{\mathcal{R}},$ then there exists $x \in \mathcal{F}$ such that $eu_{g^{-1}}\geqslant \tau_{g^{-1}}( u_g x).$ Since $u_gx \in \mathcal{F},$ we have that $eu_{g^{^{-1}}} \in \tau_{g^{-1}}(\mathcal{F}).$ Thus, $\tau_g(eu_{g^{^{-1}}}) \in \mathcal{F}$ and $1_{U_{\tau_g(eu_{g^{-1}})}} (\mathcal{F})=1_{\mathcal{R}}.$
\end{proof}

\section{Topological partial actions of groupoids, their partial skew groupoid rings, and Steinberg algebras}\label{P1-Sec:TopDynamics}

In Section~\ref{P1-sec:gen-idemp}, we saw that any partial action of a groupoid on the ring of locally constant functions with compact support over a Stone space, actually corresponds to a partial action
$\tau=(D_g, \tau_g)_{g \in G}$
of the same groupoid on a torsion-free, unital, commutative algebra generated by its idempotents, such that $D_g$ is a unital ideal of the commutative algebra, 
for every $g \in G$ (see Proposition~\ref{P1-prop:Gequivariant}).
This motivates us to investigate algebraic partial actions induced by topological partial actions as follows.

Throughout this section, let $G$ be a groupoid, 
let $\mathbb{K}$ be a field,
let $\mathcal{R}$ be a commutative unital ring, let $X$ be a locally compact, Hausdorff, zero-dimensional space, and let $\theta=(X_g,\theta_g)_{g \in G}$ be a topological partial action of $G$ on $X$,
such that $X_g$ is clopen, for every $g \in G$, and $X=\bigsqcup_{e \in G_0}X_e.$ Let $\mathcal{L}_c(X,\mathcal{R})$ denote the commutative $\mathcal{R}$-algebra of all locally constant functions $f:X \to \mathcal{R}$ with compact support, and with addition and multiplication defined pointwise. 
The \emph{support of $f \in \mathcal{L}_c(X,\mathcal{R})$} is defined as the set $\Supp(f):=\{x \in X: f(x) \neq 0 \},$ 
and by \cite[Proposition~1.3.4, Example~1.5.28]{TBeuter2018} it is always clopen.

\begin{obs}\label{P1-rem:InducedAlgebraicAction}
The topological partial action $\theta=(X_g,\theta_g)_{g \in G}$ 
induces a partial action $\alpha:=(D_g,\alpha_g)_{g \in G}$ of $G$ on the ring $\mathcal{L}_c(X,\mathcal{R}).$ Indeed, for each $g \in G,$ we put $D_g:=\{ f \in \mathcal{L}_c(X,\mathcal{R}): \Supp(f) \subseteq X_g\},$ and define the ring isomorphism $\alpha_g: D_{g^{-1}} \to D_g$ by
\begin{displaymath}
\alpha_g(f)(x):= \left\{\begin{array}{ll} f\circ \theta_{g^{-1}}(x), & \text{if} \,\ x \in X_g 
\\ 0, &  \text{otherwise} ,\end{array}\right.
\end{displaymath}
for $f \in D_{g^{-1}},$
 Furthermore, using that $X= \bigsqcup_{e \in G_0}X_e$ one gets that $\mathcal{L}_c(X,\mathcal{R}) = \bigoplus_{e \in G_0} D_e.$
\end{obs}

In Section~\ref{P1-subtop}, we will show equivalences between, on the one hand topological properties of $\theta$, and on the other hand algebraic properties of the induced partial action  $\alpha$ of $G$ on $\mathcal{L}_c(X,\mathbb{K})$,
and of its associated partial skew groupoid ring $\mathcal{L}_c(X,\mathbb{K}) \rtimes_{\alpha}G$.

In Section~\ref{P1-substeinberg},
inspired by \cite{Beuter2018},
we define the transformation groupoid $G\rtimes_{\theta}X$ induced by the topological partial action $\theta$, and notice that 
$\mathcal{L}_c(X,\mathcal{R}) \rtimes_{\alpha}G$ and the Steinberg algebra of the transformation groupoid, $A_{\mathcal{R}}(G\rtimes_{\theta}X)$, are isomorphic as $\mathcal{R}$-algebras. We also point out correspondences between properties of $\theta$ and properties of $G\rtimes_{\theta}X.$

In Section~\ref{P1-subapplications}, we apply Theorems~\ref{P1-thm:AE1} and \ref{P1-thm:AP2} together with results from Section~\ref{P1-subtop}, 
to characterize primeness, simplicity etc. 
of the partial skew groupoid ring $\mathcal{L}_c(X,\mathbb{K}) \rtimes_{\alpha}G.$ 
Using that $\mathcal{L}_c(X,\mathbb{K}) \rtimes_{\alpha}G \cong A_{\mathbb{K}}(G\rtimes_{\theta}X)$, 
we
are able to recover some known results.

\subsection{Connections between topological and algebraic actions} \label{P1-subtop}

The goal of this section is to identify correspondences between topological properties of $\theta$ and algebraic properties of the induced partial action  $\alpha$ of $G$ on $\mathcal{L}_c(X,\mathbb{K}).$
We define a map
$\mathcal{I} : \{\text{Open subsets of } X\} \to \{\text{Ideals of } \mathcal{L}_c(X,\mathbb{K})\}$
by 
\begin{align}\label{P1-map-I}
\mathcal{I}(U):=\{f \in \mathcal{L}_c(X,\mathbb{K}): \Supp(f) \subseteq U \},
\end{align}
for every open subset $U$ of $X$.
We also define a map
$\mathcal{S} :  \{\text{Ideals of } \mathcal{L}_c(X,\mathbb{K})\} \to \{\text{Open subsets of } X\}$ by
\begin{align}\label{P1-map-S}
\mathcal{S}(J):=\{x \in X : \text{there is } f \in J \text{ such that } f(x) \neq 0\}= \bigcup_{f \in J} \Supp(f),
\end{align}
for every ideal $J$ of $\mathcal{L}_c(X,\mathbb{K})$.
It is not difficult to verify that $\mathcal{I}$ and $\mathcal{S}$ are both well-defined.
Furthermore, we notice that $\mathcal{I}(\mathcal{S}(J))=J$, for every ideal $J$ of $\mathcal{L}_c(X,\mathbb{K})$,
and that
$\mathcal{S}(\mathcal{I}(U))=U$, for every open subset $U$ of $X$.
For more details, see e.g. \cite[Lemma~3.2.1, Remark~3.2.2]{TBeuter2018}.

\begin{obs}\label{P1-rem:SubsetsToIdealsAndBack}
(a)
Analogously to \cite[pp.~116--117]{TBeuter2018}, if $U$ is an open $G$-invariant subset of $X$, then $\mathcal{I}(U)$ is a $G$-invariant ideal of $\mathcal{L}_c(X,\mathbb{K})$.
Conversely, if $J$ is a $G$-invariant ideal of $\mathcal{L}_c(X,\mathbb{K})$, then $\mathcal{S}(J)$ is an open $G$-invariant subset of $X$.
Thus, by restricting $\mathcal{I}$ we get a bijection between
$$\{\text{Open } G\text{-invariant subsets of } X\}
\text{ and }
\{G\text{-invariant ideals of } \mathcal{L}_c(X,\mathbb{K})\}.$$

(b) Clearly, $F \subseteq X$ is closed if, and only if, $X\setminus F$ is open.
Moreover, it is not difficult to see that $F$ is $G$-invariant if, and only if, $X \setminus F$ is $G$-invariant.
Hence, the assignment $F \mapsto X \setminus F$ yields a bijection between
$$\{\text{Closed } G\text{-invariant subsets of } X\}
\text{ and }
\{\text{Open } G\text{-invariant subsets of } X\}.$$

(c) For any closed subset $F$ of $X$, we define $\mathcal{J}(F):=\mathcal{I}(X \setminus F)$ and this yields a bijection $\mathcal{J}
 : \{\text{Closed subsets of } X\} \to \{\text{Ideals of } \mathcal{L}_c(X,\mathbb{K})\}$.
 Note that $\mathcal{J}(F) = \{f \in \mathcal{L}_c(X,\mathbb{K}) : f(x)=0, \text{ for all } x\in F \}$.
\end{obs}

We now summarize our findings.

\begin{proposition}\label{P1-prop:corresp2}
There are one-to-one correspondences between the set of $G$-invariant ideals of $\mathcal{L}_c(X,\mathbb{K})$ and any of the following:
\begin{enumerate}[{\rm (i)}]
    \item The set of open $G$-invariant subsets of $X,$  given by $\mathcal{I};$
    \item The set of closed $G$-invariant subsets of $X,$ given by $\mathcal{J};$
    \item The set of $G$-graded ideals of $\mathcal{L}_c(X,\mathbb{K}) \rtimes_{\alpha}G,$ given by $\fr_{gr}.$
\end{enumerate} 
\end{proposition}

Next, we record some topological-dynamical notions that are used in this section.

\begin{definition}
Let $\theta$ be a topological partial action of $G$ on $X$.
\begin{enumerate}[{\rm (a)}]
    \item $\theta$ is said to be \emph{minimal}, if there is no proper open $G$-invariant subset 
    of $X;$
    \item $\theta$ is said to be \emph{topologically transitive}, if for all nonempty open subsets $U, V$  of $X$ there exists $g \in G$ such that $\theta_g(U\cap X_{g^{-1}})\cap V \neq \emptyset;$
    \item $\theta$ is said to be \emph{topologically free}, if for all $u \in G_0$ and 
    $t \in G_u^u\setminus \{u\},$ $\Int\{x \in X_{t^{-1}}: \theta_t(x)=x\} = \emptyset;$ 
    \item $\theta$ is said to be \emph{topologically free on the subset $F$} of $X$,  if for all $u \in G_0$ and $t \in G_u^u\setminus\{u\},$
$\Int_F\{x \in X_{t^{-1}}\cap F: \theta_t(x)=x\}=\emptyset.$
\end{enumerate}
\end{definition}

The next result follows immediately from Proposition~\ref{P1-prop:corresp2} and it also follows from \cite[Theorem~4.6]{BeuterRoyer}.

\begin{corolario}\label{P1-cor:MS}
$\theta$ is minimal if, and only if, $\mathcal{L}_c(X,\mathbb{K})$ is $G$-simple.
\end{corolario}

\begin{theorem}\label{P1-thm:TGP}
Let $\mathbb{K}$ be a field,
let $G$ be a groupoid, 
let $X$ be a locally compact, Hausdorff, zero-dimensional space, and let $\theta=(X_g,\theta_g)_{g \in G}$ be a topological partial action of $G$ on $X$,
such that $X_g$ is clopen, for every $g \in G$, and $X=\bigsqcup_{e \in G_0}X_e.$ 
Furthermore, let $\alpha$ be the induced partial action of $G$ on
$\mathcal{L}_c(X,\mathbb{K})$ (see
Remark~\ref{P1-rem:InducedAlgebraicAction}).
Then $\theta$ is topologically transitive if, and only if, $\mathcal{L}_c(X,\mathbb{K})$ is $G$-prime.
\end{theorem}

\begin{proof}
We first show the ``only if'' statement by contrapositivity.
Suppose that $\mathcal{L}_c(X,\mathbb{K})$ is not $G$-prime.
There are nonzero $G$-invariant ideals $I_1,I_2$ of $\mathcal{L}_c(X,\mathbb{K})$ such that $I_1 I_2 = \{0\}$.
By Remark~\ref{P1-rem:SubsetsToIdealsAndBack} and the fact that $\mathbb{K}$ is a field, there are nonempty open $G$-invariant subsets $U_1,U_2$ of $X$ such that $I_1=\mathcal{I}(U_1)$,  $I_2=\mathcal{I}(U_2)$, and $U_1 \cap U_2 = \emptyset$.
Hence,
$\theta_g(U_1 \cap X_{g^{-1}})\cap U_2 \subseteq U_1 \cap U_2 = \emptyset$, for every $g\in G$.
This shows that $\theta$ is not topologically transitive.

Now we show the ``if'' statement.
Suppose that $\mathcal{L}_c(X,\mathbb{K})$ is $G$-prime.
Let $U,V$ be nonempty open subsets of $X$.
Define $W:= \bigcup_{g\in G} \theta_g(U \cap X_{g^{-1}})$ and $Z:=\bigcup_{g\in G} \theta_g(V \cap X_{g^{-1}}).$
Clearly, $W$ and $Z$ are open, and they are also $G$-invariant.
Indeed, for any $h\in G$ we have
$$\theta_h (W \cap X_{h^{-1}}) = \theta_h \left( \left(\bigcup_{g\in G} \theta_g \left(U \cap X_{g^{-1}} \right) \right) \cap X_{h^{-1}} \right) = \bigcup_{g\in G} \theta_h \left(\theta_g \left(U \cap X_{g^{-1}} \right) \cap X_{h^{-1}} \right). $$

Note that $\theta_g(U \cap X_{g^{-1}}) \cap X_{h^{-1}} \subseteq X_g \cap X_{h^{-1}}  \subseteq X_{r(g)} \cap X_{s(h)} = \emptyset $ if $s(h) \neq r(g).$ Therefore, take $g \in G$ such that $s(h)=r(g),$ and notice that
$$\begin{aligned}
\theta_h \left(\theta_g \left(U \cap X_{g^{-1}} \right) \cap X_{h^{-1}} \right)  
& \subseteq \theta_h \left(\theta_g \left(\theta_{g^{-1}} \left(X_{h^{-1}}\cap X_g \right) \cap U \right) \right) 
& = \theta_{hg} \left(U \cap \theta_{g^{-1}} \left(X_{h^{-1}}\cap X_g \right) \right) \\
& \subseteq \theta_{hg} \left( U \cap X_{g^{-1}h^{-1}} \right) \subseteq W.
\end{aligned}$$
This shows that $W$ is $G$-invariant, and similarly one can show that $Z$ is $G$-invariant. 
By Remark~\ref{P1-rem:SubsetsToIdealsAndBack},  $I_1:=\mathcal{I}(W)$ and $I_2:=\mathcal{I}(Z)$ are nonzero $G$-invariant ideals of $\mathcal{L}_c(X,\mathbb{K})$.
Hence, by $G$-primeness of $\mathcal{L}_c(X,\mathbb{K})$, we get that $I_1I_2 \neq \{0\}$ and thus $W \cap Z \neq \emptyset$.

Choose some $z \in W \cap Z.$ There are $g,h\in G$ such that $z \in \theta_g(U \cap X_{g^{-1}}) \cap \theta_h(V \cap X_{h^{-1}}).$ Since $X_g\cap X_h \subseteq X_{r(g)}\cap X_{s(h^{-1})},$ we have that $s(h^{-1})=r(g).$ By 
Definition~\ref{P1-def:TopPartAction},
$$\begin{aligned}
\theta_{h^{-1}}(z) & \in \theta_{h^{-1}}(\theta_g(U \cap X_{g^{-1}})\cap X_h) \cap V
\subseteq \theta_{h^{-1}} \left(\theta_g \left(\theta_{g^{-1}} \left(X_h \cap X_g \right) \cap U \right) \right) \cap V \\
& = \theta_{h^{-1}g} \left(U \cap \theta_{g^{-1}} \left(X_h \cap X_g \right) \right) \cap V 
\subseteq \theta_{h^{-1}g} \left(U \cap X_{g^{-1}h} \right) \cap V,
\end{aligned}$$ 
and this shows that $\theta$ is topologically transitive.
\end{proof}

Next, we characterize  topological freeness of $\theta$
in terms of properties of the associated partial skew groupoid ring.
Note that if $F\subseteq X$ is closed and $G$-invariant, then by 
Remark~\ref{P1-rem:SubsetsToIdealsAndBack}, 
$\mathcal{J}(F)$ is a $G$-invariant ideal of $\mathcal{L}_c(X,\mathbb{K})$.
Thus, $\frac{\mathcal{L}_c(X,\mathbb{K})}{\mathcal{J}(F)} \rtimes_{\overline{\alpha}}G$ is well-defined (cf.~Proposition~\ref{P1-prop:quotientSkew}).

\begin{theorem}\label{P1-thm:TF}
Let $\mathbb{K}$ be a field,
let $G$ be a groupoid, 
let $X$ be a locally compact, Hausdorff, zero-dimensional space, and let $\theta=(X_g,\theta_g)_{g \in G}$ be a topological partial action of $G$ on $X$,
such that $X_g$ is clopen, for every $g \in G$, and $X=\bigsqcup_{e \in G_0}X_e.$ 
Furthermore, let $\alpha$ be the induced partial action of $G$ on
$\mathcal{L}_c(X,\mathbb{K})$ (see
Remark~\ref{P1-rem:InducedAlgebraicAction}).
Suppose that $F \subseteq X$ is closed and $G$-invariant. Then, $\theta$ is topologically free on $F$
 if, and only if, $\bigoplus_{e \in G_0} \overline{D}_e\delta_e$ is a maximal commutative subring of $\frac{\mathcal{L}_c(X,\mathbb{K})}{\mathcal{J}(F)} \rtimes_{\overline{\alpha}}G.$
\end{theorem}

\begin{proof}
We first show the ``only if'' statement.
Suppose that $\theta$ is topologically free on $F$.
Let $a = \sum_{g \in H} (a_g +\mathcal{J}(F)) \delta_g \in \frac{\mathcal{L}_c(X,\mathbb{K})}{\mathcal{J}(F)} \rtimes_{\overline{\alpha}}G $
be such that $ad=da$, for all $d \in \bigoplus_{e \in G_0} \overline{D}_e\delta_e$, and $a_g \notin \mathcal{J}(F)$ for all $g \in H.$ 
Seeking a contradiction, suppose that there is some $t\in H\setminus G_0$.
For any compact open subset $K$ of $X_{r(t)}$, using that $a (1_K + \mathcal{J}(F))\delta_{r(t)}= (1_K + \mathcal{J}(F))\delta_{r(t)} a$, we get that
\begin{equation}\label{P1-eq:TopFreeMaxCommCalc}
\sum_{\substack{g \in H \\ s(g) = r(t)}}\overline{\alpha}_g(\overline{\alpha}_{g^{-1}}(a_g+\mathcal{J}(F))(1_K+\mathcal{J}(F)))\delta_g= \sum_{\substack{g \in H \\ r(g) =r(t)}} (1_K+\mathcal{J}(F))(a_g+\mathcal{J}(F)) \delta_g.
\end{equation}

We also notice that if
$f+\mathcal{J}(F)=g+\mathcal{J}(F),$ then $(g-f)\lvert_F =0$ and $g\lvert_F=f\lvert_F$.

\textbf{Case 1} ($s(t) \neq r(t)$):
Let $x_0 \in F \cap X_t$ be such that $a_t(x_0)\neq 0.$
Using that $X$ is locally compact, Hausdorff and zero-dimensional, there exists a compact open set $K$ such that $x_0 \in K \subseteq X_t\subseteq X_{r(t)}.$ 
By comparing the left-hand side with the right-hand side of \eqref{P1-eq:TopFreeMaxCommCalc}, since $s(t) \neq r(t)$,  we get that
$0+\mathcal{J}(F)= a_t1_K+\mathcal{J}(F).$ Thus, $ a_t1_K \in \mathcal{J}(F)$ and
$a_t1_K\lvert_F =0,$
which is a contradiction, since $x_0 \in F$ and $a_t(x_0)1_K(x_0)=a_t(x_0) \neq 0.$

\textbf{Case 2} ($s(t) = r(t)$):
Let $z \in F \cap X_t$ be such that $a_t(z)\neq 0.$ Then, there is an open set $A$ such that $z \in A\subseteq X_t $ and $a_t(A) = \{a_t(z)\}$.
We have that $A\cap F\subseteq X_t \cap F$ is open. Notice that $t\neq r(t)$. Hence, by assumption, there is $y \in A\cap F$ such that $\theta_{t^{-1}}(y) \neq y.$
Using that $X$ is Hausdorff, there exist disjoint open sets $V, W$ such that
$y \in V$ and $\theta_{t^{-1}}(y) \in W.$ There is also a compact open set $K$ such that
$y \in K \subseteq V \cap X_t\subseteq V \cap X_{r(t)}.$ 
By comparing the left-hand side with the right-hand side of \eqref{P1-eq:TopFreeMaxCommCalc}, we get that
$\alpha_t(\alpha_{t^{-1}}(a_t)1_K)+\mathcal{J}(F)= a_t1_K+\mathcal{J}(F)$ which yields
$a_t(1_K\circ \theta_{t^{-1}})+\mathcal{J}(F)= a_t1_K+\mathcal{J}(F).$ Finally,
$a_t(1_K\circ \theta_{t^{-1}}-1_K) \in \mathcal{J}(F)$ and $(a_t(1_K\circ \theta_{t^{-1}}-1_K))\lvert_F = 0.$ This is a contradiction, since $y \in A \cap F$ and $$a_t(1_K\circ \theta_{t^{-1}}-1_K)(y)= a_t(y)(1_K\circ \theta_{t^{-1}}(y)-1_K(y))=a_t(y)(0-1) = -a_t(y) = -a_t(z) \neq 0.$$

Now we show the ``if'' statement.
Suppose that $\theta$ is not topologically free on $F.$ Then, there exist $u_0 \in G_0,$ $t_0 \in G_{u_0}^{u_0} \setminus\{u_0\}$ and an open subset $V \subseteq X$ such that $V \cap F \subseteq X_{t_0^{-1}}\cap F$ and $\theta_{t_0}(x)=x,$ for all $x \in V \cap F.$
Using that $X$ is locally compact, Hausdorff and zero-dimensional, there exists a compact open subset $K$ such that $x \in K \subseteq V \subseteq X_{t_0^{-1}}.$ Therefore, $1_K \in D_{t_0^{-1}}.$

Finally, we show that $(1_K+\mathcal{J}(F))\delta_{t_0^{-1}}$ commutes with every element of $\bigoplus_{e \in G_0} \overline{D}_e\delta_e.$
Let $b= \sum_{e \in G_0}(f_e+ \mathcal{J}(F))\delta_e \in \bigoplus_{e \in G_0} \overline{D}_e\delta_e.$ 
We get that
$$\begin{aligned}
& ((1_K+\mathcal{J}(F))\delta_{t_0^{-1}})b-b(1_K+\mathcal{J}(F))\delta_{t_0^{-1}}= \\ 
& (\alpha_{t_0^{-1}}(\alpha_{t_0}(1_K)f_{r(t_0)})+\mathcal{J}(F))\delta_{t_0^{-1}r(t_0)}-(f_{s(t_0)} 1_K +\mathcal{J}(F)) \delta_{s(t_0)t_0^{-1}} = \\ &(\alpha_{t_0^{-1}}(\alpha_{t_0}(1_K)f_{u_0})+\mathcal{J}(F))\delta_{t_0^{-1}} -(f_{u_0}1_K +\mathcal{J}(F))\delta_{t_0^{-1}} = \\ 
& (\alpha_{t_0^{-1}}((1_K \circ \theta_{t_0^{-1}})f_{u_0})+\mathcal{J}(F))\delta_{t_0^{-1}} - (f_{u_0}1_K +\mathcal{J}(F))\delta_{t_0^{-1}} = \\ & (1_K((f_{u_0}\circ \theta_{t_0}) -f_{u_0}) +\mathcal{J}(F))\delta_{t_0^{-1}}. 
\end{aligned}$$

Let $x \in F.$ If $x \in V \cap F,$ then $\theta_{t_0}(x)=x$ and
$((f_{u_0}\circ \theta_{t_0}) -f_{u_0})(x)=0.$
If $x \in F \setminus V,$ then
$1_K(x)= 0,$ because $K \subseteq V.$
Hence, $(1_K+\mathcal{J}(F))\delta_{t_0^{-1}}b-b(1_K+\mathcal{J}(F))\delta_{t_0^{-1}} = 0 +\mathcal{J}(F).$
\end{proof}

By Theorem~\ref{P1-thm:AE1} and Theorem~\ref{P1-thm:TF} we get the following.

\begin{corolario}\label{P1-cor-top-residual}
$\theta$ is topologically free on every closed $G$-invariant subset of $X$ if, and only if, the induced partial action $\alpha$ of $G$ on $\mathcal{L}_c(X,\mathbb{K})$ has the residual intersection property.
\end{corolario}

\subsection{Steinberg algebras and partial skew groupoid rings}\label{P1-substeinberg}

In this section, we point out that the partial skew groupoid ring associated with $\theta$ is isomorphic to the Steinberg algebra of the induced transformation groupoid. We also identify correspondences between topological properties of $\theta$ and properties of the associated transformation groupoid.

The \emph{transformation groupoid associated with $\theta$} is defined as the set
$$G \rtimes_{\theta}X:=\{(t,x): t \in G \,\ \text{and} \,\ x \in X_{t^{-1}} \},$$
and 
$(G \rtimes_{\theta}X)^2:=\{((v,y),(t,x)): 
(v,y),(t,x) \in G \rtimes_{\theta}X, \,\
(v,t)\in G^2 \,\ \text{and} \,\ \theta_t(x)=y \}.$
The product in $G \rtimes_{\theta}X$ is given by $(v,y)(t,x):=(vt,x)$
for all
$((v,y),(t,x)) \in (G \rtimes_{\theta}X)^2$.
For $(t,x) \in G \rtimes_{\theta}X$, we define
$(t,x)^{-1}:=(t^{-1}, \theta_t(x)).$ 

We equip $G \rtimes_{\theta}X$ with the topology induced by the product topology on $G \times X.$ Since $G$ is discrete and $X$ is Hausdorff, $G \rtimes_{\theta}X$ is Hausdorff. Notice that the inversion and product operations are continuous, because $\theta_t$ is continuous for every $t \in G,$
and that
$(G \rtimes_{\theta}X)_0=\{ (e,x): e \in G_0 \,\ \text{and} \,\ x \in X_e\}.$
Using that $X= \bigsqcup_{e \in G_0}X_e,$ there is a homeomorphism 
\begin{align}\label{P1-map-rho}
    \rho:  X \ni x \longmapsto (e, x) \in (G \rtimes_{\theta} X)_0.
\end{align}
Hence, $(G \rtimes_{\theta} X)_0$ is locally compact, Hausdorff and zero-dimensional.

For each $(g,x) \in G \rtimes_{\theta}X,$ the 
source in given by $d(g,x):=(s(g),x),$ the 
range is given by $t(g,x):=(r(g),\theta_g(x))$ and its restriction $\{g\}\times X_{g^{-1}} \to \{r(g)\}\times X_g$ is a homeomorphism. In particular, $G \rtimes_{\theta}X$ is \'{e}tale. Therefore, $G \rtimes_{\theta}X$ is ample and Hausdorff.
Recall that the Steinberg algebra,
$A_{\mathcal{R}}(G \rtimes_{\theta}X),$ is the set of all locally constant functions $f:G \rtimes_{\theta}X \to \mathcal{R}$ with compact support (see \cite{Beuter2018} for more details).

\begin{theorem}\label{P1-Steinberg-iso}
Let $\mathcal{R}$ be a commutative unital ring, let $G$ be a groupoid, let $X$ be a locally compact, Hausdorff, zero-dimensional space, and let $\theta = (X_g, \theta_g)_{g \in G}$ be a partial action of $G$ on $X$, such that
$X_g$ is clopen, for every $g \in G$, and $X=\bigsqcup_{e \in G_0}X_e.$ 
Furthermore, let $\alpha$ be the induced partial action of $G$ on
$\mathcal{L}_c(X,\mathcal{R})$ (see
Remark~\ref{P1-rem:InducedAlgebraicAction}).
Then, $\mathcal{L}_c(X,\mathcal{R}) \rtimes_{\alpha}G$ and $A_{\mathcal{R}}(G\rtimes_{\theta}X)$ are isomorphic as $\mathcal{R}$-algebras.
\end{theorem}

\begin{proof}
It is analogous to the proof of \cite[Theorem~3.2]{Beuter2018}. 
\end{proof}

\begin{obs}
Consider the inverse semigroup $S=G\cup\{z\}$ (cf. Remark~\ref{P2-InverseSemigroups}) and let $\theta':=(X'_t, \theta'_t)_{t \in S}$ be a topological partial action of $S$ on $X$ such that $X'_t:=X_t$ and $\theta'_t:=\theta_t,$ for every $t \in G,$ and $X'_{z}:=\emptyset.$ Then, the transformation groupoid $G \rtimes_{\theta}X$ corresponds exactly to the groupoid of germs $S \ltimes X$ as defined in \cite[Definition~4.1.5]{TBeuter2018}, and hence Theorem~\ref{P1-Steinberg-iso} can be obtained from \cite[Theorem~4.3.4]{TBeuter2018}.
\end{obs}

In the next definition, we recall some standard notions regarding topological groupoids. Throughout the remainder of this section, let $\mathcal{G}$ be a topological groupoid, let $\mathcal{G}_0$ denote the unit space of $\mathcal{G},$ let $s$ denote the \emph{source map} and let $r$ denote the \emph{range map}.
A subset $D$ of $\mathcal{G}_0$ is \emph{invariant} if 
$r(g) \in D$ whenever $g \in \mathcal{G}$ and $s(g) \in D.$
If $D$ is an invariant subset of $\mathcal{G}_0$, then we define
$\mathcal{G}_D:=\{g \in \mathcal{G}: s(g) \in D, r(g) \in D\}.$
Notice that $\mathcal{G}_D$ is a full subgroupoid of $\mathcal{G}$ and that $(\mathcal{G}_D)_0=D.$

\begin{definition}
Let $\mathcal{G}$ be a topological groupoid and write $\Iso(\mathcal{G}):=\{g \in \mathcal{G}: s(g)=r(g)\}.$
\begin{enumerate}[{\rm (a)}]
    \item $\mathcal{G}$ is said to be \emph{effective}, if $\mathcal{G}_0=\Int(\Iso(\mathcal{G}));$
    \item $\mathcal{G}$ is said to be \emph{strongly effective}, if for every nonempty invariant subset $D$ of $\mathcal{G}_0,$ the groupoid $\mathcal{G}_D$ is effective;
    \item $\mathcal{G}$ is said to be \emph{minimal}, if $\mathcal{G}_0$ has no nontrivial open invariant subsets;
    \item An \'{e}tale groupoid $\mathcal{G}$ is said to be \emph{topologically transitive}, if
    $s^{-1}(U)\cap r^{-1}(V)\neq \emptyset$
    whenever $U,V$ are nonempty open subsets of $\mathcal{G}_0$.
\end{enumerate}
\end{definition}

 The homeomorphism $\rho$, defined in \eqref{P1-map-rho}, also appears in the context of transformation groupoids induced by partial actions of groups (see e.g. \cite[Section~2.1]{TBeuter2018}). 

\begin{proposition}\label{P1-prop:openinv}
The homeomorphism $\rho,$ defined in
\eqref{P1-map-rho},
yields a bijection between the set of open (resp. closed) $G$-invariant subsets of $X$ and the set of open (resp. closed) invariant subsets of $(G \rtimes_{\theta}X)_0$.
\end{proposition}

\begin{proof}
Let $F$ be a $G$-invariant subset of $X.$ We claim that $\rho(F)$ is an invariant subset of $(G \rtimes_{\theta}X)_0.$ Take $(g,x) \in G \rtimes_{\theta}X$ such that $d(g,x)=(s(g),x) \in \rho(F).$ Notice that $t(g,x)=(r(g),\theta_g(x)) \in \rho(F),$ 
because $\theta_g(x) \in F\cap X_g \subseteq F\cap X_{r(g)},$
and $\rho(\theta_g(x))=(r(g),\theta_g(x))$.

Let $D$ be an invariant subset of $(G \rtimes_{\theta}X)_0$.
Take $g \in G.$ 
Let $x \in \rho^{-1}(D) \cap X_{g^{-1}} \subseteq \rho^{-1}(D) \cap X_{s(g)}.$ Then, $\rho(x)=(s(g),x) \in D.$ Using that $D$ is invariant,
$\rho(\theta_g(x)) = (r(g), \theta_g(x)) \in D.$
Thus, $\theta_g(x) \in \rho^{-1}(D).$
This shows that $\theta_g(\rho^{-1}(D) \cap X_{g^{-1}})\subseteq \rho^{-1}(D).$ 
\end{proof}

\begin{proposition}\label{P1-prop:CTS}
The following assertions hold:
\begin{enumerate}[{\rm (i)}]
    \item $\theta$ is minimal if, and only if, $G \rtimes_{\theta}X$ is minimal.
    \item $\theta$ is topologically transitive if, and only if, $G \rtimes_{\theta}X$ is topologically transitive.
    \item Let $F$ be a closed $G$-invariant subset of $X.$ Then, $\theta$ is topologically free on $F$ if, and only if, $(G \rtimes_{\theta}X)_{\rho(F)}$ is effective.
\end{enumerate}
\end{proposition}

\begin{proof} 
(i) It follows from Proposition~\ref{P1-prop:openinv}.\smallbreak

(ii)
We first show the ``if'' statement.
Suppose that $G \rtimes_{\theta}X$ is topologically transitive.
Let $U,V$ be nonempty open subsets of $X.$ Then, $\rho(U)$ and $\rho(V)$ are nonempty open subsets of $(G \rtimes_{\theta}X)_0.$ 
By assumption, there exists $(g, x) \in d^{-1}(\rho(U)) \cap t^{-1}(\rho(V)).$ Hence, $x \in X_{g^{-1}}$ and $d(g,x)=(s(g),x)= \rho(x) \in \rho(U).$ Furthermore, $t(g,x)=(r(g),\theta_g(x))=\rho(\theta_g(x)) \in \rho(V).$ Thus, $\theta_g(x) \in \theta_g(U\cap X_{g^{-1}})\cap V \neq \emptyset.$\smallbreak

Now we show the ``only if'' statement.
Suppose that $\theta$ is topologically transitive.
Let $A,B$ be nonempty open subsets of $(G \rtimes_{\theta}X)_0.$ Then, $\rho^{-1}(A)$ and $\rho^{-1}(B)$ are nonempty open subsets of $X.$
By assumption,
there exists $g \in G$ such that $\theta_g(\rho^{-1}(A)\cap X_{g^{-1}})\cap \rho^{-1}(B) \neq \emptyset.$ Let $y \in \theta_g(\rho^{-1}(A)\cap X_{g^{-1}})\cap \rho^{-1}(B).$ Then, $\theta_{g^{-1}}(y) \in \rho^{-1}(A)\cap X_{g^{-1}}$ and $d(g, \theta_{g^{-1}}(y)) =\rho(\theta_{g^{-1}}(y)) \in A.$ Furhermore, $t(g, \theta_{g^{-1}}(y)) =\rho(y) \in B.$ Hence, $(g, \theta_{g^{-1}}(y)) \in d^{-1}(A) \cap t^{-1}(B).$\smallbreak

(iii)
Suppose that $F$ is nonempty, for otherwise there is nothing to prove.

We first show the ``only if'' statement.
Suppose that $\theta$ is topologically free on $F.$ We will show that $\Int(\Iso((G \rtimes_{\theta}X)_{\rho(F)})= \rho(F).$
\\
$(\subseteq)$ Let $(g,x) \in \Int(\Iso((G \rtimes_{\theta}X)_{\rho(F)}).$ Then, there is 
an open set $U$ of $X_{g^{-1}}$
such that $(g,x) \in (\{g\} \times U)\cap (G \rtimes_{\theta}X)_{\rho(F)} \subseteq \Iso((G \rtimes_{\theta}X)_{\rho(F)}).$ Let $(g,y) \in (\{g\} \times U)\cap (G \rtimes_{\theta}X)_{\rho(F)}.$ Then,
$d(g,y)=t(g,y) \in \rho(F),$ that is, $s(g)=r(g)$ and $y=\theta_g(y) \in F.$ In particular, $x \in U \cap F \subseteq \{y \in X_{g^{-1}}\cap F: \theta_g(y)=y \}$
and $x \in \Int_F\{y \in X_{g^{-1}}\cap F:\theta_g(y)=y\}.$ Using that $\theta$ is topologically free on $F,$ we have that $g \in G_0$ and $(g, x) \in \rho(F).$
\\
$(\supseteq)$ Let $(e,x) \in \rho(F).$ Observe that $(e,x) \in (\{e\}\times \{X_e\}) \cap (G \rtimes_{\theta}X)_{\rho(F)} \subseteq \Iso((G \rtimes_{\theta}X)_{\rho(F)}).$ Therefore, $\rho(F) \subseteq \Int(\Iso((G \rtimes_{\theta}X)_{\rho(F)}))$ and we conclude that $(G \rtimes_{\theta}X)_{\rho(F)}$ is effective.\smallbreak

Now we show the ``if'' statement by contrapositivity.
Suppose that $\theta$ is not topologically free on $F.$ Then, there exist $g \in G\setminus G_0$ and some $x \in \Int_F\{y \in X_{g^{-1}}\cap F: \theta_g(y)=y\}.$ Thus, there is an open set $U \subseteq X_{g^{-1}}$ such that $x \in U \cap F \subseteq \{y \in X_{g^{-1}}\cap F: \theta_g(y)=y\}.$ Let $z \in U \cap F.$ Then, $\theta_g(z)=z,$ and $z \in (X_{r(g)}\cap X_{s(g)}) \cap F.$ Therefore, $s(g)=r(g)$ and $r(g,z)=d(g,z)= \rho(z) \in \rho(F).$ We notice that $(g,z) \in \Iso((G \rtimes_{\theta}X)_{\rho(F)}).$
Using that $F$ is $G$-invariant,
$(g,x) \in \{g\}\times(U \cap F) = (\{g\}\times U)\cap (G \rtimes_{\theta}X)_{\rho(F)} \subseteq \Iso((G \rtimes_{\theta}X)_{\rho(F)})$ and we have that $(g,x) \in \Int(\Iso((G \rtimes_{\theta}X)_{\rho(F)})),$ but $(g,x) \notin 
\rho(F),$ 
because $g \notin G_0.$ Hence, $(G \rtimes_{\theta}X)_{\rho(F)}$ is not effective.
\end{proof}

\begin{obs}
Note that by \cite[Corollary~5.5]{Steinberg2016} and \cite[Proposition~3.3]{Steinberg2019} the statements (i) and (ii) in Proposition~\ref{P1-prop:CTS} are already known for transformation groupoids induced by global actions of groups. We also point out that Corollary~\ref{P1-cor:coref} follows from \cite[Proposition~4.2.1]{TBeuter2018}.
\end{obs}
\begin{corolario}\label{P1-cor:coref}
$\theta$ is topologically free if, and only if, $G \rtimes_{\theta}X$ is effective.
\end{corolario}

\begin{corolario}\label{P1-cor:corsef}
$\theta$ is topologically free on every closed $G$-invariant subset of $X$ if, and only if, $G \rtimes_{\theta}X$ is strongly effective.
\end{corolario}

\subsection{Applications to partial skew groupoid rings induced by  topological partial actions}\label{P1-subapplications}

In this section, we let $\alpha = (D_g, \alpha_g)_{g \in G}$ be the partial action of $G$ on $\mathcal{L}_c(X,\mathbb{K})$ induced by $\theta$ (see Remark~\ref{P1-rem:InducedAlgebraicAction}). 
We aim to characterize primeness, simplicity, and when every ideal of the skew groupoid ring $\mathcal{L}_c(X,\mathbb{K}) \rtimes_{\alpha}G$ is $G$-graded using the topological properties of the partial action $\theta.$ Using that $\mathcal{L}_c(X,\mathbb{K}) \rtimes_{\alpha}G$ and the Steinberg algebra of the transformation groupoid, $A_{\mathbb{K}}(G\rtimes_{\theta}X)$, are isomorphic (see Theorem \ref{P1-Steinberg-iso}), we are able to recover some known results.

\begin{theorem}\label{P1-thm:TE1}
Let $\mathbb{K}$ be a field,
let $G$ be a groupoid, 
let $X$ be a locally compact, Hausdorff, zero-dimensional space, and let $\theta=(X_g,\theta_g)_{g \in G}$ be a topological partial action of $G$ on $X$,
such that $X_g$ is clopen, for every $g \in G$, and $X=\bigsqcup_{e \in G_0}X_e.$ 
Furthermore, let $\alpha$ be the induced partial action of $G$ on
$\mathcal{L}_c(X,\mathbb{K})$ (see
Remark~\ref{P1-rem:InducedAlgebraicAction}).
The following statements are equivalent:
\begin{enumerate}[{\rm (i)}]
    \item $\theta$ is topologically free on every open $G$-invariant subset of $X;$
    \item $\theta$ is topologically free on every closed $G$-invariant subset of $X;$ 
    \item $\bigoplus_{e \in G_0} \overline{D}_e\delta_e$ is a maximal commutative subring of $\frac{\mathcal{L}_c(X,\mathbb{K})}{\mathcal{J}(F)} \rtimes_{\overline{\alpha}}G$ for every closed $G$-invariant subset $F$ of $X;$
    \item $\alpha$ has the residual intersection property;
    \item Every ideal of $\mathcal{L}_c(X,\mathbb{K}) \rtimes_{\alpha}G$ is $G$-graded;
    \item The transformation groupoid $G \rtimes_{\theta}X$ is strongly effective.
\end{enumerate}
\end{theorem}

\begin{proof}
The equivalence between (i) and (ii) is straightforward.
The other equivalences are established through
Theorem~\ref{P1-thm:TF},
Corollary~\ref{P1-cor-top-residual},
Theorem~\ref{P1-thm:AE1}, and Corollary~\ref{P1-cor:corsef}
\end{proof}

By combining
Theorem~\ref{P1-thm:TE1},
Proposition~\ref{P1-prop:corresp2} and
Theorem~\ref{P1-thm:bijpri}, we get the following result which 
partially generalizes
\cite[Corollary~3.7]{Clark2019}.

\begin{corolario}\label{P1-cor:bijtop}
$\theta$ is topologically free on every closed $G$-invariant subset of $X$ if, and only if, there is a one-to-one correspondence
(given by
$\Gamma:=\mathcal{I}^{-1} \circ \fr$)
between the ideals of $\mathcal{L}_c(X,\mathbb{K}) \rtimes_{\alpha}G$ and the open $G$-invariant subsets of $X$.
\end{corolario}

A result similar to Theorem~\ref{P1-thm:TE2} below was proved in \cite[Theorem~1.3]{Lundstrom2012} for a groupoid dynamical system in the context of global actions. 

\begin{theorem}\label{P1-thm:TE2}
Let $\mathbb{K}$ be a field,
let $G$ be a groupoid, 
let $X$ be a locally compact, Hausdorff, zero-dimensional space, and let $\theta=(X_g,\theta_g)_{g \in G}$ be a topological partial action of $G$ on $X$,
such that $X_g$ is clopen, for every $g \in G$, and $X=\bigsqcup_{e \in G_0}X_e.$ 
Furthermore, let $\alpha$ be the induced partial action of $G$ on
$\mathcal{L}_c(X,\mathbb{K})$ (see
Remark~\ref{P1-rem:InducedAlgebraicAction}).
The following statements are equivalent:
\begin{enumerate}[{\rm (i)}]
    \item $\theta$ is topologically free;
    \item $\bigoplus_{e \in G_0} D_e\delta_e$ is a maximal commutative subring of $\mathcal{L}_c(X,\mathbb{K}) \rtimes_{\alpha}G;$ 
    \item $\alpha$ has the intersection property;
    \item The transformation groupoid $G \rtimes_{\theta}X$ is effective.
\end{enumerate}
\end{theorem}

\begin{proof}
The equivalence between (i) and (ii) follows from Theorem~\ref{P1-thm:TF} by considering $F=X.$ The equivalence between (ii) and (iii) follows from Proposition~\ref{P1-prop:max-com-int-prop}, and Corollary~\ref{P1-cor:coref} implies the equivalence between (i) and (iv).
\end{proof}

The following example shows
that the condition of being topologically free on every closed $G$-invariant subset is stronger than 
topological freeness.
In particular, the residual intersection property is a stronger condition than the intersection property.

\begin{example}[\cite{Sims2020}]\label{P1-ex:exnrip}
Let $\theta$ be an action of $\Z$ on the Alexandroff compactification $\Z\cup \{\infty\}$ such that, for all $n \in \Z,$ 
$\theta_n: \Z\cup \{\infty\} \ni x \longmapsto x+n \in \Z\cup \{\infty\}.$
Notice that $\Fix(\theta_n)=\{\infty\}$, for all $n \in \Z.$ 
Hence, $\theta$ is topologically free, because $\Int\{x \in \Z\cup \{\infty\}:\theta_n(x)=x\}=\Int\{\infty\}=\emptyset,$ for all $n \in \Z.$ Nevertheless, $\{\infty\}$ is a closed $\Z$-invariant subset of $\Z\cup \{\infty\}$ such that $\Int_{\{\infty\}}\{x \in \{\infty\}:\theta_n(x)=x\}=\Int_{\{\infty\}}\{\infty\}=\{\infty \}.$ Thus, $\theta$ in not topologically free on $\{\infty\}.$ 
\end{example}

\begin{proposition}\label{P1-prop:TopGradedsSimplePrime}
The following assertions hold:
\begin{enumerate}[{\rm(i)}]
    \item $\mathcal{L}_c(X,\mathbb{K})\rtimes_{\alpha}G$ is graded simple if, and only if, $\theta$ is minimal;
    \item $\mathcal{L}_c(X,\mathbb{K})\rtimes_{\alpha}G$ is graded prime if, and only if, $\theta$ is topologically transitive.
\end{enumerate}
\end{proposition}

\begin{proof} 
(i) It follows from Proposition~\ref{P1-prop:GradedsSimpleGSimple} and Corollary~\ref{P1-cor:MS}. 

(ii) It follows from Proposition~\ref{P1-prop:GradedPrimeGPrime} and Theorem~\ref{P1-thm:TGP}.
\end{proof}

In the next theorem, we will recover particular cases of \cite[Theorem~4.3, Theorem~4.5]{Steinberg2019} and \cite[Theorem~4.1]{Brown2014}.

\begin{theorem}\label{P1-thm:TPS}
Let $\mathbb{K}$ be a field,
let $G$ be a groupoid, 
let $X$ be a locally compact, Hausdorff, zero-dimensional space, and let $\theta=(X_g,\theta_g)_{g \in G}$ be a topological partial action of $G$ on $X$,
such that $X_g$ is clopen, for every $g \in G$, and $X=\bigsqcup_{e \in G_0}X_e.$ 
Furthermore, let $\alpha$ be the induced partial action of $G$ on
$\mathcal{L}_c(X,\mathbb{K})$ (see
Remark~\ref{P1-rem:InducedAlgebraicAction}).
The following assertions hold.
\begin{enumerate}[{\rm (i)}]
    \item If $\mathcal{L}_c(X,\mathbb{K})\rtimes_{\alpha}G$ is prime, then $\theta$ is topologically transitive.
    \item If $\theta$ is topologically free and topologically transitive, then $\mathcal{L}_c(X,\mathbb{K})\rtimes_{\alpha}G$ is prime.
    \item $\theta$ is minimal and topologically free if, and only if,  $\mathcal{L}_c(X,\mathbb{K}) \rtimes_{\alpha}G$ is simple.
\end{enumerate}
\end{theorem}

\begin{proof}
The proof follows from
Corollary~\ref{P1-cor:MS}, Theorem~\ref{P1-thm:TGP} and Theorem~\ref{P1-thm:AP2}.
\end{proof}

In \cite[Theorem~1.3]{Oinert2013},
a result similar to Theorem~\ref{P1-thm:TPS} (iii) was proved for groupoid dynamical systems. 
We finish this section by specializing to the case when $G$ is not just a groupoid, but in fact a torsion-free group.

\begin{proposition}
Let $H$ be a torsion-free group. Then, $\mathcal{L}_c(X,\mathbb{K})\rtimes_{\alpha}H$ is prime if, and only if, $\theta$ is topologically transitive.
\end{proposition}

\begin{proof}
By \cite[Theorem~13.5]{Lannstrom2021}, if $H$ is torsion-free, then $\mathcal{L}_c(X,\mathbb{K})\rtimes_{\alpha}H$ is prime if, and only if, $\mathcal{L}_c(X,\mathbb{K})$ is $H$-prime.
The desired conclusion now follows from Theorem~\ref{P1-thm:TGP}.
\end{proof}

\section{Applications to ultragraphs via labelled spaces}\label{P1-Sec:Ultragraphs}

Throughout this section, $\mathcal{G}$ denotes an ultragraph. In \cite[Definition~2.7]{Castro2021}, a normal labelled space $(\mathcal{E}_{\mathcal{G}}, \mathcal{L}_{\mathcal{G}}, \mathcal{B})$ associated with $\mathcal{G}$ was defined and, in \cite[Section~3]{Castro2020}, a topological partial action $\varphi$ of the free group generated by the labels of the edges on the tight spectrum $\textsf{T}$ of a labelled space was defined. The description of 
$\varphi$
when the labelled space is associated with an ultragraph was given in \cite[Section~4.1]{Castro2021}, in which it was also shown that the associated partial skew group ring coincides with the ultragraph Leavitt path algebra of $\mathcal G$, and with the partial skew group ring associated with 
certain 
non-topological partial actions
\cite{Goncalves2020}. 

One of the purposes of this section is to prove that 
$\varphi$ is topologically free on every closed $\mathbb{F}$-invariant subset of $\textsf{T}$ if, and only if, the ultragraph $\mathcal{G}$ satisfies Condition (K). 
By letting $\alpha$ be the algebraic partial action induced by $\varphi,$ we will show that every ideal of $\mathcal{L}_c(\textsf{T}, \mathbb{K}) \rtimes_{\alpha}\mathbb{F}$ is $\mathbb{F}$-graded if, and only if, $\mathcal{G}$ satisfies Condition (K), a result that is also new in the context of Leavitt path algebras of graphs. A version of this theorem was proved in the context of $C^*$-algebras of Boolean dynamical systems. See \cite[Theorem~6.3]{Carlsen2022} and \cite[Proposition~6.1]{Kang2022}.

\subsection{Preliminaries}
Following \cite{Boava2017}, \cite{Castro2021}, \cite{Castro2020} and, mainly \cite{Boava2021}, we repeat some basic definitions and properties concerning labelled spaces, ultragraphs via labelled spaces and the aforementioned topological partial action $\varphi$.  

\subsubsection{Labelled spaces}\label{P1-subsect.labelled.spaces}

Let $\mathcal{E}=(\mathcal{E}^0, \mathcal{E}^1, r, s)$ be a (directed) graph. The elements of $\mathcal{E}^0$ are denominated \emph{vertices} and the elements of $\mathcal{E}^1$ are denominated \emph{edges}. If $\mathcal{E}^0$ and $\mathcal{E}^1$ are countable, then we say that the graph $\mathcal{E}$ is \emph{countable}. 

A sequence of edges $\lambda_1\lambda_2\ldots \lambda_n$ such that $r(\lambda_i)=s(\lambda_{i+1})$ for all $i \in \{1, \ldots, n-1\}$ is called a \emph{path} of length $n$. We denote by $\mathcal{E}^n$ the set of paths of length $n$ and define $\mathcal{E}^*:=\cup_{n \geq 0} \mathcal{E}^n.$ An infinite sequence of edges $\lambda_1\lambda_2\ldots$ such that $r(\lambda_i)=s(\lambda_{i+1})$ for all $i \geq 1$ is called an \emph{infinite path}. We denote by $\mathcal{E}^{\infty}$ the set of infinite paths. We define vertices to be paths of length $0.$

 Let $\mathcal{E}=(\mathcal{E}^0,\mathcal{E}^1, r,s)$ be a graph and let $\mathcal{A}$ be an \emph{alphabet}, that is a nonempty set whose elements are called \emph{letters}. A \emph{labelled graph} $(\mathcal{E}, \mathcal{L})$ consists of a graph $\mathcal{E}$ and a surjective \emph{labelling map} $\mathcal{L}:\mathcal{E}^1\to \mathcal{A}.$

$\mathcal{A}^{\ast}$ denotes the set of all finite \emph{words} over $\mathcal{A}$, including the \emph{empty word} $\omega,$ and $\mathcal{A}^{\infty}$ denotes the set of all infinite words over $\mathcal{A}$. 	We consider $\mathcal{A}^{\ast}$ as a monoid whose binary operation is given by concatenation. 

We may extend the labelling map $\mathcal{L}$ to $\mathcal{L}:\mathcal{E}^n\to\mathcal{A}^{\ast}$ and $\mathcal{L}:\mathcal{E}^{\infty}\to\mathcal{A}^{\infty}.$ Denote the set of \emph{labelled paths $\alpha$ of length $|\alpha|=n$} by  $\mathcal{L}^n:=\mathcal{L}(\mathcal{E}^n)$ and the set 
of \emph{infinite labelled paths} by $\mathcal{L}^{\infty}:=\mathcal{L}(\mathcal{E}^{\infty}).$ Recall that $\omega$ is considered a labelled path with $|\omega|=0$ and we define $\mathcal{L}^{\geq 1}:=\cup_{n\geq 1}\mathcal{L}^n$, $\mathcal{L}^*:=\{\omega\}\cup\mathcal{L}^{\geq 1}$ and $\mathcal{L}^{\leqslant \infty}:=\mathcal{L}^*\cup \mathcal{L}^{\infty}.$ 

If $\alpha,\beta$ are labelled paths such that $\beta=\alpha\beta'$ for some labelled path $\beta',$ then we say that $\alpha$ is a \emph{beginning} of $\beta.$
Denote the set of all infinite words all of whose beginnings appear as finite labelled paths by
$\overline{\mathcal{L}^{\infty}}:=\{\alpha \in \mathcal{A}^{\infty} : \alpha_{1,n} \in \mathcal{L}^*, \text{ for every } n \in \mathbb{N}\}$ and define $\overline{\mathcal{L}^{\leqslant \infty}}:= \mathcal{L}^* \cup  \overline{\mathcal{L}^{\infty}}.$
We let $\mathbb{P}(X)$ denote the power set of $X$.

\begin{definition}[{\cite[Definition~2.8]{Boava2017}}] 
For $\alpha\in\mathcal{L}^*$ and $A\in \mathbb{P}(\mathcal{E}^0)$, the \emph{relative range of $\alpha$ with respect to} $A$, denoted by $r(A,\alpha)$, is the set
$$r(A,\alpha):=\{r(\lambda): \lambda \in \mathcal{E}^*, \mathcal{L}(\lambda)=\alpha, s(\lambda) \in A\},$$
if $\alpha \in \mathcal{L}^{\geq 1}$, and $r(A,\omega):=A$, if $\alpha=\omega.$
The \emph{range of $\alpha$}, denoted by $r(\alpha)$, is the set $r(\alpha):=r(\mathcal{E}^0,\alpha).$
\end{definition}
Note that $r(\omega)=\mathcal{E}^0$ and, if $\alpha\in\mathcal{L}^{\geq 1}$, then  $r(\alpha)=\{r(\lambda)\in\mathcal{E}^0\ :\ \mathcal{L}(\lambda)=\alpha\}$.

\begin{definition}[{\cite[Definition~2.9]{Boava2017}}] 
Let $(\mathcal{E}, \mathcal{L})$ be a labelled graph and  $\mathcal{B} \subseteq \mathbb{P}(\mathcal{E}^0).$ We say that $\mathcal{B}$ is \emph{closed under relative ranges}, if $r(A,\alpha) \in \mathcal{B}$, for all $A \in \mathcal{B}$ and $\alpha \in \mathcal{L}^{\geq 1}.$ If additionally $\mathcal{B}$ is closed under finite intersections and finite unions, and contains $r(\alpha)$, for every $\alpha \in \mathcal{L}^{\geq 1},$ then we say that $\mathcal{B}$ is \emph{accommodating} for $(\mathcal{E}, \mathcal{L})$, and in that case $(\mathcal{E}, \mathcal{L}, \mathcal{B})$ is called a \emph{labelled space}.
\end{definition}

A labelled space $(\mathcal{E}, \mathcal{L},\mathcal{B})$ is \emph{weakly left-resolving} if $r(A\cap B,\alpha)=r(A,\alpha)\cap r(B,\alpha)$, for all $A,B\in\mathcal{B}$ and  $\alpha\in\mathcal{L}^{\geq 1}$. We say that a weakly left-resolving labelled space $(\mathcal{E}, \mathcal{L},\mathcal{B})$ is \emph{normal}, if $\mathcal{B}$ is closed under relative complements.

For $\alpha\in\mathcal{L}^*$, define $\mathcal{B}_{\alpha}:=\mathcal{B}\cap \mathbb{P}(r(\alpha))=\{A\in\mathcal{B} : A\subseteq r(\alpha)\}.$ In the case of a normal labelled space, $\mathcal{B}_{\alpha}$ is a Boolean algebra, for each $\alpha\in\mathcal{L}^*$.

By \cite[Proposition~3.4]{Boava2017}, there is an inverse semigroup $S$, with zero element $0$, associated with the normal labelled space $(\mathcal{E}, \mathcal{L},\mathcal{B})$, whose semilattice of idempotents is 
\[E(S)=\{(\alpha, A, \alpha) \ : \ \alpha \in \mathcal{L}^* \ \mbox{and} \ A \in \mathcal{B}_{\alpha} \}\cup\{0\}.\]
 
The set of all tight filters in $E(S)$ is denoted by $\textsf{T}$, which is called the \emph{tight spectrum} (see \cite[Section 12]{Exel2008}). For each $e\in E(S)$, define
$V_e:=\{\xi \in \textsf{T} : e \in \xi\}.$ The subsets $$V_{e:e_1,\ldots,e_n}:= V_e\cap V_{e_1}^c \cap \cdots \cap V_{e_n}^c=\{\xi \in \textsf{T} : e \in \xi,e_1 \notin \xi,\ldots,e_n \notin \xi \},$$
with $\{e_1,\ldots,e_n\}$ a finite (possibly empty) subset of $E(S)$, form a basis for the topology on $\textsf{T}$ (see \cite{Lawson2012} and \cite[Corollary~3.3]{Castro2020}). If $E(S)$ is countable, then $\textsf{T}$ is second-countable and therefore metrizable. See \cite[Sections~2.4-2.5]{Boava2021}, for more details.

Let $(\mathcal{E}, \mathcal{L},\mathcal{B})$ be a weakly-left resolving labelled space. Let $\alpha\in\overline{\mathcal{L}^{\leqslant \infty}}$  and let $\{\mathcal{F}_n\}_{0\leq n\leq|\alpha|}$ (where ${0 \leq n \leq |\alpha|}$ means that $0 \leq n < \infty$ for $\alpha \in \overline{\mathcal{L}^{\infty}}$) be a family such that $\mathcal{F}_n$ is a filter in $\mathcal{B}_{\alpha_{1,n}}$ for every $n>0$, and $\mathcal{F}_0$ is either a filter in $\mathcal{B}$ or $\mathcal{F}_0=\emptyset$. The family $\{\mathcal{F}_n\}_{0\leq n\leq|\alpha|}$ is  a  \emph{complete family for} $\alpha$, if
$\mathcal{F}_n = \{A\in \mathcal{B}_{\alpha_{1,n}} \ : \ r(A,\alpha_{n+1})\in\mathcal{F}_{n+1}\}$
whenever $0\leq n < |\alpha|$.

\begin{theorem}[{\cite[Theorem~4.13]{Boava2017}}]\label{P1-corresp-filters}
	Let $(\mathcal{E}, \mathcal{L},\mathcal{B})$ be a weakly left-resolving labelled space and $S$   its associated inverse semigroup. 
	There is a bijection
	between filters in $E(S)$ and pairs $(\alpha, \{\mathcal{F}_n\}_{0\leq n\leq|\alpha|})$, where $\alpha\in\overline{\mathcal{L}^{\leqslant \infty}}$ and $\{\mathcal{F}_n\}_{0\leq n\leq|\alpha|}$ is a complete family for $\alpha$.
\end{theorem}

\begin{obs}
Theorem~\ref{P1-corresp-filters} was originally proved, but misstated, in \cite[Theorem~4.13]{Boava2017}. A comment about this later appeared in \cite[Theorem~2.4]{Castro2020}.
\end{obs}

We say that a filter $\xi$ in $E(S)$ is of \emph{finite type}, if it is associated with a pair $(\alpha, \{\mathcal{F}_n\}_{0 \leq n \leq |\alpha|})$ where $|\alpha|<\infty,$ and of \emph{infinite type} otherwise. The filter $\xi$ is denoted by $\xia$ and the filters in the complete family will be denoted by $\xi^{\alpha}_n$ (or just $\xi_n$). To be precise $\xi^{\alpha}_n=\{A\in\mathcal{B} \ : \ (\alpha_{1,n},A,\alpha_{1,n}) \in \xia\}.$

\begin{obs}[{\cite[Remark~2.6]{Castro2020}}]\label{P1-obs-2.6}
For a filter $\xi^\alpha$ in $E(S)$ and an element $(\beta,A,\beta)\in E(S)$, we have that $(\beta,A,\beta)\in\xi^{\alpha}$ if, and only if, $\beta$ is a beginning of $\alpha$ and $A\in\xi^{\alpha}_{|\beta|}$.
\end{obs}

Let $\beta \in \overline{\mathcal{L}^{\leqslant \infty}}.$ Denote by $\textsf{T}_{\beta}$ the set of all tight filters in $E(S)$ associated with the word $\beta$. For labelled paths $\alpha \in \mathcal{L}^{\geq 1}$ and $\beta \in \overline{\mathcal{L}^{\leqslant \infty}},$ consider $\textsf{T}_{(\alpha)\beta}:=\{\xi \in \textsf{T}_{\beta}: r(\alpha) \in \xi_0 \}.$

\begin{obs} Let $\alpha \in \mathcal{L}^{\geq 1}$ and $\beta \in \overline{\mathcal{L}^{\leqslant \infty}}.$ By \cite[Remark~2.10]{Castro2020}, if $\xi^{\beta} \in \textsf{T}_{(\alpha)\beta},$ then $\alpha \beta \in \overline{\mathcal{L}^{\leqslant \infty}}.$
\end{obs}

In particular, for $\alpha \in \mathcal{L}^*,$ denote the set of all filters in $E(S)$ whose associated labelled path $\beta$ can be ``glued''
 to $\alpha$ by the disjoint union: 
$$\bigsqcup_{\beta}\textsf{T}_{(\alpha)\beta}:=\bigcup\{\textsf{T}_{(\alpha)\beta} : \beta \in \overline{\mathcal{L}^{\leqslant \infty}} 
\text{ and } \alpha \beta \in \overline{\mathcal{L}^{\leqslant \infty}}\}.$$
Furthermore, denote the set of all filters in $E(S)$ whose associated word begins with $\alpha$ by the disjoint union:
$$\bigsqcup_{\beta}\textsf{T}_{\alpha\beta}:=\bigcup\{\textsf{T}_{\alpha \beta} : \beta \in \overline{\mathcal{L}^{\leqslant \infty}} 
\text{ and } \alpha \beta \in \overline{\mathcal{L}^{\leqslant \infty}}\}.$$

\begin{lema}[{\cite[Lemma~3.4]{Castro2020}}]\label{P1-lemma-T1} Let $(\mathcal{E}, \mathcal{L},\mathcal{B})$ be a weakly-left resolving labelled space.
Fix $\alpha \in \mathcal{L}^{\geq 1}.$ Then, $V_{(\alpha,r(\alpha),\alpha)} = \bigsqcup_{\beta}\textsf{T}_{\alpha\beta},$ and $V_{(\omega,r(\alpha),\omega)} = \bigsqcup_{\beta}\textsf{T}_{(\alpha)\beta}.$
\end{lema}

\subsubsection{Ultragraphs via labelled spaces}

Ultragraph $C^{\ast}$-algebras were defined in \cite{Tomforde2003}. We now recall some definitions and results about ultragraphs and labelled spaces associated with ultragraphs.

\begin{definition}[{\cite[Definition~2.1]{Goncalves2017}}]
An \emph{ultragraph} is a quadruple $\mathcal{G}=(G^0, \mathcal{G}^1, r,s)$ consisting of two countable sets $G^0, \mathcal{G}^1$, a map $s:\mathcal{G}^1 \to G^0$, and a map $r:\mathcal{G}^1 \to \mathbb{P}(G^0)\setminus \{\emptyset\}$, where $\mathbb{P}(G^0)$ is the power set of $G^0$.
\end{definition}

\begin{definition}[{\cite[Definition~2.3]{Goncalves2017},\cite[Definition~2.6]{Castro2021}}]
Let $\mathcal{G}$ be an ultragraph. Define $\mathcal{G}^0$ to be the smallest subset of $\mathbb{P}(G^0)$ that contains $\{v\}$, for all $v\in G^0$, contains $r(e)$, for all $e\in \mathcal{G}^1$, contains $\emptyset$, and is closed under finite unions and finite intersections. Elements of $\mathcal{G}^0$ are called \emph{generalized vertices}.
\end{definition}

A finite path $\alpha \in \mathcal{G}^*$ with $|\alpha|>0$ is called a \textit{loop}, if $s(\alpha)\in r(\alpha)$.  A loop $\alpha$ is said to be \emph{based at $A \in \mathcal{G}^0$},
if $s(\alpha) \in A.$ A loop $\alpha=\alpha_1\ldots\alpha_n$ is a \textit{simple loop}, if $s(\alpha_i)\neq s(\alpha_1)$ for $i \neq 1.$ An \textit{exit} for a loop $\alpha=\alpha_1\ldots\alpha_n$ is either of the following:
\begin{enumerate}
    \item an edge $e\in \mathcal{G}^1$ such that there exists an $i$ for which $s(e)\in r(\alpha_i)$,  but $e\neq  \alpha_{i+1}$,
    \item a sink $w$ such that $w \in r(\alpha_i),$ for some $i$.
\end{enumerate}

\begin{definition}[{\cite[Definition~2.6]{Castro2021}}]
The \emph{accommodating family $\mathcal{B}$ associated with $\mathcal{G}$} is the smallest family of subsets of $G^0$ that contains $\mathcal{G}^0$ and is closed under relative complements, finite unions and finite intersections.
\end{definition}

Next, we recall how ultragraphs are realized as labelled graphs.

\begin{definition}[{\cite[Definition~2.7]{Castro2021}}]
Fix an ultragraph $\mathcal{G}=(G^0, \mathcal{G}^1, r,s)$. Let $\mathcal{E}_{\mathcal{G}}=(\mathcal{E}_{\mathcal{G}}^0, \mathcal{E}_{\mathcal{G}}^1,r^{\prime},s^{\prime})$, where $\mathcal{E}_{\mathcal{G}}^0=G^0$, $\mathcal{E}_{\mathcal{G}}^1=\{(e,\upsilon):e\in\mathcal{G}^1, \upsilon \in r(e)\}$ and define  $r^{\prime}(e,\upsilon):=\upsilon$ and $s^{\prime}(e,\upsilon):=s(e)$. Set $\mathcal{A}=\mathcal{E}_{\mathcal{G}}^1$, let $\mathcal{B}$ be the accommodating family of $\mathcal{G}$, and define $\mathcal{L}_{\mathcal{G}}:\mathcal{E}_{\mathcal{G}}^1 \to \mathcal{A}$ by $\mathcal{L}_{\mathcal{G}}(e,\upsilon):=e$. Then,
$(\mathcal{E}_{\mathcal{G}}, \mathcal{L}_{\mathcal{G}}, \mathcal{B})$ is the normal labelled space associated with $\mathcal{G}$.
\end{definition}

\noindent Notice that labelled paths correspond to paths on the ultragraph and that $\mathcal{L}^{\infty}_{\mathcal{G}}=\overline{\mathcal{L}^{\infty}_{\mathcal{G}}}.$

\begin{obs}[{\cite[Remark~3.1]{Castro2021}}]
Notice that, for all $A\subseteq G^0$ and $e\in\mathcal{G}^1$, we have that $r(A,e)=r(e)$, if $s(e)\in A$, and $r(A,e)=\emptyset$ otherwise. In particular, $r(\{s(e)\},e)=r(e)$.
\end{obs}

\begin{lema}[{\cite[Lemma~3.2]{Castro2021}}]
Let $\alpha$ be a path in $\mathcal{G}$ such that $|\alpha|\geq 1$, and let $\{\mathcal{F}_n\}_{n=0}^{|\alpha|}$ be a complete family of filters for $\alpha$. If $0\leq n <|\alpha|$, then $\mathcal{F}_n=\uparrow_{\mathcal{B}_{\alpha_n}}\{s(\alpha_{n+1})\}$.
\end{lema}

\begin{lema}[{\cite[Lemma~3.4]{Castro2021}\label{P1-lema-family-inf}}] 
Let $\alpha$ be an infinite path in $\mathcal{G}$. Then $\{\mathcal{F}_n\}_{n=0}^\infty:=\{\uparrow_{\mathcal{B}_{\alpha_n}}\{s(\alpha_{n+1})\}\}_{n=0}^{\infty}$ is the only complete family of filters (in particular, ultrafilters) for $\alpha$.
\end{lema}

\begin{obs}\label{P1-obs-unique-filter}
By \cite[Proposition~3.5, Proposition~3.6]{Castro2021},
for each infinite path $\alpha$ in $\mathcal{G}$, there is a unique element $\xi^{\alpha} \in \textsf{T}$ whose associated word is $\alpha$. And, a filter of infinite type is completely described by the infinite path associated to it.
\end{obs}

\subsubsection{The topological partial action on $\textsf{T}$}
 
A partial action of the free group generated by the alphabet $\mathcal{A}$ on the tight spectrum $\textsf{T}$ of a labelled space was defined in \cite[Section~3]{Castro2020}. A description of that partial action, when the labelled space is associated with an ultragraph, was given in \cite[Section~4.1]{Castro2021}. For the convenience of the reader, we recall the characterization from \cite[Section~4.1]{Castro2021} in the following two paragraphs.

Fix a weakly-left resolving labelled space $(\mathcal{E}, \mathcal{L}, \mathcal{B})$ and let $\mathbb{F}$ be the free group generated by $\mathcal{A}$ (identifying the identity of $\mathbb{F}$ with $\omega$). Then, by \cite[Proposition~3.12]{Castro2020}, for every $t\in \mathbb{F}$ there is a compact open set $V_t\subseteq \textsf{T}$ and a homeomorphism $\varphi_t:V_{t^{-1}}\to V_t$ such that 
\begin{equation} \label{P1-eqn:labelled.top.partial.action}
\varphi=(\{V_t\}_{t\in\mathbb{F}},\{\varphi_t\}_{t\in\mathbb{F}})    
\end{equation}
is a topological partial action of $\mathbb{F}$ on $\textsf{T}$. In particular, $V_{\omega}:=\textsf{T}$ and if 
$\alpha,\beta\in \mathcal{L}^*$, then $V_\alpha:=V_{(\alpha,r(\alpha),\alpha)}, V_{\alpha^{-1}}:=V_{(\omega,r(\alpha),\omega)}$, and $V_{(\alpha\beta^{-1})^{-1}}= \varphi_{\beta^{-1}}^{-1}(V_{\alpha^{-1}})$, with $V_{(\alpha\beta^{-1})^{-1}}\neq\emptyset$ if, and only if, $r(\alpha)\cap r(\beta)\neq\emptyset$ \cite[Lemma~3.10]{Castro2020}. 

For the labelled space associated with an ultragraph, we can intuitively describe the map $\varphi_{\alpha\beta^{-1}}$, for $\alpha,\beta\in\mathcal{L}^*$, as cutting $\beta$ from the beginning of an element $\xi\in V_{\alpha^{-1}\beta}$ and then gluing $\alpha$ in front of it. For filters of finite type we just have to take into account the last filter of the corresponding complete family. In most cases, the last filter is kept fixed, unless the empty word is involved. If $\beta$ is the labelled path associated with an element $\xi\in V_{\alpha^{-1}\beta}$ then, by cutting $\beta$, we get the filter associated with the pair $(\omega,\usetr{\xi_{|\beta|}}{\mathcal{B}})$ and, by gluing $\alpha$ afterwards, we get the filter associated with the pair $(\alpha,\mathcal{F})$, where $\mathcal{F}=\{A\cap r(\alpha) : A\in \usetr{\xi_{|\beta|}}{\mathcal{B}}\}$.

Let $\mathbb{F} \rtimes_{\varphi} \textsf{T}=\{(t,\xi)\in \mathbb{F} \times \textsf{T} : \xi\in V_{t}\}$
denote the transformation groupoid associated with the partial action $\varphi$, as defined in \cite{Abadie2004}. 

The following results were proven in \cite{Castro2020} and \cite{Boava2021} for the partial action of the free group generated by the alphabet $\mathcal{A}$ on the tight spectrum $\textsf{T}$ of a normal  labelled space (not necessarily associated with an ultragraph).

\begin{obs}\label{P1-lemma-orb}
Let $t \in \mathbb{F}$ be in reduced form. By \cite[Lemma~3.11]{Castro2020}, if $t \notin \{\omega\}\cup\{\alpha:\alpha \in \mathcal{L}^{\geq 1}\}\cup\{\alpha^{-1}:\alpha \in \mathcal{L}^{\geq 1}\}\cup\{\alpha\beta^{-1}:\alpha, \beta \in \mathcal{L}^{\geq 1}, r(\alpha)\cap r(\beta) \neq \emptyset\},$ then $V_t=V_{t^{-1}}=\emptyset.$ 
\end{obs}

\begin{obs}\label{P1-obs-unicidade} By \cite[Remark~6.11]{Castro2020},
if $\xi^{\alpha} \in  \Fix(t)$ for some $t \in \mathbb{F}\setminus\{\omega\},$ then there must exist $\beta\in\mathcal{L}^*$ and $\gamma\in\mathcal{L}^{\geq 1}$ such that the associated labelled path of $\xi^{\alpha}$ is $\alpha = \beta\gamma^{\infty}$ and $t$ is either $\beta\gamma\beta^{-1}$ or $\beta\gamma^{-1}\beta^{-1}$.
\end{obs} 

\begin{lema}[{\cite[Lemma~8.9]{Boava2021}}]\label{P1-lemma-8.9}
Let $(t,\xi) \in \Iso(\F\rtimes_\varphi \textsf{T})$ with $t=\beta\gamma\beta^{-1}$ or $t=\beta\gamma^{-1}\beta^{-1}$, where $\beta\gamma^{\infty}$ is the labelled path associated with $\xi$. 
Then, for every neighborhood $U\subseteq \textsf{T}$ of $\xi$, there exists $k>1$  and  $B\in\xi_{|\beta\gamma^k|}$ such that
$ \xi\in V_{(\beta\gamma^k,B,\beta\gamma^k)}\subseteq U. $
\end{lema}

\subsection{Applications to ultragraphs, their associated partial actions and partial skew groupoid rings}

Let $\mathcal{G}$ be an ultragraph, let $(\mathcal{E}_{\mathcal{G}}, \mathcal{L}_{\mathcal{G}}, \mathcal{B})$ be the normal labelled space associated with $\mathcal{G}$, and let $\varphi$ be the topological partial action of the free group $\mathbb{F}$ generated by the labels of the ultragraph on the tight spectrum $\textsf{T}$ as in \cite[Section~4.1]{Castro2021}. We aim to prove that $\varphi$ is topologically free on every closed $\mathbb{F}$-invariant subset of $\textsf{T}$ if, and only if, $\mathcal{G}$ satisfies Condition (K).

In the context of ultragraphs, Condition (K) was first defined in \cite[Definition~2.4]{Katsura2008}. We present an equivalent definition below.

\begin{definition}[{\cite[Section~4]{Castro2021.2}}]
An ultragraph $\mathcal{G}$ satisfies \emph{Condition (K)}, if for every $v \in G^0,$ there is either no simple loop based at $v$ or at least two simple loops based at $v.$
\end{definition}

The following definition is based on \cite[Definition~37.20]{Exel2017}.

\begin{definition}
Let $\gamma$ be a loop in $\mathcal{G}.$ 
\begin{enumerate}[{\rm (a)}]
    \item $\gamma$ is said to be \emph{recurrent}, if there exists another loop $\rho$ in $\mathcal{G}$ such that $s(\rho)=s(\gamma)$ and
    $ \gamma \rho \gamma^{\infty} \neq \gamma^{\infty}.$
    \item $\gamma$ is said to be \emph{transitory}, if it is not recurrent.
\end{enumerate}
\end{definition}

\begin{proposition}\label{P1-prop:KR}
Every loop in 
$\mathcal{G}$ is recurrent if, and only if, $\mathcal{G}$ satisfies Condition (K).
\end{proposition}

\begin{proof}
We first show the ''only if'' statement.
Suppose that every loop in 
$\mathcal{G}$ is recurrent. Let $v \in G^0$ be a vertex and let $\alpha$ be a simple loop based at $v.$ In particular, $s(\alpha)=v.$ By assumption, there is a loop $\beta$ such that $\alpha^{\infty} \neq \alpha \beta \alpha^{\infty}.$ Observe that $\beta=\gamma^{(1)} \ldots \gamma^{(n)}$ is such that $\gamma^{(j)}$ is a simple loop based at $s(\alpha)$, for every $j \in \{1,\ldots,n \}.$ Since $\alpha^{\infty} \neq \alpha \beta \alpha^{\infty},$ there exists $k \in \{1,\ldots,n\}$ such that $\gamma^{(k)}$ is distinct from $\alpha.$ Thus, $\mathcal{G}$ satisfies Condition (K). \smallbreak

Now we show the ''if'' statement.
Suppose that $\mathcal{G}$ satisfies Condition (K). Let $\alpha$ be a loop in $\mathcal{G}$ and notice that $\alpha=\gamma^{(1)}\ldots\gamma^{(n)}$ is such that $\gamma^{(j)}$ is a simple loop based at $s(\alpha),$ for every $j \in \{1,\ldots,n\}.$ If there is $k \in \{1,\ldots,n\}$ such that $\gamma^{(k)} \neq \gamma^{(1)},$ then let $\beta:=\gamma^{(k)}$. Note that, $\alpha\beta\alpha^{\infty}\neq \alpha^{\infty}.$ If $\alpha=\gamma^{(1)}\gamma^{(1)}\ldots\gamma^{(1)},$ then by Condition (K), there exists a simple loop $\delta$ based at $s(\alpha)$ and distinct from $\gamma^{(1)}.$ Therefore, $\alpha\delta\alpha^{\infty}\neq \alpha^{\infty},$ and hence $\alpha$ is recurrent.
\end{proof}

\begin{lema}\label{P1-lemma-T2}
Let  $\alpha \in \mathcal{L}_{\mathcal{G}}^{\geq 1}$ and $\beta \in \mathcal{L}_{\mathcal{G}}^{\infty}$ be such that $\alpha\beta \in \mathcal{L}_{\mathcal{G}}^{\infty}$ and $\xi^{\beta} \in \textsf{T}_{\beta}.$ Then $\xi^{\beta} \in \textsf{T}_{(\alpha)\beta}.$
\end{lema}

\begin{proof}
By the definition of $\textsf{T}_{(\alpha)\beta},$ we only need to prove that $r(\alpha) \in \xi^{\beta}_0.$ Since $\beta \in \mathcal{L}_{\mathcal{G}}^{\infty},$ Lemma~\ref{P1-lema-family-inf} implies that $\xi^{\beta}_0=  \, \uparrow_{\mathcal{B}_{\omega}}\{s(\beta_1)\}=  \, \uparrow_{\mathcal{B}}\{s(\beta_1)\}.$ By assumption, $\alpha\beta \in \mathcal{L}_{\mathcal{G}}^{\infty}.$ Therefore, $s(\beta_1) \in r(\alpha)$ and $r(\alpha) \in \, \uparrow_{\mathcal{B}}\{s(\beta_1)\}.$
\end{proof}

Proposition~\ref{P1-prop:prop2} and Proposition~\ref{P1-prop:prop1} 
below are inspired by techniques from \cite[Proposition~37.21]{Exel2017} and \cite[Chapter~12]{Exel1999}, and will be used to prove the main result of this section.

\begin{proposition}\label{P1-prop:prop2}
Let $t \in \mathbb{F}\setminus\{\omega\}$ and $\xi^{\alpha} \in \Fix(t)$ be such that $\alpha = \beta\gamma^{\infty}$ and that $t$ is either $\beta\gamma\beta^{-1}$ or $\beta\gamma^{-1}\beta^{-1}$ (see Remark~\ref{P1-obs-unicidade}).  Then, $\xi^{\alpha}$ is an isolated point in $\Orb(\xi^{\alpha})$ if, and only if, $\gamma$ is a transitory loop.
\end{proposition}

\begin{proof}
We first show the ``if'' statement.
Suppose that $\gamma$ is a transitory loop and let $U$ be an open set such that $\xi^{\alpha} \in U.$ Since $\xi^{\alpha} \in \Fix(t),$ we have that $(t,\xi^{\alpha}) \in \Iso(\F\rtimes_\varphi\textsf{T}).$ By Lemma~\ref{P1-lemma-8.9}, there exists $k>1$  and  $B\in\xi^{\alpha}_{|\beta\gamma^k|}$ such that
$ \xi^{\alpha}\in V_{(\beta\gamma^k,B,\beta\gamma^k)}\subseteq U. $

We claim that $V_{(\beta\gamma^k,B,\beta\gamma^k)} \cap \Orb(\xi^{\alpha})=\{\xi^{\alpha}\}.$ Let $\eta^{\delta} \in V_{(\beta\gamma^k,B,\beta\gamma^k)} \cap \Orb(\xi^{\alpha}).$
Using that $\eta^{\delta} \in \Orb(\xi^{\alpha}),$ there exists $s \in \mathbb{F}$ such that $\xi^{\alpha} \in V_{s^{-1}}$ and $\eta^{\delta}= \varphi_s(\xi^{\alpha}).$ Since $V_{s^{-1}}\neq \emptyset,$ Remark~\ref{P1-lemma-orb} implies that $s \in \{\omega\}\cup\{\alpha':\alpha' \in \mathcal{L}_{\mathcal{G}}^{\geq 1}\}\cup\{\alpha'^{-1}:\alpha' \in \mathcal{L}_{\mathcal{G}}^{\geq 1}\}\cup\{\alpha'\beta'^{-1}:\alpha', \beta' \in \mathcal{L}_{\mathcal{G}}^{\geq 1}, r(\alpha')\cap r(\beta') \neq \emptyset\}.$ Since $\alpha = \beta\gamma^{\infty}$ and by the definition of the partial action $\varphi,$ the labelled path $\delta$ associated to $\eta^\delta$ is eventually periodic with period component $\gamma.$

By assumption $\eta^{\delta} \in V_{(\beta\gamma^k,B,\beta\gamma^k)},$ and we have that $(\beta\gamma^k,B,\beta\gamma^k) \in \eta^{\delta}.$  By Remark~\ref{P1-obs-2.6}, $\beta\gamma^k$ is a beginning of $\delta.$
Using that $\delta$ eventually returns to $\gamma$, and that $\gamma$ is not recurrent, we have that
$\delta = \beta \gamma^{\infty}=\alpha.$
Therefore, $\eta$ is associated with the word $\alpha$, and hence, by Remark~\ref{P1-obs-unique-filter}, $\eta^{\alpha} = \xi^{\alpha}.$\smallbreak

Now we show the ``only if'' statement.
Suppose that $\xi^{\alpha}$ is an isolated point in $\Orb(\xi^{\alpha}).$ 
There exists an open set $U \subseteq \textsf{T}$ such that 
$U \cap \Orb(\xi^{\alpha})=\{\xi^{\alpha}\}.$
By Lemma~\ref{P1-lemma-8.9}, there exist $k>1$  and  $B\in\xi^{\alpha}_{|\beta\gamma^k|}$ such that
$ \xi^{\alpha}\in V_{(\beta\gamma^k,B,\beta\gamma^k)}\subseteq U. $
Hence,
$V_{(\beta\gamma^k,B,\beta\gamma^k)} \cap \Orb(\xi^{\alpha})= \{\xi^{\alpha}\}.$

Seeking a contradiction, suppose that $\gamma$ is a recurrent loop. Then, there exists a loop $\rho$ such that $s(\rho)=s(\gamma)$ and $\gamma \rho \gamma^{\infty} \neq \gamma^{\infty}.$ 
Define
$\delta := \beta \gamma^k \rho \gamma^{\infty}.$
By 
Remark~\ref{P1-obs-unique-filter}, there is a unique element $\eta^{\delta}$ associated with $\delta.$ Notice that, since $\beta\gamma^k$ is a beginning of $\delta$ and $s(\rho)=s(\gamma),$ by Lemma~\ref{P1-lema-family-inf}, the elements of the complete family of ultrafilters associated with $\alpha$ and $\delta$ satisfy
$\xi^{\alpha}_n=\eta^{\delta}_n,$ for all $n \in \{1,\ldots, |\beta \gamma^k|\}.$ In particular, since $B \in \xi^{\alpha}_{|\beta \gamma^k|} =\eta^{\delta}_{|\beta \gamma^k|},$ by Remark~\ref{P1-obs-2.6}, 
$(\beta \gamma^k,B,\beta \gamma^k) \in \eta^{\delta}.$ Hence, $\eta^{\delta} \in V_{(\beta \gamma^k,B,\beta \gamma^k)}.$

We claim that $\eta^{\delta} \in \Orb(\xi^{\alpha}).$ Let $s= \beta \gamma^k \rho \beta^{-1}.$ Observe that $\xi^{\alpha} \in V_{s^{-1}},$ because $V_{s^{-1}}= \varphi_{\beta^{-1}}^{-1}(V_{(\beta \gamma^k \rho)^{-1}})$ and $\varphi_{\beta^{-1}}(\xi^{\alpha})=\mu^{\gamma^{\infty}} \in \textsf{T}_{\gamma^{\infty}}.$ Since $s(\gamma)=s(\rho) \in r(\rho),$ then $\beta \gamma^k \rho \gamma^{\infty} \in \mathcal{L}_{\mathcal{G}}^{\infty}$ and Lemma~\ref{P1-lemma-T2} implies that $\mu^{\gamma^{\infty}} \in \textsf{T}_{(\beta \gamma^k \rho)\gamma^{\infty}}$. Thus, by Lemma~\ref{P1-lemma-T1}, 
$\mu^{\gamma^{\infty}} \in \textsf{T}_{(\beta \gamma^k \rho)\gamma^{\infty}} \subseteq V_{(\beta \gamma^k \rho)^{-1}} $ and $\xi^{\alpha} = \varphi_{\beta^{-1}}^{-1}(\mu^{\gamma^{\infty}}) \in V_{s^{-1}}.$ By Remark~\ref{P1-obs-unique-filter}, there is a unique element associated with the word $\delta.$ Hence, $\varphi_s(\xi^{\alpha})= \eta^{\delta}$ and  we conclude that $\eta^{\delta} \in V_{(\beta \gamma^k,B,\beta \gamma^k)} \cap \Orb(\xi^{\alpha}).$ This is a contradiction.
\end{proof} 

\begin{corolario}\label{P1-cor:corprop2}
Let $t \in \mathbb{F}\setminus\{\omega\}$ and $\xi^{\alpha} \in \Fix(t)$ be such that $\alpha = \beta\gamma^{\infty}$ and that $t$ is either $\beta\gamma\beta^{-1}$ or $\beta\gamma^{-1}\beta^{-1}$ (see Remark~\ref{P1-obs-unicidade}).  Then, $\xi^{\alpha}$ is not an isolated point in $\Orb(\xi^{\alpha})$ if, and only if, $\gamma$ is a recurrent loop.
\end{corolario}

\begin{proposition}\label{P1-prop:prop1}
The partial action $\varphi$ of $\mathbb{F}$ on $\textsf{T}$ is topologically free on every closed $\mathbb{F}$-invariant subset of $\textsf{T}$ if, and only if, for every $t \in \mathbb{F} \setminus \{\omega\}$ and $\xi \in \Fix(t),$  $\xi$ is not an isolated point in $\Orb(\xi).$
\end{proposition}

\begin{proof}
We first show the ``if'' statement by contrapositivity.
Suppose that $\varphi$ is not topologically free on every closed $\mathbb{F}$-invariant subset of $\textsf{T}.$ Then, there exist a closed $\mathbb{F}$-invariant subset $C \subseteq \textsf{T}$ and an element $t \in \mathbb{F} \setminus \{\omega\}$  such that $\Int_C(\Fix(t) \cap C) \neq \emptyset.$

Let $\xi^{\alpha} \in \Int_C(\Fix(t) \cap C).$ Since $\xi^{\alpha} \in \Fix(t),$ by Remark~\ref{P1-obs-unicidade}, there exist $\beta\in\mathcal{L}^*$ and $\gamma\in\mathcal{L}^{\geq 1}$, such that the associated labelled path of $\xi^{\alpha}$ is $\alpha = \beta\gamma^{\infty}$, and $t$ is either $\beta\gamma\beta^{-1}$ or $\beta\gamma^{-1}\beta^{-1}$.
By Remark~\ref{P1-obs-unique-filter}, since $\alpha$ is an infinite path in $\mathcal{G},$ $\xi^{\alpha}$ is the unique element in $\textsf{T}$ whose associated word is $\alpha$. Therefore,
$\Int_C(\Fix(t)\cap C) =\{\xi^{\alpha}\}.$
Notice that
$$\xi^{\alpha} \in \Orb(\xi^{\alpha})= \bigcup_{t \in \mathbb{F}} \varphi_t(\{\xi^{\alpha}\}\cap V_{t^{-1}}) \subseteq \bigcup_{t \in \mathbb{F}} \varphi_t(C \cap V_{t^{-1}}) 
\subseteq C. $$
Thus, $\xi^{\alpha}$ is an isolated point in $\Orb(\xi^{\alpha}).$\smallbreak

Now we show the ``only if'' statement by contrapositivity.
Suppose that there exist 
$t \in \mathbb{F} \setminus \{\omega\}$
 and $\xi \in \Fix(t),$ such that $\xi$ is an isolated point in $\Orb(\xi).$ Then, there exist an open subset $V$ of $\textsf{T}$ such that
$\Orb(\xi) \cap V = \{\xi\}.$
We claim that
$\overline{\Orb(\xi)} \cap V=\{\xi\}.$ Seeking a contradiction, suppose that there is $\eta \in \overline{\Orb(\xi)} \cap V$ such that $\eta \neq \xi.$ Then
$\eta \in V \setminus \{\xi\},$ 
the latter being an open set. But,
$V \setminus \{\xi\} \cap \Orb(\xi)=\emptyset$ and this is a contradiction, since $\eta \in \overline{\Orb(\xi)}.$
Therefore, $C=\overline{\Orb(\xi)}$ is a closed $\mathbb{F}$-invariant subset of $\textsf{T}$ such that $\xi \in \Int_C(\Fix(t) \cap C). $ Hence, $\varphi$ is not topologically free on $C.$
\end{proof}

\begin{proposition}\label{P1-prop:tloop}
The partial action $\varphi$ of $\mathbb{F}$ on $\textsf{T}$ is topologically free on every closed $\mathbb{F}$-invariant subset of $\textsf{T}$ if, and only if, every loop in 
$\mathcal{G}$ is recurrent.
\end{proposition}

\begin{proof}
We first show the ``if'' statement by contrapositivity.
Suppose that there is a closed $\mathbb{F}$-invariant subset on which $\varphi$ is not topologically free. By Proposition~\ref{P1-prop:prop1}, there exists $t \in \mathbb{F} \setminus \{\omega\}$ such that $\xi^{\alpha} \in \Fix(t)$ is a isolated point in $\Orb(\xi^{\alpha}).$ By Remark~\ref{P1-obs-unicidade}, there exist $\beta\in\mathcal{L}^*$, $\gamma\in\mathcal{L}^{\geq 1}$ such that the associated labelled path of $\xi^{\alpha}$ is $\alpha = \beta\gamma^{\infty}$ and $t$ is either $\beta\gamma\beta^{-1}$ or $\beta\gamma^{-1}\beta^{-1}$. By Proposition~\ref{P1-prop:prop2}, $\gamma$ is a transitory loop.\smallbreak

Now we show the ``only if'' statement by contrapositivity.
Suppose that $\mathcal{G}$ has a transitory loop $\gamma.$ Set $\alpha:= \gamma^{\infty}$ and notice that $\xi^{\alpha} \in \Fix(\gamma).$ By Proposition~\ref{P1-prop:prop2}, $\xi^{\alpha}$ is an isolated point in $\Orb(\xi^{\alpha}),$ and, by Proposition~\ref{P1-prop:prop1}, 
$\varphi$ is not topologically free on every closed $\mathbb{F}$-invariant subset of $\textsf{T}.$
\end{proof}

Let $\mathbb{K}$ be a field and let $\alpha$ be the partial action of $\mathbb{F}$ on $\mathcal{L}_c(\textsf{T}, \mathbb{K})$ induced by the topological partial action $\varphi.$
The natural $\mathbb{F}$-grading on the partial skew group ring $\mathcal{L}_c(\textsf{T},\mathbb{K}) \rtimes_{\alpha}\mathbb{F}$ is  given by
$\mathcal{L}_c(\textsf{T},\mathbb{K}) \rtimes_{\alpha}\mathbb{F} = \bigoplus_{t\in \mathbb{F}} \mathcal{L}_c(V_t,\mathbb{K})\delta_t$. Using the map $\alpha\beta^{-1} \mapsto |\alpha|-|\beta|$, for $\alpha\beta^{-1}\in \mathbb{F}$ in reduced from, and the fact that the partial action $\varphi$ is semi-saturated \cite[Proposition 3.12]{Castro2020}, it follows that $\mathcal{L}_c(\textsf{T},R) \rtimes_{\alpha}\mathbb{F}$ has a $\Z$-grading with homogeneous component of degree $n\in \Z$ given by 
$D_n:=\Span_{\mathbb{K}}\{f_{\alpha\beta^{-1}}\delta_{\alpha\beta^{-1}} : \alpha,\beta\in\mathcal{L}^* \ \mbox{and} \ |\alpha|-|\beta|=n \}. $ See \cite[Section~4]{Boava2021}, for more details.

\begin{theorem}\label{P1-thm:ultgraded}
Let $\mathbb{K}$ be a field and let $\mathcal{G}$ be an ultragraph. 
Then, every ideal of $\mathcal{L}_c(\textsf{T},\mathbb{K}) \rtimes_{\alpha}\mathbb{F}$ is $\mathbb{F}$-graded if, and only if, every ideal of $\mathcal{L}_c(\textsf{T},\mathbb{K}) \rtimes_{\alpha}\mathbb{F}$ is $\Z$-graded.
\end{theorem}

\begin{proof}
By \cite[Theorem~5.6]{Boava2021}, there is a $\Z$-graded isomorphism between $\mathcal{L}_c(\textsf{T},\mathbb{K}) \rtimes_{\alpha}\mathbb{F}$ and the Leavitt path algebra, $L_{\mathbb{K}}(\mathcal{E_{\mathcal{G}}},\mathcal{L}_{\mathcal{G}},\mathcal{B}),$ of the labelled space  $(\mathcal{E_{\mathcal{G}}},\mathcal{L}_{\mathcal{G}},\mathcal{B}).$ Moreover, $L_{\mathbb{K}}(\mathcal{E_{\mathcal{G}}},\mathcal{L}_{\mathcal{G}},\mathcal{B})$ is isomorphic to the ultragraph Leavitt path algebra $L_{\mathbb{K}}(\mathcal{G})$, see \cite[Example~7.2]{Boava2021}. Finally, by \cite[Theorem~4.3]{Imanfar2020}, since $\mathbb{K}$ is a field, every ideal of $L_{\mathbb{K}}(\mathcal{G})$ is $\Z$-graded if, and only if, the ultragraph $\mathcal{G}$ satisfies Condition$(K).$ Thus, every ideal of $\mathcal{L}_c(\textsf{T},\mathbb{K}) \rtimes_{\alpha}\mathbb{F}$ is $\Z$-graded if, and only if, $\mathcal{G}$ satisfies Condition (K). Proposition~\ref{P1-prop:KR} and Proposition~\ref{P1-prop:tloop} imply that $\mathcal{G}$ satisfies Condition (K) if, and only if, $\varphi$ is topologically free on every closed $\mathbb{F}$-invariant subset of $\textsf{T}.$ By Theorem~\ref{P1-thm:TE1}, this is equivalent to every ideal of $\mathcal{L}_c(\textsf{T}, \mathbb{K}) \rtimes_{\alpha}\mathbb{F}$ being $\mathbb{F}$-graded.
\end{proof}

Now we state the main result of this section, partially generalizing \cite[Proposition~9.1]{Boava2021}.

\begin{theorem}\label{P1-thm:TFinal}
Let $\mathbb{K}$ be a field, let $\mathcal{G}$ be an ultragraph, let $(\mathcal{E}_{\mathcal{G}}, \mathcal{L}_{\mathcal{G}}, \mathcal{B})$ be the normal labelled space associated with $\mathcal{G},$ 
let $\varphi$ be the topological partial action of the free group $\mathbb{F}$ generated by the labels of the ultragraph on the tight spectrum $\textsf{T}$, 
and let $\alpha$ be the partial action of $\mathbb{F}$ on $\mathcal{L}_c(\textsf{T}, \mathbb{K})$ induced by the topological partial action $\varphi.$ The following statements are equivalent:
\begin{enumerate}[{\rm (i)}]
    \item $\overline{D}_0\delta_0$ is a maximal commutative subring of $\frac{\mathcal{L}_c(\textsf{T}, \mathbb{K})}{\mathcal{J}(F)} \rtimes_{\overline{\alpha}}\mathbb{F},$ for every closed $\mathbb{F}$-invariant subset $F$ of $\textsf{T};$ 
    \item $\alpha$ has the residual intersection property;
    \item Every ideal of $\mathcal{L}_c(\textsf{T}, \mathbb{K}) \rtimes_{\alpha}\mathbb{F}$ is $\mathbb{F}$-graded;
    \item Every ideal of $\mathcal{L}_c(\textsf{T},\mathbb{K}) \rtimes_{\alpha}\mathbb{F}$ is $\Z$-graded;
    \item $\varphi$ is topologically free on every closed $\mathbb{F}$-invariant subset of $\textsf{T};$ 
    \item $\varphi$ is topologically free on every open $\mathbb{F}$-invariant subset of $\textsf{T};$ 
    \item $\mathbb{F} \rtimes_{\varphi}\textsf{T}$ is strongly effective;
    \item Every loop in
    $\mathcal{G}$ is recurrent;
    \item $\mathcal{G}$ satisfies Condition (K).
\end{enumerate}
\end{theorem}

\begin{proof}
The equivalences between (i), (ii), (iii), (v), (vi), and (vii) follow from Theorem~\ref{P1-thm:TE1}. The other equivalences follow from Proposition~\ref{P1-prop:KR}, Proposition~\ref{P1-prop:tloop} and Theorem~\ref{P1-thm:ultgraded}.
\end{proof}

\begin{obs}
We point out that the equivalences between 
(iv), (vii), and (viii) in the above theorem are already known in the context of Leavitt path algebras of ultragraphs (see \cite[Theorem~4.3]{Imanfar2020}), but the other equivalences are new even in the context of graphs. \end{obs}

%\section*{Acknowledgement}
%The third named author was financed in part by the
%Coordenação de Aperfeiçoamento de Pessoal de Nível Superior - Brasil (CAPES) - Finance Code 001.


\addcontentsline{toc}{section}{References}
\begin{thebibliography}{99}

\bibitem{Abadie2004}
Abadie, F., On partial actions and groupoids. \newblock {\em Proc. Amer. Math. Soc.} \textbf{132} (2004), no. 4, 1037--1047.

\bibitem{Bagio2010}
Bagio, D., Flores, D. and Paques, A., Partial actions of ordered groupoids on rings. 
\newblock {\em J. Algebra Appl.} \textbf{9} (2010), no. 3, 501--517.

\bibitem{Bagio2012}
Bagio, D. and Paques, A., Partial groupoid actions: globalization, Morita theory, and Galois theory. \newblock {\em Comm. Algebra} \textbf{40} (2012), no. 10, 3658--3678.

\bibitem{TBeuter2018}
Beuter, V., {\em Partial actions of inverse semigroups and their associated algebras.}
\newblock Ph.D. thesis, Universidade Federal de Santa Catarina (2018).

\bibitem{Beuter2018}
Beuter, V. and Gonçalves, D., The interplay between Steinberg algebras and skew rings. 
\newblock {\em J. Algebra} \textbf{497} (2018), 337--362.

\bibitem{BeuterRoyer}
Beuter, V., Gonçalves, D., Öinert, J. and Royer, D.,
Simplicity of skew inverse semigroup rings with applications to Steinberg algebras and topological dynamics.
\newblock {\em Forum Math.} \textbf{31} (2019), no. 3, 543--562.

\bibitem{Boava2021}
Boava, G., de Castro, G., Gonçalves, D. and van Wyk, D., Leavitt path algebras of labelled graphs.
\newblock {\em arXiv preprint arXiv:2106.06036}  (2021).

\bibitem{Boava2017}
Boava, G., de Castro, G. and Mortari, F., Inverse semigroups associated with labelled spaces and their tight spectra. \newblock {\em Semigroup Forum} \textbf{94} (2017), no. 3, 582--609.


\bibitem{Brown2014}
Brown, J., Clark, L., Farthing, C. and Sims, A., Simplicity of algebras associated to étale groupoids. \newblock {\em Semigroup Forum} \textbf{88} (2014), no. 2, 433--452.

\bibitem{Buss2012}
Buss, A. and Exel, R., Inverse semigroup expansions and their actions on $C^{\ast}$-algebras. \newblock {\em Illinois J. Math.} \textbf{56} (2012), no. 4, 1185--1212.

\bibitem{Carlsen2022}
Carlsen, T. M. and Kang, E. J.,
Condition (K) for Boolean dynamical systems.
\newblock {\em J. Aust. Math. Soc.} \textbf{112} (2022), no. 2, 145--169.

\bibitem{Castro2021.2}
de Castro, G., Gonçalves, D. and van Wyk, D.,
Topological full groups of ultragraph groupoids as an isomorphism invariant. 
\newblock{\em Münster J. Math.} \textbf{14} (2021), no. 1, 165--189.

\bibitem{Castro2021}
de Castro, G., Gonçalves, D. and van Wyk, D. W., Ultragraph algebras via labelled graph groupoids, with applications to generalized uniqueness theorems.  
\newblock{\em J. Algebra} \textbf{579} (2021), 456--495.


\bibitem{Castro2020}
de Castro, G. and van Wyk, D., Labelled space $C^{\ast}$-algebras as partial crossed products and a simplicity characterization. 
\newblock{\em  J. Math. Anal. Appl.} \textbf{491} (2020), no. 1, 124290.


\bibitem{Kang2022}
de Castro, G. G. and Kang, E. J., $C^*$-algebras of generalized Boolean dynamical systems as partial crossed products. \newblock {\em arXiv preprint arXiv:2202.02008} (2022).

\bibitem{Clark2019}
Clark, L., Edie-Michell, C., Huef, A. and Sims, A., Ideals of Steinberg algebras of strongly effective groupoids, with applications to Leavitt path algebras. \newblock {\em Trans. Amer. Math. Soc.} \textbf{371} (2019), no. 8, 5461--5486.

\bibitem{Connell1963}
Connell, I., On the group ring. \newblock {\em Canadian J. Math.} \textbf{15} (1963), 650--685.


\bibitem{Cordeiro2021}
Cordeiro, L., Gonçalves, D. and Hazrat, R., The talented monoid of a directed graph with applications to graph algebras. \newblock {\em Rev. Mat. Iberoam.} \textbf{38} (2022), no. 1, 223--256.


\bibitem{Dokuchaev2019}
Dokuchaev, M., Recent developments around partial actions. \newblock {\em São Paulo J. Math. Sci.} \textbf{13} (2019), no. 1, 195--247.

\bibitem{Dokuchaev2005}
Dokuchaev, M. and Exel, R., Associativity of crossed products by partial actions, enveloping actions and partial representations. \newblock {\em  Trans. Amer. Math. Soc.} \textbf{357} (2005), no. 5, 1931--1952.


\bibitem{Dokuchaev2015}
Dokuchaev, M. and Khrypchenko, M., Partial cohomology of groups.
\newblock {\em J. Algebra} \textbf{427} (2015), 142--182.

\bibitem{Duyen2021}
Duyen, T., Gonçalves, D. and Nam, T., 
On the ideals of ultragraph Leavitt path algebras. 
\newblock {\em arXiv preprint arXiv:2109.10440} (2021).


\bibitem{Exel2008}
Exel, R., Inverse semigroups and combinatorial {$C^{\ast}$}-algebras. \newblock {\em Bull. Braz. Math. Soc. (N.S.)} \textbf{39} (2008), no. 2, 191--313.


\bibitem{Exel2017}
Exel, R., \emph{ Partial dynamical systems, Fell bundles and applications}.
Mathematical Surveys and Monographs, 224,
American Mathematical Society (2017).


\bibitem{Exel1999}
Exel, R. and Laca, M.,
Cuntz-Krieger algebras for infinite matrices.
\newblock {\em J. Reine Angew. Math.} \textbf{512} (1999), 119--172.

\bibitem{Flores2020}
Flores, F. and Mantoiu, M.,
Topological Dynamics of Groupoid Actions. 
\newblock {\em arXiv preprint arXiv:2011.08882} (2020).

\bibitem{Gilbert2005}
Gilbert, N., Actions and expansions of ordered groupoids. 
\newblock{\em J. Pure Appl. Algebra} \textbf{198} (2005), no. 1-3, 175--195.


\bibitem{Thierry2014}
Giordano, T. and Sierakowski, A.,
Purely infinite partial crossed products.
\newblock {\em  J. Funct. Anal.} \textbf{266} (2014), no. 9, 5733--5764.

\bibitem{Goncalves2014}
Gonçalves, D., \"Oinert, J. and Royer, D.,
Simplicity of partial skew group rings with applications to Leavitt path algebras and topological dynamics.
\newblock{\em J. Algebra} \textbf{420} (2014), 201--216.

\bibitem{Goncalves2014.1}
Gonçalves, D. and Royer, D.,
Leavitt path algebras as partial skew group rings. 
\newblock {\em Comm. Algebra} \textbf{42} (2014), no. 8, 3578--3592.

\bibitem{Goncalves2017}
Gonçalves, D. and Royer, D.,
Ultragraphs and shift spaces over infinite alphabets.
\newblock{\em Bull. Sci. Math.} \textbf{141} (2017), no. 1, 25--45.

\bibitem{Goncalves2020}
Gonçalves, D. and Royer, D., 
Simplicity and chain conditions for ultragraph Leavitt path algebras via partial skew group ring theory. 
\newblock {\em J. Aust. Math. Soc.} \textbf{109} (2020), no. 3, 299--319.

\bibitem{Goncalves2016}
Gonçalves, D. and Yoneda, G., Free path groupoid grading on Leavitt path algebras. 
\newblock {\em Internat. J. Algebra Comput.} \textbf{26} (2016), no. 6, 1217--1235.


\bibitem{Hazrat2013.1}
Hazrat, R.,
The dynamics of Leavitt path algebras.
\newblock {\em J. Algebra} \textbf{384} (2013), 242--266.

\bibitem{Hazrat2013}
Hazrat, R., The graded Grothendieck group and the classification of Leavitt path algebras. \newblock {\em Math. Ann.} \textbf{355} (2013), no. 1, 273--325.

\bibitem{Imanfar2020}
Imanfar, M., Pourabbas, A. and Larki, H.,
The Leavitt path algebras of ultragraphs.
\newblock {\em Kyungpook Math. J.} \textbf{60} (2020), no. 1, 21--43.

\bibitem{Katsura2008}
Katsura, T., Muhly P., Sims, A. and Tomforde, M.,
Ultragraph {$C^\ast $}-algebras via topological quivers.
\newblock {\em Studia Math.} \textbf{187} (2008), no. 2, 137--155.

\bibitem{Keimel1970}
Keimel, K.,
Alg\`ebres commutatives engendr\'ees par leurs \'el\'ements idempotents. 
\newblock {\em Canadian J. Math.} \textbf{22} (1970), no. 5, 1071--1078.

\bibitem{Lannstrom2021}
Lännström, D., Lundström, P., Öinert, J. and Wagner, S.,
Prime group graded rings with applications to partial crossed products and Leavitt path algebras. \newblock {\em arXiv preprint arXiv:2105.09224} (2021).

%book
\bibitem{Lawson1998} 
Lawson, M.,
\emph{Inverse Semigroups. The Theory of Partial Symmetries.}
\newblock 
World Scientific Publishing Co., River Edge, NJ (1998).

\bibitem{Lawson2012}
Lawson, M.,
Non-commutative Stone duality: inverse semigroups, topological groupoids and {$C^\ast$}-algebras.
\newblock {\em Internat. J. Algebra Comput.} \textbf{22} (2012), no. 6, 1250058, 47 pp.


\bibitem{Lundstrom2012}
Lundstr\"om, P. and \"Oinert, J., Skew category algebras associated with partially defined dynamical systems. \newblock {\em Internat. J. Math.} \textbf{23} (2012), no. 4, 1250040, 16 pp.

\bibitem{NystedtSurvey2019}
Nystedt, P., A survey of $s$-unital and locally unital rings. 
\newblock {\em Rev. Integr. Temas Mat.} \textbf{37} (2019), no. 2, 251--260.

\bibitem{NystedtSystems2020}
Nystedt, P., Simplicity of algebras via epsilon-strong systems. 
\newblock {\rm Colloq. Math.} \textbf{162} (2020), no. 2, 279--301.

\bibitem{Oinert2013}
Nystedt, P. and \"Oinert, J., Simple skew category algebras associated with minimal partially defined dynamical systems. 
\newblock {\em Discrete Contin. Dyn. Syst.} \textbf{33} (2013), no. 9, 4157--4171.

\bibitem{Oinert2012}
\"Oinert, J. and Lundstr\"om, P., The ideal intersection property for groupoid graded rings. 
\newblock {\em Comm. Algebra} \textbf{40} (2012), no. 5, 1860--1871.

%book
\bibitem{Sims2020}
Sims, A., Szab{\'o}, G.
and Williams, D., \newblock {\em 
Operator algebras and dynamics: groupoids, crossed products, and Rokhlin dimension.} 
Adv. Courses in Math., CRM Barcelona, 
Birkh\"{a}user/Springer
%Springer International Publishing 
(2020). %89–102.


\bibitem{Steinberg2016}
Steinberg, B.,
Simplicity, primitivity and semiprimitivity of étale groupoid algebras with applications to inverse semigroup algebras. 
\newblock {\em J. Pure Appl. Algebra} \textbf{220} (2016), no. 3, 1035--1054.

\bibitem{Steinberg2019}
Steinberg, B.,
Prime étale groupoid algebras with applications to inverse semigroup and Leavitt path algebras.
\newblock {\em J. Pure Appl. Algebra} \textbf{223} (2019), no. 6, 2474--2488.

\bibitem{Tomforde2003}
Tomforde, M., A unified approach to Exel-Laca algebras and {$C^\ast$}-algebras associated to graphs. 
\newblock {\em J. Operator Theory} \textbf{50} (2003), no. 2, 345--368.

\bibitem{Vas2022.1}
Vas, L.,
Every graded ideal of a Leavitt path algebra is graded isomorphic to a Leavitt path algebra. 
\newblock {\em Bull. Aust. Math. Soc.} \textbf{105} (2022), no. 2, 248--256.

\bibitem{Vas2022.2}
Vas, L., Graded irreducible representations of Leavitt path algebras: A new type and complete classification.
\newblock {\em J. Pure Appl. Algebra} \textbf{227} (2023), no. 3, 107213, 18 pp.


\end{thebibliography}
\end{document}